\documentclass[12pt]{article}

\usepackage[left=2.7cm,bottom=2.7cm,right=2.7cm,top=2.7cm]{geometry}

\usepackage{enumerate}
\usepackage{multirow}
\usepackage{tabularx}
\usepackage{indentfirst}
\usepackage{setspace}
\usepackage{color}
\usepackage{url}            
\usepackage{booktabs}       
\usepackage{makecell}
\usepackage{comment}
\usepackage{amsfonts}       
\usepackage{amsmath}
\usepackage{amsfonts}
\usepackage{mathrsfs}
\usepackage{algorithm}
\usepackage{algorithmic}
\usepackage{mathtools}
\usepackage{tikz}
\usetikzlibrary{patterns}
\usetikzlibrary{decorations.pathreplacing,calligraphy}
\usetikzlibrary {patterns.meta}
\tikzdeclarepattern{
  name=mylines,
  parameters={
      \pgfkeysvalueof{/pgf/pattern keys/size},
      \pgfkeysvalueof{/pgf/pattern keys/angle},
      \pgfkeysvalueof{/pgf/pattern keys/line width},
  },
  bounding box={
    (0,-0.5*\pgfkeysvalueof{/pgf/pattern keys/line width}) and
    (\pgfkeysvalueof{/pgf/pattern keys/size},
0.5*\pgfkeysvalueof{/pgf/pattern keys/line width})},
  tile size={(\pgfkeysvalueof{/pgf/pattern keys/size},
\pgfkeysvalueof{/pgf/pattern keys/size})},
  tile transformation={rotate=\pgfkeysvalueof{/pgf/pattern keys/angle}},
  defaults={
    size/.initial=5pt,
    angle/.initial=45,
    line width/.initial=.4pt,
  },
  code={
      \draw [line width=\pgfkeysvalueof{/pgf/pattern keys/line width}]
        (0,0) -- (\pgfkeysvalueof{/pgf/pattern keys/size},0);
  },
}

\usepackage{amsmath,amsfonts}
\usepackage{array}
\usepackage[caption=false,font=normalsize,labelfont=sf,textfont=sf]{subfig}
\usepackage{textcomp}
\usepackage{stfloats}
\usepackage{url}
\usepackage{verbatim}
\usepackage{graphicx}
\usepackage{mathtools}
\usepackage{algorithm}
\usepackage{algorithmic}
\usepackage{cases}
\usepackage{enumerate}
\usepackage{tabularx}
\usepackage{multirow}
\usepackage{algorithm}
\usepackage{authblk}
\usepackage{algorithmic}
\usepackage{indentfirst}
\usepackage{setspace}
\usepackage{color}
\usepackage[utf8]{inputenc} 
\usepackage[T1]{fontenc}    
\usepackage{hyperref}       
\usepackage{url}            
\usepackage{booktabs}       
\usepackage{amsfonts}       
\usepackage{nicefrac}       
\usepackage{microtype}      
\usepackage{lipsum}		
\usepackage{graphicx}
\usepackage{natbib}
\usepackage{doi}
\usepackage{amssymb}                  
\usepackage{mathrsfs} 
\usepackage[resetlabels]{multibib}
\usepackage[bottom]{footmisc}
\newtheorem{definition}{Definition}
\newtheorem{lemma}{Lemma}
\newtheorem{theorem}{Theorem}
\newtheorem{corollary}{Corollary}
\newtheorem{proof}{Proof}
\newtheorem{remark}{Remark}
\newtheorem{assumption}{Assumption}

\usepackage[utf8]{inputenc}
\usepackage{tikz}

\def\D{\Delta}
\def\A{\hat{A}(\bar{s}_0)}
\def\s{\hat{s}(\bar{s}_0)}
\begin{document}

\title{A minimax optimal approach to high-dimensional double sparse linear regression}

\author[1]{Yanhang Zhang}
\author[2]{Zhifan Li}
\author[1]{Shixiang Liu}
\author[1, 3]{Jianxin Yin}

\affil[1]{\footnotesize School of Statistics, Renmin University of China}
\affil[2]{\footnotesize Beijing Institute of Mathematical Sciences and Applications}
\affil[3]{\footnotesize Center for Applied Statistics and School of Statistics, Renmin University of China}

\date{}
\maketitle \sloppy

\begin{abstract}
  In this paper, we focus our attention on the high-dimensional double sparse linear regression, that is, a combination of element-wise and group-wise sparsity. To address this problem, we propose an IHT-style (iterative hard thresholding) procedure that dynamically updates the threshold at each step. We establish the matching upper and lower bounds for parameter estimation, showing the optimality of our proposal in the minimax sense. More importantly, we introduce a fully adaptive optimal procedure designed to address unknown sparsity and noise levels. Our adaptive procedure demonstrates optimal statistical accuracy with fast convergence. Additionally, we elucidate the significance of the element-wise sparsity level $s_0$ as the trade-off between IHT and group IHT, underscoring the superior performance of our method over both. Leveraging the beta-min condition, we establish that our IHT-style procedure can attain the oracle estimation rate and achieve almost full recovery of the true support set at both the element level and group level.
   Finally, we demonstrate the superiority of our method by comparing it with several state-of-the-art algorithms on both synthetic and real-world datasets.
\end{abstract}


\section{Introduction}
Over the last decade, the rapid growth of high-dimensional data has drawn broad attention to sparse learning across many scientific communities,
with plenty of remarkable achievements in algorithms, theory, and applications.
One of the well-studied problems is the sparsity-constrained linear regression, also known as the best subset selection.
We consider a linear model
$$
y = X\beta^* + \xi,
$$
where $y\in \mathbb{R}^n$ is the response vector, $X \in \mathbb{R}^{n\times p}$ is the design matrix, $\beta^* \in \mathbb{R}^p$ is the underlying regression coefficient and $\xi \in \mathbb{R}^n$ is the sub-Gaussian random error with scale parameter $\sigma^2$. 
In the high-dimensional framework, we focus on the case where $p \gg n$ and the coefficient $\beta^*$ is sparse in the sense that only a few covariates are important to the model.
Traditionally, element-wise $\ell_0$ sparse problem considers the parameter space 
\begin{equation*}\label{sparsity1}
  \beta^* \in \{\beta \in \mathbb{R}^p:\sum_{i=1}^p\mathrm{I}(\beta_i\ne 0 )\le s'\},
\end{equation*}
where $\beta_i$ is the $i$th entry of $\beta$ and $s'$ is some positive integer, which controls the sparsity level of the sparsity-constrained linear regression problem.
Best subset selection is a famous NP-hard problem \citep{natarajan1995sparse},
and it has been widely studied in the fields of statistics and machine learning \citep{bertsimas2016, zhang2018,huang2018constructive, Zhu202014241}.

Recently, an increasing number of studies on high-dimensional variable selection have focused on the concept of structured sparsity. These studies assume that important variables form specific structures or patterns, with group-wise sparsity being one of the most prominent examples.
The group-wise $\ell_0$ sparsity considers the parameter space
\begin{equation*}
\beta^* \in \{\beta \in \mathbb{R}^p:\sum_{j=1}^m \mathrm{I}(\beta_{G_j} \neq 0) \le s\},
\end{equation*}
where $\{G_j\}_{j=1}^m$ are the indices of $m$ non-overlapping groups such that
$\cup_{j=1}^mG_j=\{1,\ldots,p\}$.  
Here positive integer $s$ controls the number of nonzero groups in the model.
The group sparsity means that within a group, the coefficients are either all zeros or at least one nonzero. In particular, when $|G_1| = \ldots = |G_m| = 1$, the group selection problem boils down to the standard best subset selection.
To date, a variety of practical algorithms have been explored and investigated to conduct group $\ell_0$ selection \citep{BOMP, huang11b, hazimeh2021grouped, zhang2022}.

When considering each group that has been selected, it is generally accepted that only a few of the variables that make up the group are actually significant.
We refer to this idea as double sparsity and define it as follows:
\begin{definition}[Double sparsity]
The regression coefficient $\beta^* \in \mathbb{R}^{p}$ 
is called $(s, s_0)$-sparse if
\begin{equation}\label{def1}
\|\beta^*\|_{0, 2} \coloneqq \sum_{j=1}^m \mathrm{I}(\beta^*_{G_j} \neq 0) \le s
\quad \text{and}\quad
\|\beta^*\|_{0} \coloneqq \sum_{i=1}^{p} \mathrm{I}(\beta^*_i \neq 0) \le s s_0 .
\end{equation}
\end{definition}
Double sparsity promotes sparsity both within and between groups. Specifically, it restricts the number of nonzero groups included in the model to $s$, and within these $s$ groups, the number of nonzero elements must be no more than $ss_0$. 
Intuitively, $s_0$ can be thought of as the average sparsity within the $s$ selected groups, providing insight into the sparsity levels within the nonzero groups.

\subsection{Related Work}

Recently, sparse group selection has emerged as a prominent area of high-dimensional structured sparsity learning. To tackle this problem, a combination of two penalized methods is often considered.
In order to perform sparse group selection, \citet{friedman2010note} and \citet{simon2013sparse} proposed sparse group Lasso (SGLasso), a combination of the Lasso penalty \citep{T1996} and the group Lasso penalty \citep{Y2006} joined together. Numerous efforts have been dedicated to accelerating the convergence of SGLasso \citep{ida2019fast, zhang2020efficient}.

The theoretical research on double sparsity began with \citet{cai2019sparse}, which established the minimax lower bounds for the estimation error of the double sparse linear regression, and the near-optimal upper bounds for the estimation error of SGLasso are obtained under the irrepresentable condition.
Moreover, they provided the theoretical guarantees for both the sample
complexity and estimation error of SGLasso. 
\citet{li2022minimax} concentrated on the Gaussian location model with a double sparse structure. They established the minimax rates for the estimation error over $\ell_u(\ell_{q})$ mixed-norm for $u,q \in [0, 1]$.
Despite these advancements, there still remains a dearth of methods with optimal theoretical guarantees.

Traditional convex relaxation-based methods, such as SGLasso, inherently introduce estimation bias for the coefficients, especially when large coefficients undergo significant shrinkage.
Moreover, \citet{bellec2018noise} demonstrated that convex estimators, such as the Lasso-type estimator, cannot attain the oracle estimation rate $O(\sigma\sqrt{\frac{ss_0}{n}})$, even when the beta-min condition is satisfied.
This phenomenon motivates us to develop computationally feasible non-convex algorithms, with iterative hard thresholding (IHT, \citet{blumensath2009iterative}) being a representative example.
IHT and its variants have garnered increasing attention for their efficacy in addressing a variety of high-dimensional statistical inference problems \citep{blumensath2010normalized,jain2014iterative, yuan2020dual, JMLR:v22:19-969}.
Given sparsity level $s'$, IHT performs a gradient descent step on the parameter $\beta$, followed by the selection of the $s'$ largest absolute values at each subsequent step.
Under restricted convexity/smoothness conditions, \citet{jain2014iterative} showed that IHT can obtain a minimax optimal estimator for high-dimensional M-estimation given a sufficient large sparsity level.
\citet{zhang2018} investigated the parameter estimation and support recovery of IHT for both $s=s^*$ and $s\gg s^*$ under RIP-type conditions.
\citet{giraud2021introduction} employed the IHT procedure in the context of linear regression with group sparsity and established the optimal upper bound for parameter estimation.
However, most of the related works consider the known sparsity level $s'$ as prior information,
making it challenging to analyze theoretical guarantees in the non-asymptotic sense without the knowledge of $s'$.
To tackle this problem, \citet{ndaoud2020scaled} proposed a fully adaptive IHT-style procedure, which can achieve the optimal rates for parameter estimation with unknown $s'$.

\subsection{Main Results and Contributions}
In this paper, our goal is to construct feasible methods for double sparse linear regression that are not only efficient but also with optimal statistical properties guaranteed.  To the best of our knowledge, our paper is the first to develop a fully adaptive optimal procedure for high-dimensional double sparse linear regression with unknown $s$, $s_0$, and $\sigma$. 
  
Addressing the signal under the double sparse assumption was an unresolved challenge until \citet{cai2019sparse,li2022minimax}. The approach employed in \citet{cai2019sparse} relies on sub-gradient and dual certificate constructions, applicable only in the context of $\ell_1$-type penalties. An earlier work by \citet{li2022minimax} introduced an IHT-style algorithm for detecting signals with a double sparse structure. They demonstrated the minimax optimality of the proposed algorithm for parameter estimation. However, this algorithm is impractical because it depends on the unknown parameters $s, s_0$, and $\sigma$. Notably, achieving adaptivity for double sparsity is much more challenging than for element-wise or group-wise sparsity. A natural approach is using a grid search technique for tuning the unknown parameters $s$ and $s_0$ such as \citet{cai2019sparse}. However, the grid search approach is computationally infeasible, and difficult to establish optimal guarantees from a theoretical perspective. Motivated by the adaptive framework for element-wise sparsity \citep{verzelen2012minimax, ndaoud2020scaled}, we develop a two-step adaptive procedure for parameter estimation and variable selection in the context of double sparse linear regression. 

Importantly, our procedure is not a simple combination of classical IHT \citep{ndaoud2020scaled} and group IHT \citep{giraud2021introduction}. The sequence of our two-step IHT operators is critical and the order cannot be interchanged. Specifically, reversing the order of these two steps could compromise the logical framework of the proof by contradiction. 

The advantages of double sparse IHT over convex counterparts, such as sparse group Lasso, are evident. Our theory is entirely based on the RIP-type condition, while the theory of sparse group Lasso (cf. \citet{cai2019sparse}) relies on a stronger irrepresentable condition. 
We further establish that under the beta-min conditions, our algorithm can achieve the oracle estimation rate $O(\sigma\sqrt{\frac{ss_0}{n}})$, showcasing the superiority of our algorithm over sparse group Lasso. Moreover, as far as we know, support recovery results in sparse group Lasso have not been established under mild assumptions, while we obtain the almost full recovery \citep{butucea2018variable} at both the element-wise and group-wise levels. This is further supported by synthetic and real-world data analyses.

In conclusion, the main contribution of this paper is summarized as follows:
\begin{itemize}
  \item We introduce a novel double sparse IHT operator that ensures both element-wise and group-wise sparsity. This operator consists of two steps that control model complexity efficiently. Building upon the double sparse IHT operator, we introduce a novel IHT-style procedure that dynamically updates the threshold at each iteration. We analyze upper bounds on the estimation error of our method and establish matching minimax lower bounds for the estimation error $O\left(\sqrt{\frac{\sigma^2}{n}(ss_0\log\frac{ed}{s_0}+s\log \frac{em}{s})} \right)$, conclusively demonstrating the optimality of our proposed approach.
  \item We propose a fully adaptive optimal procedure that handles unknown sparsity levels $s, s_0$ and noise level $\sigma$. Through our research, we demonstrate that the estimator obtained by our adaptive procedure attains optimal performance in the minimax sense. As far as we know, it is the first minimax adaptive procedure for the double sparse linear regression.
    Furthermore, we discover the pivotal role of the element sparsity level $s_0$ as the trade-off between IHT and group IHT, underscoring the superior performance of our method over both. We have implemented our proposals in an open-source R package named $\mathtt{ADSIHT}$.
    
    \item Under the element-wise and group-wise beta-min conditions, we establish that our algorithm attains the oracle estimation rate $O(\sigma\sqrt{\frac{ ss_0}{n}})$. This result indicates that our procedure performs comparably to the ordinary least-squares estimator when given the true support set. It highlights the superiority of our DSIHT procedure over convex counterparts such as sparse group Lasso in theory. Additionally, we demonstrate that our procedure achieves almost full recovery of the true support set at both the element and group levels.  
  \item We apply our proposed methods to both synthetic and real-world datasets, and comprehensive empirical comparisons with several state-of-the-art methods show the superiority of our method across a variety of metrics. Additionally, computational results for a real-world dataset demonstrate that our approach produces more accurate predictive power with fewer variables and groups.
\end{itemize}

\subsection{Organization}
The remainder of the paper is structured as follows. We introduce the notation used throughout the paper towards the end of this section. 
In Section \ref{analysis}, we introduce an IHT-style procedure with fast convergence and establish matching upper and lower bounds for estimation error. In Section \ref{adaptive}, we firstly propose a novel information criterion to determine the optimal stopping time and develop an adaptive procedure for conducting sparse group selection with unknown $s$ and $\sigma$. 
Then, we elucidate the connection between our work, IHT, and group IHT. We also present a minimax adaptive procedure to select the optimal value of $s_0$, which makes our method a fully adaptive optimal procedure. In Section \ref{oracle}, we establish that our DSIHT algorithm achieves the oracle estimation rate and accomplishes almost full recovery under the beta-min conditions.
In Section \ref{numerical}, we present numerical experiments comparing our methods with several state-of-the-art approaches using both synthetic and real-world datasets. Finally, in Section \ref{conclusion}, we provide a summary of our study and offer detailed proofs of our main results in the Appendix.

\subsection{Notations}\label{notation}
For the given sequences $a_n$ and $b_n$, we say that $a_n = O(b_n)$ or $a_n \lesssim b_n$ (resp. $a_n  = \Omega(b_n)$ or $a_n \gtrsim b_n$) when $a_n \le cb_n$ (resp. $a_n \ge c b_n$) for some
positive constant $c$. We write $a_n \asymp b_n$ if $a_n = O(b_n)$ and $a_n  = \Omega(b_n)$.
Let $d=\max_{1\le j \le m} |G_j|$ be the maximum group size.
Denote $[m]$ as the set $\{1,2,\ldots,m\}$, and $\mathrm{I}(\cdot)$ as the indicator function.
Let $x \vee y $ be the maximum of $x$ and $y$, while $x \wedge y $ is the minimum of $x$ and $y$.
Denote $S^* = \{i: \beta^*_i \neq 0\} \subseteq [p]$ as the support set of $\beta^*$.
Similarly, 
let $G^* = \{j: \beta^*_{G_j} \neq 0, G_j\subseteq [p], \textrm{\ and\ } G_j \cap G_{j^\prime} = \emptyset,  \forall j \neq j^\prime\} \subseteq [m]$ be the group-wise support set of $\beta^*$. Let $S_{G^*}=\cup_{j\in G^*}G_j$ be all the elements contained in groups $G^*$. Obviously, $S^* \subseteq S_{G^*}$.
For any set $S$ with cardinality $|S|$,
let $\beta^*_S=(\beta_j, j \in S) \in \mathbb{R}^{|S|}$ and $X_S = (X_j, j \in S) \in \mathbb{R}^{n \times |S|}$, and let $(X^\top X)_{SS} \in \mathbb{R}^{|S|\times |S|}$ be the submatrix of $X^\top X$ whose rows and columns are both listed in $S$.
For a vector $\beta$, denote $\|\beta\|_2$ as its Euclidean norm.
For a matrix $A$, denote $\|A\|_2$ as its spectral norm and $\|A\|_F$ as its Frobenius norm. Denote $\mathbb{I}_p$ as the $p\times p$ identity matrix.
Let $C, C_0, C_1,\ldots$ denote positive constants whose actual values vary from time to time.
Denote the parameter space of double sparsity as $\Theta^{m,d}(s, s_0)$.
Denote $\mathcal{S}^{m,d}(s,s_0)$ as the space consisting of all the support sets of $(s, s_0)$-sparse vector.
Notably, according to the definition of double sparsity,
we have $\mathcal{S}^{m,d}(a_1 s, b_1 s_0) \subseteq \mathcal{S}^{m,d}(a_2 s, b_2 s_0)$ for any positive constants $a_1b_1 =a_2b_2$ and $a_1 \leq a_2$.
For example, $\mathcal{S}^{m,d}(2s,2s_0)$ is a subspace of $\mathcal{S}^{m,d}(4s,s_0)$.
To facilitate computation, we assume $\|X_j\|_2  = \sqrt{n}$, $\forall j \in [p]$. 

\section{Analysis of minimax optimality}\label{analysis}
In Section \ref{operator}, we introduce the double sparse iterative hard thresholding (DSIHT) operator.
In particular, we provide a clear explanation of its construction and develop a DSIHT algorithm with known sparsity and noise levels. Following this, in Section \ref{upperbound}, we analyze the sources of estimation error. Then, we establish the upper bounds for parameter estimation of the DSIHT algorithm in Section \ref{mainresult}. In Section \ref{lowerbound}, we derive the minimax lower bound for double sparse linear regression, which yields that the upper bound in Section \ref{mainresult} is minimax optimal.

\subsection{Double sparse iterative hard thresholding operator}\label{operator}
Given $\lambda, s_0>0$, we define the double sparse iterative hard thresholding operator $\mathcal{T}_{\lambda, s_0}:\mathbb{R}^p \rightarrow \mathbb{R}^p$ as the following two steps:

{\bf Step 1 (Element-wise Condition Checking)}: define an element-wise hard thresholding operator $\mathcal{T}_{\lambda}^{(1)}: \mathbb{R}^p \rightarrow \mathbb{R}^p$ on $\beta \in \mathbb{R}^{p}$ as
$$
      \{\mathcal{T}^{(1)}_{\lambda}(\beta)\}_{j} = \beta_{j}\mathrm{I}(|\beta_{j}|\ge \lambda) , \quad \forall\\ j \in [p].
$$
The operator $\mathcal{T}^{(1)}_{\lambda}$ preserves the signal whose absolute magnitude is greater than or equal to $\lambda$, thus it can be seen as a preliminary screening process for identifying important variables.

{\bf Step 2 (Group-wise Condition Checking)}: denote
\begin{equation*}
    \mathcal{J}_{s_0} \coloneqq \{j \in [m] : \|\beta_{G_j}\|_2^2 \ge s_0 \lambda^2\}.
\end{equation*}
The definition of operator $\mathcal{T}^{(2)}_{\lambda, s_0}:\mathbb{R}^p \rightarrow \mathbb{R}^p$ is
\begin{align*}
  \{\mathcal{T}^{(2)}_{\lambda, s_0}(\beta)\}_{G_j} =
  \begin{cases}
    \beta_{G_j},\ &\text{if}\ j \in \mathcal{J}_{s_0}.\\
    0, &\text{if}\ j \in [m]\backslash\mathcal{J}_{s_0}.
  \end{cases}
  \end{align*}
  The operator $\mathcal{T}^{(2)}_{\lambda, s_0}$ selects groups with large magnitudes, utilizing group information to further filter the important variables. The operator $\mathcal{T}_{\lambda, s_0} = \mathcal{T}^{(2)}_{\lambda, s_0}\circ \mathcal{T}^{(1)}_{\lambda}$ is a composition of these two steps.
Unlike the classical IHT procedure, our procedure updates the threshold $\lambda$ in $\mathcal{T}_{\lambda, s_0}$ at each step in order to achieve both optimal statistical accuracy and fast convergence. Given $\lambda_0 > \lambda_{\infty} > 0$ and $0 < \kappa < 1$, we provide the form of the sequence $\{\lambda_t\}_{t=1}^{\infty}$ as follows
\begin{equation}\label{eq:iter2}
  \lambda_t = \kappa^t\lambda_0 \vee \lambda_{\infty},\ t = 0, 1, 2,\ldots
\end{equation}
For a given $s_0$ and sequence of threshold $\{\lambda_t\}_{t=1}^{\infty}$, we denote the estimators $\{\beta^t\}_{t=1}^{\infty}$ as

\begin{equation}\label{eq:iter1}
  \beta^t = \mathcal{T}_{\lambda_t, s_0}\left(\beta^{t-1} + \frac{1}{n}X^{\top} (y-X\beta^{t-1})\right),\ t = 1, 2, \ldots.
\end{equation}
Moreover, we denote the corresponding support set of $\{\beta^t\}_{t=1}^{\infty}$ as  $\{S^t\}_{t=1}^{\infty}$.
In the studies of variable selection, the misidentification of true support set $S^*$, i.e., $S^t \cap (S^*)^c$ is called type-I error,
and the omission of $S^*$, i.e., $(S^t)^c \cap S^*$ is called type-II error. We summarize our procedure as the following algorithm:

\begin{algorithm}[H]
  \caption{\label{alg:iht0}\textbf{D}ouble \textbf{S}parse \textbf{IHT} (DSIHT) algorithm with known $s, s_0$ and $\sigma$.}
  \begin{algorithmic}[1]
    \REQUIRE $X,\ y,\ \{G_j\}^m_{j=1},\ \kappa,\ \lambda_0, \ s_0,\ s,\ \sigma$.
    \STATE Initialize $t=0$, $\beta^t = 0$ and $\lambda_{\infty} = 4\sqrt{\frac{\sigma^2}{n}(\log \frac{ed}{s_0}+\frac{1}{s_0}\log \frac{em}{s})}$.
    \WHILE {$\lambda_t \geq \lambda_{\infty},\ $}
    \STATE ${\beta}^{t+1} = \mathcal{T}_{\lambda_t, s_0}\left({\beta}^{t} + \frac{1}{n}X^{\top}(y-X{\beta}^{t})\right)$.
    \STATE $\lambda_{t+1} = \kappa \lambda_{t}$.
    \STATE $t = t+1$.
    \ENDWHILE
    \ENSURE $\hat \beta = \beta^{t}$.
  \end{algorithmic}
\end{algorithm}
Here we offer an intuitive explanation for the choice of $\lambda_t$.
A large $\lambda_t$ promotes sparsity in the estimator $\beta^t$, which significantly reduces the type-I error by preventing spurious variables from being incorporated into the model.
However, excessive sparsity can result in a high type-II error by omitting too many true variables. As Section \ref{upperbound} shows, it leads to a high estimation error because the magnitude of $\beta^*$ is drastically shrunk to zero.
Conversely, a small $\lambda_t$ can reduce the type-II error by increasing the complexity of the model. Nevertheless, this allows too many spurious variables into the model, resulting in a high type-I error.
This intuition motivates us to choose the specific form of the sequence $\{\lambda_t\}_{t=1}^{\infty}$ by balancing these two types of errors.

In our procedure, we employ a decreasing sequence \eqref{eq:iter2} instead of directly setting the threshold as this order. The reason is that such a small threshold can potentially result in the selection of too many unimportant variables at the initial step.
This lack of sparsity makes our procedure hard to benefit from the contraction property of the DSRIP condition, and the estimation error cannot be well-controlled in iterations.
In comparison, a sufficiently large $\lambda_0$ identifies a small set of variables, effectively controlling the false discoveries of the initial solution. 
With the decrease of the threshold, we optimize the solution in an appropriate direction iteratively without losing sparsity. A novelty of our procedure lies in the fact that it implicitly controls the type-I error at a low level at each step, and reduces the type-II error through iterations.
In Theorem \ref{th1}, we choose $\lambda_{\infty}\asymp\sqrt{\frac{\sigma^2}{n}(\log \frac{ed}{s_0}+\frac{1}{s_0}\log \frac{em}{s})}$ and show its optimality in the minimax sense.

\subsection{Analysis of estimation error}\label{upperbound}
To conduct the theoretical analysis, we decompose the iterative term into three parts:
\begin{align}\label{eq:H}
  \begin{split}
  H^{t+1} \coloneqq&{\beta}^{t} + \frac{1}{n}X^{\top}(y-X{\beta}^t)\\
   =& \beta^*+\left(\frac{1}{n} X^{\top} X-\mathbb{I}_p\right)(\beta^* - {\beta}^t) +\frac{1}{n} X^{\top} \xi\\
   = & \beta^* + \Phi (\beta^*-\beta^t) + \Xi,
  \end{split}
  \end{align}
where 
$\Phi \coloneqq \frac{1}{n}X^{\top}X-\mathbb{I}_p\ \text{and}\ \Xi  \coloneqq \frac{1}{n}X^{\top}\xi .
$
Equation \eqref{eq:H} shows that the estimation error comes from three sources: 
\begin{itemize}
  \item The true parameters $\beta^*$ shrunk by mistake.
  \item The optimization error that $\beta^t$ approximates $\beta^*$.
  \item The randomness caused by the errors $\xi$.
\end{itemize}
Among these three sources, the optimization error corresponds to the iterative procedure,
and the randomness of our proposed procedure mainly comes from the third term $\Xi$.
In what follows, we detail how to upper bound the latter two sources of errors accurately.
Firstly, we introduce an essential condition for the design matrix $X$ in order to get a contraction of the optimization error. 

\begin{assumption}
[DSRIP condition]\label{df2}
We say that $X \in \mathbb R^{n\times p}$ satisfies the Double Sparse Restricted Isometry Property $DSRIP(s,s_0, \delta)$ with constant $0 < \delta < 1$, if $\forall S \in \mathcal S^{m,d}(s,s_0)$ and $\forall u \neq 0, u \in \mathbb{R}^{|S|}$, it holds that
$$
1-\delta  \leq \frac{\left\|X_{S} u\right\|_2^2}{n\|u\|_2^2} \leq 1+\delta.
$$
\end{assumption}

\begin{remark}
The Double Sparse Restricted Isometry Property (DSRIP) serves as a natural extension of the ordinary RIP condition \citep{candes2005} under the double sparse linear regression. For sub-Gaussian design, considering a $p$-dimensional $ss_0$-sparse structure, we require a sample size of $n = \Omega(ss_0\log \frac{ep}{ss_0})$ to ensure that the RIP condition holds with high probability. However, for the satisfaction of the DSRIP condition, we only need $n = \Omega(ss_0\log \frac{ed}{s_0}+s\log \frac{em}{s})$. It is worth noting that, given $p = m\times d$, the DSRIP condition can be satisfied with a smaller sample size compared to RIP. Further details can be found in Appendix C.
\end{remark}

DSRIP serves as an essential component for analyzing the high-dimensional double sparse linear regression \citep{li2022minimax}.
It imposes a less stringent condition than the ordinary RIP. Assuming the same element-wise sparsity, DSRIP only requires subsets of $ss_0$-sparse vectors with no more than $s$ groups to be satisfied, whereas RIP requires all $ss_0$-sparse vectors to hold.
If design matrix $X$ satisfies DSRIP($s, s_0, \delta$), we have $\|\Phi\|_2 \leq \delta < 1$, demonstrating that $\Phi$ serves as the contraction factor for all $(s, s_0)$-sparse vectors. As a result, by leveraging both DSRIP and the sparse structure of the signal, the contraction factor $\Phi$ enables iterative reduction of the optimization error.

Next, we turn to the analysis of the random error term $\Xi$.
To upper bound this source of error, we need to capture the complexity of the noise term.
\begin{lemma}\label{lemma:iht1}
  Assume that $X$ satisfies DSRIP$(s, s_0, \delta)$.
  Then, there exists a constant $C>0$, the event
  \begin{align*}
    \mathcal{E} \coloneqq \left\{\forall S\in \mathcal{S}^{m,d}(s,s_0):  \sum_{i\in S}\Xi_{i}^2 \le \frac{4\sigma^2}{n} \left(ss_0 \log \frac{ed}{s_0}+s\log \frac{em}{s} \right) \right\}
  \end{align*}
  holds with probability at least $1-\exp\left\{-C(ss_0 \log \frac{ed}{s_0}+s\log \frac{em}{s})\right\}$.
\end{lemma}
Lemma \ref{lemma:iht1} provides the uniform upper bounds of the random error term with high probability.
We now analyze the random term $\Xi$ in detail and decompose the source of random errors into two parts: 

\begin{itemize}
 \item The random errors $\Xi$ attached to the true support set $S^*$.
 \item The random errors $\Xi$ caused by type-I error, the mis-identification of true parameters $\beta^*$. More concretely, some random errors escape from operator $\mathcal{T}_{\lambda,s_0}$, which we call these errors as pure errors below.
\end{itemize}
The errors caused by random errors $\Xi$ can be attributed to two sources: the random errors corresponding to $S^*$ and $(S^*)^c$, respectively. Since $S^* \in \mathcal{S}^{m,d}(s,s_0)$ is with a sparse prior, the random errors attached to $S^*$ can be well-bounded by event $\mathcal{E}$ with high probability.
However, it is difficult to find an upper bound for the pure errors since the amount of the pure errors is undetermined.
Therefore, the central problem that operator $\mathcal{T}_{\lambda,s_0}$ addresses is to bound the support set of the pure errors.
Intuitively, we want to collect the pure errors in some subsets belonging to $\mathcal{S}^{m,d}(s,s_0)$.
Then, the magnitude of pure errors can be upper bounded by event $\mathcal{E}$.

We consider applying $\mathcal{T}_{\lambda,s_0}$ to the pure errors directly and show that if the pure errors overflow $\mathcal{S}^{m,d}(s,s_0)$, it will contradict with $\mathcal{E}$ with high probability.
According to the structure of $\mathcal{S}^{m,d}(s,s_0)$, we decompose the discussion into two cases:
\begin{itemize}
  \item[\bf{Case 1:}] Assume that the set selected by $\mathcal{T}_{\lambda,s_0}$ lies in no more than $s$ groups but the amount exceeds $ss_0$. 
  Element-wise condition checking ensures that all the selected entries are larger than $\lambda$. 
  Then, for any $(s, s_0)$-shaped subset of this set with cardinality $ss_0$,
  the total magnitude of these subsets exceeds $ss_0\lambda^2$. With the choice of $\lambda \geq 2\sqrt{\frac{\sigma^2}{n}(\log \frac{ed}{s_0}+\frac{1}{s_0}\log \frac{em}{s})}$, we have $ss_0\lambda^2\geq 4 \frac{\sigma^2}{n}(ss_0\log \frac{ed}{s_0}+s\log \frac{em}{s})$, which contradicts event $\mathcal{E}$ with high probability.
  We provide an illustrative example in Figure \ref{fig1}.
  \item[\bf{Case 2:}] Assume that the set selected by $\mathcal{T}_{\lambda,s_0}$ lies in more than $s$ groups, yet within any $s$ selected groups, the number of the selected entries does not exceed $ss_0$. 
  Group-wise condition checking implies that the magnitude of each selected group is larger than $s_0 \lambda^2$.
  Consequently, the $(s, s_0)$-shaped subset consisting of any $s$ selected groups satisfies that the total magnitude exceeds $ss_0\lambda^2$.
  For $\lambda \geq 2\sqrt{\frac{\sigma^2}{n}(\log \frac{ed}{s_0}+\frac{1}{s_0}\log \frac{em}{s})}$, it contradicts with event $\mathcal{E}$ with high probability.
  We provide an illustrative example in Figure \ref{fig2}.  
  Notably, if there exist $s$ selected groups with the number of selected entries exceeding $ss_0$, we analyze this case similarly to {\bf{Case 1}}.

\end{itemize}

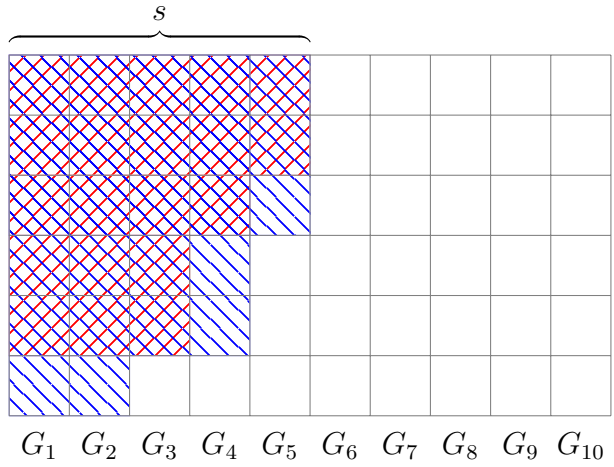
\begin{figure}
  \begin{center}
    \begin{tikzpicture}[scale = 0.8]
      \draw[color=red!40,
      pattern={mylines[size= 5pt,line width=.8pt,angle=45]},
      pattern color=red] (0,4) rectangle (5,6);
      \draw[color=red!40,
      pattern={mylines[size= 5pt,line width=.8pt,angle=45]},
      pattern color=red] (0,3) rectangle (4,4);
      \draw[color=red!40,
      pattern={mylines[size= 5pt,line width=.8pt,angle=45]},
      pattern color=red] (0,1) rectangle (3,3);
  \draw[color=blue!40,
      pattern={mylines[size= 5pt,line width=.8pt,angle=-45]},
      pattern color=blue] (0,1) rectangle (4,6);
      \draw[color=blue!40,
      pattern={mylines[size= 5pt,line width=.8pt,angle=-45]},
      pattern color=blue] (0,0) rectangle (2,1);
      \draw[color=blue!40,
      pattern={mylines[size= 5pt,line width=.8pt,angle=-45]},
      pattern color=blue] (0,3) rectangle (5,6);
    \node[] at (2.5,6.7) {$s$};
    \node[] at (0.5,-0.5) {$G_1$};
    \node[] at (1.5,-0.5) {$G_2$};
    \node[] at (2.5,-0.5) {$G_3$};
    \node[] at (3.5,-0.5) {$G_4$};
    \node[] at (4.5,-0.5) {$G_5$};
    \node[] at (5.5,-0.5) {$G_6$};
    \node[] at (6.5,-0.5) {$G_7$};
    \node[] at (7.5,-0.5) {$G_8$};
    \node[] at (8.5,-0.5) {$G_9$};
    \node[] at (9.5,-0.5) {$G_{10}$};
    \draw[step=1,color=gray] (0,0) grid (10,6);
    \draw [thick, decorate, 
		decoration = {calligraphic brace, 
			raise=5pt, 
			aspect=0.5, 
			amplitude=4pt 
		}] (0,6) --  (5,6);
    \end{tikzpicture} 
    \end{center}
    \caption{Illustrative example of {\bf{case 1}}. There are 10 groups with equal group size $d=6$, and we reshape the group structure as a $6\times 10$ matrix with each column representing a group.
    Here $s=5$ and $s_0 = 4$.
    The blue region represents the selected set, 
    and the red region represents a $(s, s_0)$-shaped subset satisfying that total magnitude exceeds $ss_0 \lambda^2$.
    Here the cardinality of the red-colored set is $s \times s_0=20$. Note that the whole vector of support is reshaped into a matrix with a particular group structure. 
  }\label{fig1}
  \end{figure}

  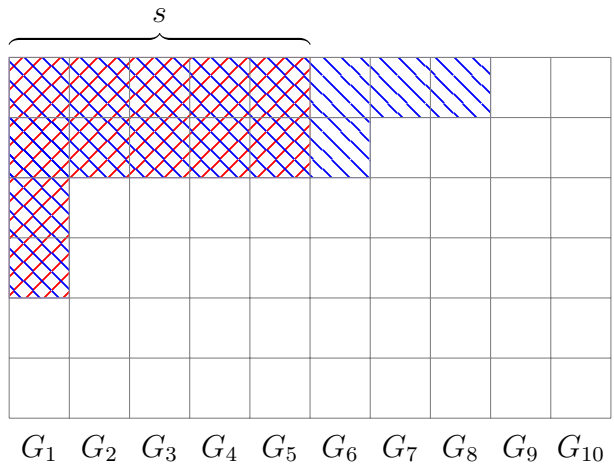
\begin{figure}
    \begin{center}
      \begin{tikzpicture}[scale = 0.8]
        \draw[color=red!40,
        pattern={mylines[size= 5pt,line width=.8pt,angle=45]},
        pattern color=red] (0,4) rectangle (5,6);
        \draw[color=red!40,
        pattern={mylines[size= 5pt,line width=.8pt,angle=45]},
        pattern color=red] (0,2) rectangle (1,4);
    \draw[color=blue!40,
        pattern={mylines[size= 5pt,line width=.8pt,angle=-45]},
        pattern color=blue] (0,4) rectangle (6,5);
        \draw[color=blue!40,
        pattern={mylines[size= 5pt,line width=.8pt,angle=-45]},
        pattern color=blue] (0,5) rectangle (8,6);
    \draw[color=blue!40,
        pattern={mylines[size= 5pt,line width=.8pt,angle=-45]},
        pattern color=blue] (0,2) rectangle (1,5);
      \node[] at (0.5,-0.5) {$G_1$};
      \node[] at (1.5,-0.5) {$G_2$};
      \node[] at (2.5,-0.5) {$G_3$};
      \node[] at (3.5,-0.5) {$G_4$};
      \node[] at (4.5,-0.5) {$G_5$};
      \node[] at (5.5,-0.5) {$G_6$};
      \node[] at (6.5,-0.5) {$G_7$};
      \node[] at (7.5,-0.5) {$G_8$};
      \node[] at (8.5,-0.5) {$G_9$};
      \node[] at (9.5,-0.5) {$G_{10}$};
      \draw[step=1,color=gray] (0,0) grid (10,6);
      \node[] at (2.5,6.7) {$s$};
      \draw [thick, decorate, 
      decoration = {calligraphic brace, 
        raise=5pt, 
        aspect=0.5, 
        amplitude=4pt 
      }] (0,6) --  (5,6);
      \end{tikzpicture} 
      \end{center}
      \caption{Illustrative example of {\bf{case 2}}.
      The elements in Figure \ref{fig2} are the same as in Figure \ref{fig1}.
      The entries of the red region cover $s=5$ groups and its cardinality is less than $s\times s_0=20$.
      }\label{fig2}
    \end{figure}
Overall, by applying operator $\mathcal{T}_{\lambda, s_0}$ directly, $\Xi$ can be shrunk into a $(s, s_0)$-shaped subset with high probability.

\subsection{Upper bound for estimation error}\label{mainresult}
In Section \ref{upperbound}, we have introduced the idea to control the estimation error caused by optimization error and randomness.
Formally speaking, the three sources of estimation error can be bounded in sequence. 
In what follows, we analyze the error bounds of our proposed procedure.
The main result of our theoretical analysis is given by Theorem \ref{th1}.

\begin{theorem}\label{th1}
  Assume that $\beta^*$ is $(s,s_0)$-sparse and $X$ satisfies DSRIP$(3s, \frac53 s_0,\delta )$. 
  Assume that $\delta<0.11\wedge \kappa^{10}$, $\|\beta^* \|_2 \le \sqrt{ss_0}\lambda_0$ and
  $
  \lambda_{\infty} \ge 4\sqrt{\frac{\sigma^2}{n}(\log \frac{ed}{s_0}+\frac{1}{s_0}\log \frac{em}{s})}.
  $
  We run Algorithm \ref{alg:iht0} and obtain the corresponding solution sequence $\{\beta^{t}\},t=1,2,\cdots$.
 Then, with probability at least  $1-\exp\left\{-C(ss_0\log\frac{ed}{s_0}+s\log \frac{em}{s})\right\}$, we have
 \begin{itemize}
  \item [(\romannumeral1)] Inside groups $G^*$, the type-I error can be controlled by a $(s, s_0)$-shaped subset, that is, 
  \begin{equation}\label{eq:iht3}
    S_{G^*} \cap S^t \cap (S^*)^c  \in \mathcal{S}^{m,d}(s,s_0).
  \end{equation}
  \item [(\romannumeral2)] Outside groups $G^*$, the type-I error can be controlled by a $(s, s_0)$-shaped subset, that is, 
  \begin{equation}\label{eq:iht5}
    S_{G^*}^c \cap S^t  \in \mathcal{S}^{m,d}(s,s_0).
  \end{equation}
  \item [(\romannumeral3)] The upper bounds for estimation error are
  \begin{equation}\label{eq:iht4}
    \|\beta^* - \beta^t \|_2 \le \frac{3}{2}(1+\sqrt{2})\sqrt{ss_0}\lambda_t.
  \end{equation}
\end{itemize}
\end{theorem}
Part (\romannumeral1) of Theorem \ref{th1} shows that the type-I error of $\{\beta^t\}$ within the true groups $G^*$ can be controlled in a $(s, s_0)$-shaped set.
Part (\romannumeral2) of Theorem \ref{th1} asserts that our procedure selects fewer than $s$ incorrect groups into the model, and at most $ss_0$ variables outside groups $G^*$.
Together, they show that the solution sequence $\{\beta^t\}$ generated by our procedure is $(2s, \frac{3}{2}s_0)$-sparse at each step, affirming that our procedure effectively controls false discoveries at both the element and group levels.
The non-convexity of the IHT-style method may cause the parameter estimation error to not decrease at each step. To address this issue, a common approach to get around this issue is constructing a surrogate function of the upper bound that decreases exponentially \citep{zhang2018, Zhu202014241, zhang2022}.
With the choice of $\{\lambda_t\}$, \eqref{eq:iht4} gives a decreasing upper bound for the parameter estimation error.
Notably, with the choice of $\lambda_{\infty} \asymp\sqrt{\frac{\sigma^2}{n}(\log \frac{ed}{s_0}+\frac{1}{s_0}\log \frac{em}{s})}$, the upper bound decays geometrically to the minimax lower bounds in \eqref{lower_bound}, which demonstrates the optimality of our procedure in the minimax sense.
\begin{remark}
  In the above discussion, we have discussed the idea of the construction of $\mathcal{T}_{\lambda, s_0}$ by applying it to $\Xi$ directly.
  In our practical procedure, we apply $\mathcal{T}_{\lambda, s_0}$ to $H^t$ rather than $\Xi$.
  Referring to the two cases above, we can show that
  \begin{itemize}
    \item [(i)] Inside the true groups $G^*$, if $S^t \cap (S^*)^c \notin \mathcal{S}^{m,d}(s,s_0)$, there exists a $(s,s_0)$-shaped subset $\tilde{S}_{1,t} \subseteq S^t\cap S_{G^*}\cap (S^*)^c$ such that $ss_0 \lambda_{t+1}^2 \le \sum_{i \in \tilde{S}_{1, t}}\{\mathcal{T}_{\lambda_{t+1,s_0}}(H^{t+1})\}_{i}^2$.
    \item [(ii)] Outside the true groups $G^*$, if $S^t \cap (S^*)^c \notin \mathcal{S}^{m,d}(s,s_0)$, there exists a $(s,s_0)$-shaped subset $\tilde{S}_{2,t} \subseteq S^t\cap S_{G^*}^c$ such that $ss_0 \lambda_{t+1}^2 \le \sum_{i \in \tilde{S}_{2, t}}\{\mathcal{T}_{\lambda_{t+1,s_0}}(H^{t+1})\}_{i}^2$.
  \end{itemize}
 Notably, our proof mainly relies on the method of mathematical induction. Assuming the results \eqref{eq:iht3},\eqref{eq:iht5},\eqref{eq:iht4} in Theorem \ref{th1} hold for step $t$, we first prove that \eqref{eq:iht3} and \eqref{eq:iht5} hold for step $t+1$ by induction hypothesis. 
  We then combine the induction hypothesis with \eqref{eq:iht3} and \eqref{eq:iht5} for step $t+1$ to establish \eqref{eq:iht4}, completing the inductive steps.
\end{remark}

\begin{remark}
Here we elaborate on why we split the analysis of false discovery into two cases.
Subsequently, we present an example demonstrating that in the false discovery $S^t \cap (S^*)^c$, there does not exist a subset $\tilde S_t$ satisfying  $\tilde{S}_{t} \subseteq S_{G^*}^c$ such that $ss_0 \lambda_{t+1}^2 \le \sum_{i \in \tilde{S}_{t}}\{\mathcal{T}_{\lambda_{t+1,s_0}}(H^{t+1})\}_{i}^2$.
  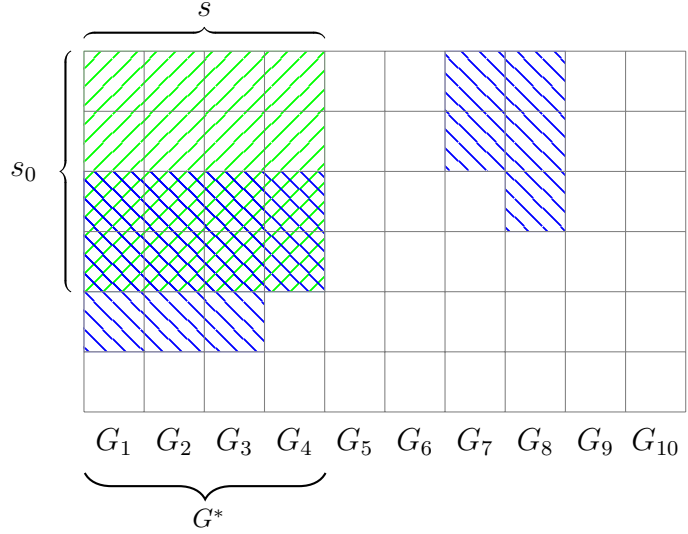
\begin{figure}
    \begin{center}
      \begin{tikzpicture}[scale = 0.8]
        \draw[color=green!40,
        pattern={mylines[size= 5pt,line width=.8pt,angle=45]},
        pattern color=green] (0,2) rectangle (4,6);
    \draw[color=blue!40,
        pattern={mylines[size= 5pt,line width=.8pt,angle=-45]},
        pattern color=blue] (0,1) rectangle (3,4);
    \draw[color=blue!40,
        pattern={mylines[size= 5pt,line width=.8pt,angle=-45]},
        pattern color=blue] (0,2) rectangle (4,4);
    \draw[color=blue!40,
        pattern={mylines[size= 5pt,line width=.8pt,angle=-45]},
        pattern color=blue] (6,4) rectangle (7,6);\
    \draw[color=blue!40,
        pattern={mylines[size= 5pt,line width=.8pt,angle=-45]},
        pattern color=blue] (7,3) rectangle (8,6);
      \node[] at (0.5,-0.5) {$G_1$};
      \node[] at (1.5,-0.5) {$G_2$};
      \node[] at (2.5,-0.5) {$G_3$};
      \node[] at (3.5,-0.5) {$G_4$};
      \node[] at (4.5,-0.5) {$G_5$};
      \node[] at (5.5,-0.5) {$G_6$};
      \node[] at (6.5,-0.5) {$G_7$};
      \node[] at (7.5,-0.5) {$G_8$};
      \node[] at (8.5,-0.5) {$G_9$};
      \node[] at (9.5,-0.5) {$G_{10}$};
      \node[] at (2,6.7) {$s$};
      \draw [thick, decorate, 
      decoration = {calligraphic brace, 
        raise=5pt, 
        aspect=0.5, 
        amplitude=4pt 
      }] (0,6) --  (4,6);
      \node[] at (-1,4) {$s_0$};
      \draw [thick, decorate, 
      decoration = {calligraphic brace, 
        raise=5pt, 
        aspect=0.5, 
        amplitude=4pt 
      }] (0,2) --  (0,6);
      \draw[step=1,color=gray] (0,0) grid (10,6);
      \draw [thick, decorate, decoration={brace,amplitude=10pt,mirror},xshift=0.4pt,yshift=-0.4pt](0,-1) -- (4,-1) node[black,midway,xshift = 0.05cm, yshift=-0.6cm] {\footnotesize $G^*$};
      \end{tikzpicture} 
      \end{center}
      \caption{Illustrative example of two cases of false discovery. Here $G^* = \{G_1, G_2, G_3, G_4\}$ and $s=s_0=4$.
      The green region represents the true support set $S^*$ and the blue region represents the selected set $S^t$.
      The remaining elements in Figure \ref{fig3} are the same as in Figure \ref{fig1}.
      }\label{fig3}
    \end{figure}

    In Figure \ref{fig3}, it is easy to verify that $S^{t} \cap (S^*)^c$ has 8 entries and $S^{t} \cap (S^*)^c \notin \mathcal{S}^{m,d}(s, s_0)$ since it covers $5$ groups.
    By the group-wise condition checking, $\|\{\mathcal{T}_{\lambda_{t+1,s_0}}(H^{t+1})\}_{G_i}\|_2^2 \geq s_0 \lambda_{t+1}^2$ for $i = 7, 8$.
    On the other hand, inside $G^*$, the absolute value of each element of $S^{t} \cap (S^*)^c$ is not less than $\lambda_{t+1}$.
    However, we cannot find a $(s, s_0)$-shaped subset such that $ss_0 \lambda_{t+1}^2 \le \sum_{i \in \tilde{S}_{t}}\{\mathcal{T}_{\lambda_{t+1,s_0}}(H^{t+1})\}_{i}^2$.
    Therefore, we consider covering the false discovery inside $G^*$ and outside $G^*$ by two $(s, s_0)$-shaped subsets, respectively.
\end{remark}

\subsection{Minimax lower bound for double sparse linear regression}\label{lowerbound}

In previous works, minimax rates for the high-dimensional sparse linear regression have been studied thoroughly. A number of papers focus on element-wise $s$-sparsity class \citep{raskutti2011minimax, verzelen2012minimax, bellec2018slope}, and there is also some work devoted to group sparsity such as \citet{10.1214/09-AOS778} and \citet{tsy2011}. 
Recently, \citet{cai2019sparse} provided the non-asymptotic minimax lower bounds of double sparse linear regression.
Here we prove it using a more concise technique. 
Consider parameter space $\widetilde\Theta^{m,d}(s, s_0)$:
 $$
 \widetilde\Theta^{m,d}(s, s_0) \coloneqq \{\beta \in \mathbb{R}^p :\|\beta\|_{0,2}\le s\ \text{and}\ \|\beta_{G_j}\|_0\le s_0,\forall j\in [m]\}.
 $$ 
 Unlike $\Theta^{m,d}(s, s_0)$, $\widetilde\Theta^{m,d}(s, s_0)$ imposes an $\ell_0$-ball constraint on each group with a radius of $s_0$. Additionally, the total sparsity of $\widetilde\Theta^{m,d}(s, s_0)$ is limited to $ss_0$. It can be easily observed that $\widetilde\Theta^{m,d}(s, s_0) \subseteq \Theta^{m,d}(s, s_0)$.
 Therefore,
 \begin{equation*}\label{lower_relation}
 \inf_{\hat\beta}\sup_{\beta^* \in \Theta^{m,d}(s, s_0)}\mathbf{E}_{\hat \beta}\|\hat\beta-\beta^*\|_2^2 \geq \inf_{\hat\beta}\sup_{\beta^* \in \widetilde\Theta^{m,d}(s, s_0)}\mathbf{E}_{\hat \beta} \|\hat\beta-\beta^*\|_2^2,
\end{equation*}
 where $\mathbf E_{\hat\beta}$ represents the expectation with respect to $\hat \beta$.
\begin{definition}[Packing Number] 
A $\rho$-packing of a set $\mathcal{S}$ with repsect to a metric $\|\cdot\|_\psi$ is a collection $\left\{\beta^{1}, \ldots, \beta^{M}\right\} \subset \mathcal{S}$ such that $\|\beta^{i}- \beta^{j}\|_\psi>\rho$ for all distinct $i, j \in [M]$. The $\rho$-packing number $M(\delta ; \mathcal{S}, \|\cdot\|_\psi)$ is the cardinality of the largest $\rho$-packing.
\end{definition}
 Let $M(\rho;\widetilde\Theta^{m,d}(s, s_0), \|\cdot\|_H)$ be the cardinality of $\rho$-packing set of the parameter space $\widetilde\Theta^{m,d}(s, s_0)$ with repsect to Hamming metric $\|\cdot\|_H$.
 The lower bounds for the packing number of $ \widetilde\Theta^{m,d}(s, s_0)$ are provided as follows.
 \begin{lemma}[Lower bounds for the packing number \citep{li2022minimax}]\label{lem2}
	The cardinality of $ \frac{ss_0}{4}$-packing set of $\widetilde\Theta^{m,d}(s, s_0)$ is lower bounded as
	\begin{equation*}
		\log\left(M(\frac{ss_0}{4};\widetilde\Theta^{m,d}(s, s_0), \|\cdot\|_H)\right) \geq \frac{ss_0\log \frac{ed}{s_0}+s\log \frac{em}{s}}{4}.
	\end{equation*}
\end{lemma}
\citet{li2022minimax} leveraged the structures of double sparsity and combined multi-ary Gilbert-Varshamov bounds \citep{gilbert1952comparison} to construct the packing set of $\widetilde\Theta^{m,d}(s, s_0)$ in a more concise way.
By combining Lemma \ref{lem2}, we establish a minimax lower bound that is consistent with the results presented in \citet{cai2019sparse}. This is stated in the following theorem.
\begin{theorem}\label{thm:lower}
  Consider linear regression model $y = X\beta^* + \varepsilon$, where $\varepsilon \sim \mathcal{N}(0,\sigma^2\mathrm{I}_n)$. 
  Denote the maximal $(2s, 2s_0)$-sparse eigenvalue as
$$
\vartheta_{\max} = \max_{u \in \Theta^{m,d}(2s, 2s_0)}\frac{\|Xu\|_2}{\sqrt{n}\|u\|_2}.
$$
  Assume that $\vartheta_{\max} < \infty$.
  Then, we have
\begin{equation}\label{lower_bound}
	 \inf_{\hat\beta}\sup_{\beta^* \in \Theta^{m,d}(s, s_0)} {\mathbf{E}}_{\hat \beta } \|\hat\beta-\beta^*\|_2^2 \geq \frac{\sigma^2}{512 \vartheta_{\max}^2n}\left(ss_0\log \frac{ed}{s_0}+s\log \frac{em}{s}\right).
\end{equation}
\end{theorem}
Theorem \ref{thm:lower} establishes the lower bounds for the $\ell_2$ estimation errors, which are consistent with the results in \citet{cai2019sparse}.
 The estimation error for $\beta^{t_{\infty}}$ matches the minimax lower bound \eqref{lower_bound}, demonstrating the optimality of our IHT-style procedure. 

\section{A fully adaptive optimal procedure}\label{adaptive}
The procedure proposed in Section \ref{analysis} relies on the unknown sparsity levels $s, s_0$, and noise level $\sigma$, which pose a challenge in practical applications. To address this, we adopt a data-driven approach to determine the initial threshold and the optimal stopping time of our procedure, making it more feasible for real-world settings.
Given $s_0$, we introduce a procedure that is adaptive to the unknown $s$ and $\sigma$ in Section \ref{adaptives}.
In Section \ref{tradeoff}, we explore the trade-off between classical IHT and group IHT with respect to different values of $s_0$. Finally, we propose a data-adaptive tuning approach for $s_0$ and demonstrate its optimality, rendering our method a fully adaptive procedure.

\subsection{Adaptation to unknown $s$ and $\sigma$}\label{adaptives}
In the remaining part of Section \ref{adaptives}, we assume that sparsity level $s_0$ is given.
Firstly, we introduce the adaptive choice of the initial threshold $\lambda_0$.
The assumption of Theorem \ref{th1} provides a lower bound for the choice of $\lambda_0$. 
However, choosing a significantly large value of $\lambda_0$ may decrease the efficiency of the algorithm from an optimization perspective since it can result in more redundant iterations.
In the rest of our paper, denote
$$
M \coloneqq \frac{1}{n}X^{\top} y = \beta^* + \Phi \beta^* + \Xi\quad \text{and}\quad {\sigma}_t^2 \coloneqq \frac{1}{n}\|y - X{\beta}^t\|_2^2.
$$
We provide an explicit form of $\lambda_0$ as 
\begin{equation}\label{iht5}
 \lambda_0 \coloneqq \frac{100}{9}\sqrt{\frac{\sigma_0^2}{n} \left(\log\frac{ed}{s_0}+\frac{1}{s_0}\log em \right) } \vee \frac{19}{4}\|M\|_{\infty},
\end{equation}
where $\|M\|_{\infty} \coloneqq \max\limits_{i} |M_i|$. 

\begin{theorem}\label{lem:lam0}
  Assume that $\beta^*$ is $(s,s_0)$-sparse and $X$ satisfies $\mbox{DSRIP}(2s, \frac{3}{2}s_0, \delta )$. 
  Assume that $\delta<0.11 $ and  $n > 105^2(ss_0\log \frac{ed}{s_0}+s\log em)$. 
  Then, with probability at least $1-\exp\{-C(ss_0\log \frac{ed}{s_0}+s\log \frac{em}{s})\}$, we have $\|\beta^*\|_2 \leq ss_0\lambda_0$.
\end{theorem}
Theorem \ref{lem:lam0} states that the choice of \eqref{iht5} guarantees the satisfaction of the assumption in Theorem \ref{th1} with high probability. Next, we define three stopping times $t_{\infty}, t_0$ and $\bar t$ as follows
		\begin{align}
      \begin{split}
				&t_\infty \coloneqq \inf\left\{t: \lambda_t \le 4\sqrt{\frac{\sigma^2}{n}\left(\log \frac{ed}{s_0}+\frac{1}{s_0}\log \frac{em}{s}\right)}\right\},\\
				&t_0 \coloneqq \inf\left\{t: \lambda_t \le 12\sqrt{\frac{\sigma^2}{n}\left(\log \frac{ed}{s_0}+\frac{1}{s_0}\log em\right)}\right\},\\
				&\bar{t} \coloneqq \inf\left\{t: \lambda_t \le 8\sqrt{\frac{\sigma_t^2}{n}\left(\log \frac{ed}{s_0}+\frac{1}{s_0}\log em\right)}\right\}.
      \end{split}
      \end{align}
$t_\infty$ is the stopping time that hits the optimal threshold $\lambda_{\infty}$.
Obviously, $\bar t$ is an accessible stopping time that is independent of $s$ and $\sigma$. On the other hand, $t_0$ and $t_{\infty}$ are the theoretical stopping time that corresponds to the unknown parameters $s$ and $\sigma$.
We state the relationship among these three stopping times in the following theorem.
\begin{theorem}\label{thm:stopping}
Assume all the conditions in Theorem \ref{th1} hold and sample size $n > 105^2(ss_0\log \frac{ed}{s_0}+s\log em)$. Then, with probability at least  $1-\exp\left\{-C(ss_0\log \frac{ed}{s_0}+s\log \frac{em}{s})\right\}$, we have
$$t_0\leq \bar t \leq t_{\infty}.$$
\end{theorem}
Theorem \ref{thm:stopping} shows that $\bar t$ can be bounded by the theoretical stopping times $t_0$ and $t_{\infty}$. 
In particular, since $\bar t$ is dominated by the optimal stopping time $t_{\infty}$, the estimation error $\|\beta^{\bar t} - \beta^*\|$ can be upper bounded by \eqref{eq:iht4}.
Additionally, Theorem \ref{th1} implies that $\beta^{t_0}$ is sub-optimal in the minimax sense.
More concretely, we can deduce that $\beta^{\bar t}$ achieves optimal statistical accuracy up to a logarithmic factor. 
We state this minimax sub-optimal result as Corollary \ref{col1}.

\begin{corollary}\label{col1}
Assume the conditions in Lemma \ref{thm:stopping} hold. Then, we have
$$
\sup\limits_{S^* \in \mathcal{S}^{m, d}(s, s_0)} P\left(\|\beta^{\bar t} - \beta^*\|_2 \geq 50\sqrt{\frac{\sigma^2}{n} \left(ss_0\log \frac{ed}{s_0}+s\log em\right)}\right)\leq e^{-C(ss_0\log \frac{ed}{s_0}+s\log \frac{em}{s})}.
$$
\end{corollary}
Corollary \ref{col1} is a direct consequence of Theorem \ref{thm:stopping}. It demonstrates that stopping at $\bar t$ is a minimax sub-optimal procedure.
The next open question is whether we can improve this sub-optimal procedure to be minimax optimal. 
The following analysis answers the question positively under certain conditions.
Denote
\begin{equation}\label{massart}
   \Omega(\beta) \coloneqq  s_0\|\beta\|_{G} \log \frac{ed}{s_0} + \|\beta\|_{G}\log \frac{em}{\|\beta\|_{G}}, 
\end{equation}
where $\|\beta\|_{G} \coloneqq \|\beta\|_{0,2} \vee \frac{\|\beta\|_0}{s_0}$.
We consider a variant of Birg\'{e}-Massart criterion \citep{birge2001gaussian} :
\begin{equation}\label{iht6}
	\tilde{t} = \arg\min_{t\in [T]\backslash[\bar t-1]}\left\{\frac{1}{n}\left\|y-X{\beta}^t\right\|_2^2 + \frac{1000{\sigma}_{\bar{t}}^2\Omega({\beta}^t)}{n}\right\},
\end{equation}
where
$
T \coloneqq \inf\{t: \lambda_t \le 4\frac{{\sigma}_{\bar{t}}}{\sqrt{n}}\}+1.
$
Here stopping time $T$ takes a value larger than $t_{\infty}$ to ensure a sufficiently large search domain.
Once the iterations hit the sub-optimal stopping time $\bar t$, we begin to select the optimal iteration according to \eqref{iht6}.
Now we are ready to present the detailed pseudocode of our adaptive proposed procedure in Algorithm \ref{alg:iht2}.

\begin{algorithm}[htbp]
\caption{\label{alg:iht2}\textbf{D}ouble \textbf{S}parse \textbf{IHT} (DSIHT) algorithm with known $s_0$}
  \begin{algorithmic}[1]
    \REQUIRE $X,\ y,\ \{G_j\}^m_{j=1},\ \kappa,\ s_0$.
    \STATE Initialize $t=0$, $\beta^0 = 0$ and $\lambda_0 = \frac{100}{9}\sqrt{\frac{\sigma_0^2}{n}(\log\frac{ed}{s_0}+\frac{1}{s_0}\log em)} \vee \frac{19}{4}\|M\|_{\infty}$.
    \WHILE {$\lambda_t \geq 8\sqrt{\frac{\sigma_t^2}{n}(\log\frac{ed}{s_0}+\frac{1}{s_0}\log em)},\ $}
    \STATE ${\beta}^{t+1} = \mathcal{T}_{\lambda_t, s_0}\left({\beta}^{t} + \frac{1}{n}X^{\top}(y-X{\beta}^{t})\right)$.
    \STATE $\lambda_{t+1} = \kappa\lambda_{t}$.
    \STATE $t = t+1$.
    \ENDWHILE
    \STATE Compute $\sigma^2_{\bar t} = \frac{1}{n}\|y-X\beta^t\|_2^2$.
    \WHILE {$\lambda_t \geq \frac{4 \sigma_{\bar t}}{\sqrt{n}},\ $}
    \STATE Compute $\text{C}_t = \frac{1}{n}\left\|y-X{\beta}^t\right\|_2^2 + \frac{1000{\sigma}_{\bar{t}}^2\Omega({\beta}^t)}{n}$.
    \STATE ${\beta}^{t+1} = \mathcal{T}_{\lambda_t, s_0}\left({\beta}^{t} + \frac{1}{n}X^{\top}(y-X{\beta}^{t})\right)$.
    \STATE $\lambda_{t+1} = \kappa \lambda_{t}$.
    \STATE $t = t+1$.
    \ENDWHILE
    \STATE $\tilde t = \mathop{\mathrm{argmin}}\limits_{t} \text{C}_t$.
    \ENSURE $\hat \beta = \beta^{\tilde{t}}$.
  \end{algorithmic}
\end{algorithm}

Algorithm \ref{alg:iht2} relies on the parameter $s_0$ and eliminates the dependence on the unknown values of $s$ and $\sigma$.
 The optimal results of stopping time $\tilde t$ are presented as follows.

\begin{theorem}\label{thm:optimal}
  Assume that $\beta^*$ is $(s,s_0)$-sparse and $X$ satisfies $\mbox{DSRIP}(5s, s_0, \delta )$. 
  Assume that $\delta<0.11\wedge \kappa^{10}$ and $n > 1000^2(ss_0\log \frac{ed}{s_0}+s\log em)$. Then, we have  $$
\sup\limits_{S^* \in \mathcal{S}^{m, d}(s, s_0)} P\left(\|\beta^{\tilde{t}} - \beta^*\|_2 \geq 150 \sqrt{\frac{\sigma^2}{n}\left(ss_0\log \frac{ed}{s_0}+s\log \frac{em}{s} \right)}\right)\leq e^{-C(ss_0\log \frac{ed}{s_0}+s\log \frac{em}{s})},
$$
and
$$
\sup\limits_{S^* \in \mathcal{S}^{m, d}(s, s_0)} P\left(\|\beta^{\tilde{t}}\|_G \geq 4s \right)\leq e^{-C(ss_0\log \frac{ed}{s_0}+s\log \frac{em}{s})}.
$$
\end{theorem}
Theorem \ref{thm:optimal} establishes the upper bound for the estimation error of $\beta^{\tilde{t}}$, indicating that $\beta^{\tilde{t}}$ adaptively achieves the minimax optimal rate of convergence. Moreover, Theorem \ref{thm:optimal} demonstrates that our procedure can guarantee the sparsity of the estimator $\beta^{\tilde{t}}$ with high probability. Specifically, we can control the model size $\|\beta^{\tilde{t}}\|_0$ within the order of $O(ss_0)$ and the selected number of groups $\|\beta^{\tilde{t}}\|_{0,2}$ within the order of $O(s)$.

\begin{corollary}\label{cor:iter_num}
  Assume that all conditions in Theorem \ref{thm:optimal} hold. For the stopping time $T$ in (\ref{iht6}), we have
  $$
  \sup\limits_{S^* \in \mathcal{S}^{m, d}(s, s_0)} P\left(T \geq \log\left(6(\frac{\sqrt{n}\|\beta^*\|_2}{\sigma} \vee \sqrt{\log ep})\right)/\log(1/{\kappa})+1 \right)\leq e^{-C(ss_0\log \frac{ed}{s_0}+s\log \frac{em}{s})}.
  $$
\end{corollary}
Corollary \ref{cor:iter_num} guarantees that our IHT procedure achieves optimal statistical accuracy with linear convergence with high probability, demonstrating the efficiency of our proposed method.

\subsection{Adaptive trade-off between IHT and group-IHT}\label{tradeoff}
In this section, we investigate the problem of misspecification of 
$s_0$, which is typically unobservable in real-world applications. Let $\bar{s}_0$ be the input parameter in Algorithm \ref{alg:iht2}. Notably, given the sample $(X, y)$ and step size $\kappa$, estimator $\hat{\beta}$ is solely determined by $\bar{s}_0$ in Algorithm \ref{alg:iht2}. Therefore, we introduce the following statistical measures derived from Algorithm \ref{alg:iht2} with the given $\bar{s}_0$:

\begin{itemize}
	\item $\hat\beta(\bar{s}_0)$ denotes the estimator of Algorithm \ref{alg:iht2} given $\bar{s}_0$.
	\item $\hat{s}(\bar{s}_0)$ denotes the selected number of groups of $\hat\beta(\bar{s}_0)$.
	\item $\A$ denotes the number of nonzero entries of $\hat\beta(\bar{s}_0)$. 
\end{itemize}
By the definition of $(s,s_0)$-sparsity and parameter space $\mathcal{S}^{m,d}(s,s_0)$, we establish the relationship
\begin{equation*}\label{misspecified}
    \mathcal{S}^{m,d}(s,s_0) \subseteq 
    \begin{cases}
    	  \mathcal{S}^{m,d}(ss_0/\bar{s}_0,\bar{s}_0), & \bar{s}_0 \le s_0.\\
    		  \mathcal{S}^{m,d}(s,\bar{s}_0), & \bar{s}_0 > s_0.
    \end{cases}
\end{equation*}	
On one hand, when $\bar{s}_0 > s_0$ in Algorithm \ref{alg:iht2}, the design matrix $X$ satisfies $\text{DSRIP}(5s, \bar{s}_0, \delta)$, and $\beta^*$ is $(s, \bar{s}_0)$-sparse. Algorithm \ref{alg:iht2} can obtain a minimax optimal estimator concerning parameter space $\mathcal{S}^{m,d}(s, \bar{s}_0)$, preserving all the previous theoretical results from Theorem \ref{th1} to Corollary \ref{cor:iter_num}. On the other hand, given $\bar{s}_0 \le s_0$ in Algorithm \ref{alg:iht2}, if the design matrix $X$ satisfies $\text{DSRIP}(5ss_0/\bar{s}_0, \bar{s}_0, \delta)$, and $\beta^*$ is $(ss_0/\bar{s}_0, \bar{s}_0)$-sparse, Algorithm \ref{alg:iht2} can obtain a minimax optimal estimator with respect to parameter space $\mathcal{S}^{m,d}(ss_0/\bar{s}_0, \bar{s}_0)$, preserving all the previous theoretical results. We summarize these results in Table \ref{table:1}:

\begin{table}[htb]   
	\begin{center}   
		\caption{Properties for $s_0$-mis-specified models.}  
		\label{table:1} 
		\renewcommand{\arraystretch}{2}
		\begin{tabular}{ m{1.8cm}<{\centering} m{3.3cm}<{\centering} c m{3.3cm}<{\centering} }  
        \toprule
	\textbf{Value} & \textbf{Parameter space} & \textbf{Minimax Rate}& \textbf{Support Control} \\[0.3cm]   
		\midrule   $\bar{s}_0 < s_0$ & $\mathcal{S}^{m,d}(ss_0/\bar{s}_0,\bar{s}_0)$ & $\sqrt{\frac{\sigma^2}{n}\left(ss_0\log\frac{ed}{\bar{s}_0 }+\frac{ss_0}{\bar{s}_0}\log\frac{em\bar{s}_0}{ss_0}\right)}$ & $\A \lesssim ss_0$ \\  [0.3cm] 
   	$\bar{s}_0 = s_0$ & $\mathcal{S}^{m,d}(s,s_0)$ & $\sqrt{\frac{\sigma^2}{n}\left(ss_0\log\frac{ed}{s_0}+ s\log\frac{em}{s}\right)}$ & \thead{$\A \lesssim ss_0$,\\ $\s \lesssim s$} \\[0.3cm] 
		$\bar{s}_0 > s_0$ & $\mathcal{S}^{m,d}(s,\bar{s}_0)$ & $\sqrt{\frac{\sigma^2}{n}\left(s\bar{s}_0\log\frac{ed}{ \bar{s}_0}+s\log\frac{em}{s}\right)}$ &  $\s \lesssim s$ \\      
			\bottomrule  
		\end{tabular}   
	\end{center}   
\end{table}
Table \ref{table:1} indicates that the theoretical properties differ significantly between the cases $\bar s_0 > s_0$ and $\bar s_0 < s_0$. When $\bar{s}_0 <s_0$, the upper bound for estimation error is given as $\sqrt{\frac{\sigma^2}{n}\left(ss_0\log\frac{ed}{\bar{s}_0 }+\frac{ss_0}{\bar{s}_0}\log\frac{em\bar{s}_0}{ss_0}\right)}$, and model size can be controlled within an order of $O(ss_0)$.
In the case of $\bar{s}_0 > s_0$, the upper bound is $\sqrt{\frac{\sigma^2}{n}\left(ss_0\log\frac{ed}{s_0}+ s\log\frac{em}{s}\right)}$, and the selected groups can be controlled within an order of $O(s)$. Notably, whether $\bar{s}_0 < s_0$ or $\bar{s}_0 > s_0$, simultaneous control of sparsity at both the element and group levels is unattainable.

We illustrate the minimax rate with varying values of $s_0$ from 1 to $d$ in Figure \ref{fig:rate}. As depicted in Figure \ref{fig:rate}, when $1 \leq \bar{s}_0 <s_0$, the minimax rate tends to be an inversely proportional function. On the other hand, when $s_0 < \bar{s}_0 \leq d$, the minimax rate exhibits a trend of near-linear growth. Notably, for $\bar{s}_0 = s_0$, the minimax rate attains the minimum among these values.
\begin{figure}[htbp]
  \begin{center}
    \begin{tikzpicture}
      \centering
      \draw[<->](11.5,0)--(0,0)--(0,4.5);
      \draw[red,domain=0.7:4] plot(\x, {ln(5/\x)+1/\x*ln(7*\x)-0.3}) node at (3.7,3.2){$\frac{\sigma^2}{n}(ss_0\log\frac{ed}{\bar{s}_0 }+\frac{ss_0}{\bar{s}_0}\log\frac{em\bar{s}_0}{ss_0})$};
      \draw[blue,domain=4:10.3] plot(\x,{\x*ln(35/\x)-7.9}) node at (8,1.8){$\frac{\sigma^2}{n}(s\bar{s}_0\log\frac{ed}{ \bar{s}_0}+s\log\frac{em}{s})$};      
      \node[below] at (11.5,0) {\large $\bar s_0$};
      \node[right] at (-2.7,3.5) {\large Minimax rate};
      \node[below] at (4,0) {\large $s_0$};
      \node[below] at (0.7,0) {\large 1};
      \node[below] at (10.3,0) {\large $d$};
      \draw[dashed](0.7, 0)--(0.7, 4);
      \draw[dashed](4, 0)--(4, 0.8);
      \draw[dashed](10.3, 0)--(10.3, 4.7);
    \end{tikzpicture}
    \caption{Minimax rate with metric $\|\cdot\|_2^2$ for different parameter spaces.}\label{fig:rate}
  \end{center}
\end{figure}

\begin{remark}

Regardless of the value of $\bar{s}_0$, the above results provide the upper bound for estimation error and properties of sparsity control for Algorithm \ref{alg:iht2}.
In particular, when $\bar{s}_0=1$, the DSIHT algorithm reduces to the classical IHT algorithm \citep{ndaoud2020scaled}, and the results in Table \ref{table:1} recover the minimax rate $O(\sqrt{\frac{\sigma^2}{n}ss_0\log \frac{ep}{ss_0}})$ \citep{raskutti2011minimax}.
When $\bar{s}_0=d$, the results in Table \ref{table:1} recover the minimax rate of group sparsity, namely, $O(\sqrt{\frac{\sigma^2}{n}(sd+s\log \frac{em}{s})})$ \citep{tsy2011}.
Therefore, DSIHT can be viewed as the trade-off between IHT \citep{ndaoud2020scaled} and group IHT \citep{giraud2021introduction} determined by the parameter $s_0$.
\end{remark}

\subsection{Data-adaptive tuning for unknown $s_0$}\label{adaptives0}

Previous sections have introduced an adaptive procedure to address cases with unknown $s$ and $\sigma$. In this section, we focus on constructing an adaptive estimator that achieves minimax optimality without prior knowledge of $s_0$, further demonstrating that our method (cf. Algorithm \ref{alg:iht3}) is a fully adaptive algorithm.

Given a sequence $\{s_{0,l}\}_{l=1}^L$, an intuitive approach to determine the optimal choice involves treating $s_0$ as a tuning parameter. This entails running the DSIHT algorithm along the sequence and employing a model selection criterion to identify the optimal model size. Here, we utilize a variant of the Birg\'{e}-Massart criterion introduced by \citet{verzelen2012minimax}. This variant implicitly incorporates the knowledge of $\sigma^2$, rather than plugging in a same-order estimator of $\sigma$ as demonstrated in criterion \eqref{iht6}. 
Motivated by this, we propose a novel double sparse information criterion (DSIC) as follows, with $\hat A(\bar{s}_0)$ and $\hat s(\bar s_0)$ defined at the beginning of section \ref{tradeoff}:

\begin{equation}\label{verzelen}
    \text{DSIC}(\bar{s}_0) = \log \left(\frac{\|y-X\hat\beta(\bar{s}_0)\|_2^2}{n} \right)+\frac K n \left(\hat A(\bar{s}_0) \log ed + \hat s(\bar{ s_0})\log \frac{em}{\hat s(\bar{s}_0)}\right),
\end{equation}
where $K$ is a positive constant.
The estimator $\hat\beta^{\bar s_0}$ minimizing \eqref{verzelen} is the optimal solution of our procedure.
The algorithm is summarized as follows:
\begin{algorithm}[H]
\caption{\label{alg:iht3}\textbf{A}daptive \textbf{D}ouble \textbf{S}parse \textbf{IHT} (ADSIHT) algorithm}
  \begin{algorithmic}[1]
    \REQUIRE $X,\ y,\ \{G_j\}^m_{j=1}, \ \kappa,\ \{s_{0,l}\}_{l=1}^L$.
    \FOR {$l=1,\ldots, L$,}
    \STATE $\hat \beta^l$ = Algorithm 2($X, y, \{G_j\}^m_{j=1}, \kappa, s_{0,l}$).
    \STATE Compute the double sparse information criterion $\text{DSIC}(s_{0,l})$.
    \ENDFOR
    \STATE $l^* = \mathop{\mathrm{argmin}}\limits_{l \in [L]} \{\text{DSIC}(s_{0,l}) \}$.
    \ENSURE $\hat \beta = \hat\beta^{l^*}$.
  \end{algorithmic}
\end{algorithm}

\begin{remark}
As discussed in Section \ref{tradeoff}, achieving optimal statistical performance necessitates that $\bar s_0$ is of the same order as $s_0$. Following the approach of \citet{bellec2018slope}, we set the candidate values of $s_0$ as an exponential sequence $\{s_{0,l}\}_{l=1}^L=\left\{2^{\frac{l-1}{2}}, 1 \leq l \leq L \right \}$, where $L \coloneqq \max \left\{l \in \mathbb{N}: 2^{\frac{l-1}{2}} \leq d \right\}$. This setting ensures that the candidate set includes a value of the same order as $s_0$. Recall that \citet{cai2019sparse} introduced candidate sets for the unknown parameters $s$ and $s_0$,
  and employed a grid search technique for their tuning. In contrast, Algorithm \ref{alg:iht3} requires only a candidate set for $s_0$ with $O(\log d)$ elements, making it a much more computationally efficient tuning approach.
\end{remark}

Before presenting our theoretical results, we require some assumptions on the sample size and design matrix. First, we assume that there exists an interval $\mathcal{S}_0 := [s_{0,\min},s_{0,\max}]$ such that $s_0 \in \mathcal{S}_0$. 
\begin{assumption}[Sample size assumption]\label{samplesize}
 We assume that the sample size $n$ satisfies 
    $n \gtrsim \left\{\left( ss_{0} \log ed + \frac{ss_0}{s_{0,\min}} \log em \right) 
\vee \Big( ss_{0,\max} \log ed + s \log em \Big)\right\}$.\end{assumption}
Assumption \ref{samplesize} is a necessary technical assumption for the minimax adaptation with an unknown noise level $\sigma$ \citep{verzelen2012minimax, 10.1214/12-STS398}.
In addition, we require the DSRIP condition to satisfy each element of $\mathcal{S}_0$.
\begin{assumption}[Adaptive DSRIP condition]\label{dsrip2}
    We assume that the design matrix $X$ satisfies both DSRIP$( 5s, s_{0,\max}, \delta)$ and DSRIP$(5ss_0/s_{0,\min},s_{0,\min}, \delta)$.
\end{assumption}

\begin{remark}
    In particular, when $s_{0,\min}$ is relatively small, especially for $s_{0,\min}=1$, we observe that DSRIP$(5ss_0/s_{0,\min},s_{0,\min}, \delta)$ reduces to the classical RIP condition \citep{candes2005}. Conversely, when $s_{0,\max} = d$, DSRIP$( 5s, s_{0,\max}, \delta)$ becomes the group RIP condition \citep{eldar2009robust}. 
\end{remark}

Now we give the minimax adaptive result in the following theorem:

\begin{theorem}\label{adaptive2}
Assume that $\beta^*$ is $(s,s_0)$-sparse.
  Given interval $\mathcal{S}_0$,   assume that Assumption \ref{samplesize} and \ref{dsrip2} hold and $\delta<0.11\wedge \kappa^{10}$. 
  Let
  $
\hat{s}_0 = \arg\min_{\bar{s}_0 \in \mathcal{S}_0} \text{DSIC}(\bar{s}_0)
  $ with a sufficiently large $K$.
Then, with probability greater than $1-\exp\big\{ -C_1(ss_0 \log(ed/s_0) + s \log(em/s)) \big\}$, we have
\begin{equation}\label{eq:so}
\left\|\hat{\beta}(\hat{s}_0) - \beta^*\right\|_2 \le C_2\sigma \sqrt{\frac{ss_0\log ed+s\log(em/s) }{n}}.
\end{equation}
\end{theorem}
Theorem \ref{adaptive2} shows that our adaptive procedure, i.e., Algorithm \ref{alg:iht3}, is an optimal fully adaptive procedure. Importantly, Algorithm \ref{alg:iht3} obtains the minimax adaptive solution without the knowledge of $s$, $s_0$ and $\sigma$.

\begin{remark}
The significance of adapting to $s_0$ lies in achieving an optimal trade-off between classical IHT \citep{ndaoud2020scaled} and group IHT \citep{giraud2021introduction}. If both Assumptions \ref{samplesize} and \ref{dsrip2} are satisfied, this optimal trade-off can be attained. It is important to emphasize that when $s_0$ is unknown, simultaneous control of element-wise sparsity and group-wise sparsity is unattainable. Consequently, we derive near-optimal estimation error bounds for our adaptive estimator. Further details are provided in the proof of Theorem \ref{adaptive2}.

\end{remark}

\section{Oracle estimation rate with beta-min condition}\label{oracle}
As is well-known, the ordinary least-squares (OLS) estimator supported on the true support set $S^*$ can achieve the oracle estimation rate of $O(\sigma\sqrt{\frac {ss_0}{n}})$. In this section, under the beta-min condition, we demonstrate that the DSIHT algorithm can also attain the oracle estimation rate. This implies that the estimator obtained by DSIHT performs as well as the oracle OLS estimator. Furthermore, DSIHT exhibits almost full recovery \citep{butucea2018variable} of the true support set $S^*$ under the beta-min condition.

Denote \begin{equation}\label{eq:lambdae}
    \tilde\lambda_a : = a  \sqrt{ \frac{8\sigma^2}{n} \left(\log \frac{ed}{s_0}+\frac{1}{s_0}\log \frac{em}{s} \right)  },~ a>0.
\end{equation}
Given an initial estimator $\tilde\beta^0$, we update the estimator by using a fixed threshold $\tilde \lambda_2$ in the DSIHT operator $\mathcal T_{\tilde \lambda_{2},s_0}$. In specific, we update the coefficient by
\begin{equation}\label{eq:fixiter}
\tilde \beta^{t+1} = \mathcal{T}_{\tilde \lambda_{2},s_0}\left({\tilde\beta}^{t} + \frac{1}{n}X^{\top}(y-X{\tilde\beta}^{t})\right).
\end{equation}
Denote $\tilde S^t$ as  the support set of $\tilde \beta^t$.
The following theorem investigates the theoretical guarantees of the iteration procedure with a fixed threshold.

\begin{theorem}\label{T6}
Assume $\min\limits_{i \in S^*}|\beta^*_{i}| \ge (\sqrt{2} + \epsilon)\tilde \lambda_2$ and $\min\limits_{j \in G^*}\|\beta^*_{G_j}\|_2 \ge (\sqrt{2} + \epsilon) \sqrt{s_0} \tilde \lambda_2$ for any constant $\epsilon > 0$.     Assume that $X$ satisfies $DSRIP(3s, \frac53 s_0,\delta)$ and $\delta \le \epsilon^4 \wedge 0.05$.
Let $\tilde \beta^0$ be an initial estimator satisfying \eqref{eq:iht3}-\eqref{eq:iht4} in Theorem \ref{th1}.
We run \eqref{eq:fixiter} and obtain the corresponding solution sequence $\{\tilde \beta^{t}\}$.
Then, for $\forall t \ge 0 $,
as $\min\left\{\log \frac{ed}{s_0}+\frac{1}{s_0}\log \frac{em}{s},~ \frac{ss_0}{\log \frac{ed}{s_0}+\frac{1}{s_0}\log \frac{em}{s}}\right\} \to \infty$, 
with probability tending to $1$\footnote{In specific, when $\Delta := \frac{1}{s_0}\log(em/s) + \log(ed/s_0)$ is sufficiently large, this probability is greater than $1-C_1 \exp \left(-C_2  {ss_0}/{\Delta} \right)- C_3 \Delta^2 \exp\left(-C_4 \Delta \right)$. And the tail probability of Theorem \ref{T7} is the same case.},
we have
    \begin{itemize}
        \item[(\romannumeral1)] $S_{G^*}^c \cap \tilde S^t \in \mathcal S^{m,d}(s,s_0) $.
        
        \item[(\romannumeral2)] $S_{G^*} \cap (S^*)^c \cap \tilde S^t \in \mathcal S^{m,d} (s,s_0 ) $.

        \item[(\romannumeral3)] The upper bound for estimation error satisfies
        \begin{equation}\label{scaledoracle}
            \left\|\tilde \beta^{t} - \beta^* \right\|_2
            <~ 16\left(\frac34 \right)^{t} \sqrt{\frac{\sigma^2}{n} \left(ss_0\log \frac{ed}{s_0}+s\log \frac{em}{s}\right)} +   
                16  \sqrt{ \frac{\sigma^2 ss_0}{n} }.
        \end{equation}
    \end{itemize}
\end{theorem}
The fixed iteration procedure preserves the results of false discoveries control, as shown in Theorem \ref{T6}. 
Specifically, under the beta-min conditions, result \eqref{scaledoracle} indicates that the upper bound for estimation error can be decomposed into two components: a diminishing optimization error $ 16\left(\frac34 \right)^{t} \sqrt{\frac{\sigma^2}{n} (ss_0\log \frac{ed}{s_0}+s\log \frac{em}{s})}$ that approaches zero as $t \rightarrow \infty$, and a statistical error $16 \sqrt{ \frac{\sigma^2ss_0}{n} }$. When the optimization error becomes smaller than the statistical error, the term $O\left(\sqrt{\frac{\sigma^2ss_0}{n}} \right)$ dominates the estimation error.

As a consequence of Theorem \ref{T6}, for a sufficiently large $t$, the estimator $\tilde \beta^t$ can achieve the oracle estimation rate and almost recover the true support set at both the element and group levels. To clarify this property, we denote the  element-wise decoder $\eta^* \in \{0, 1\}^p$ as $\eta^*_{i} = \mathrm I (\beta^*_{i} \ne 0)$, and the group-wise decoder $\eta^*_G \in \{0, 1\}^m$  as $(\eta^*_G)_j = \mathrm I( \beta^*_{G_j} \ne \mathbf 0)$. For $\tilde\beta^t$, denote $\tilde\eta^t \in \{0, 1\}^p$ as $\tilde\eta^t_{i} = \mathrm I (\tilde\beta^t_{i} \ne 0)$, and the group-wise decoder $\tilde\eta^t_G \in \{0, 1\}^m$  as $(\tilde\eta^t_G)_j = \mathrm I( \tilde\beta^t_{G_j} \ne \mathbf 0)$.

\begin{theorem}\label{T7}
    Assume that all the conditions in Theorem \ref{T6} hold. For $\forall t > 2 \log\big(256 (\log \frac{ed}{s_0}+\frac{1}{s_0}\log \frac{em}{s}) \big) $, as $\min\left\{\log \frac{ed}{s_0}+\frac{1}{s_0}\log \frac{em}{s},~ \frac{ss_0}{\log \frac{ed}{s_0}+\frac{1}{s_0}\log \frac{em}{s}}\right\} \to \infty$, with a probability tending to 1, we have:
    \begin{itemize}
        \item[(\romannumeral1)] The estimator $\tilde \beta^t$ satisfies
        \begin{equation} \label{T7:eq1}
            \left\|\tilde \beta^{t} - \beta^* \right\|_2
            \le ~ 17\sigma \sqrt{ \frac{ss_0}{n} }.
        \end{equation}

        \item[(\romannumeral2)] The estimator $\tilde \beta^{t}$ achieves group-wise almost full recovery, that is,
        \begin{equation}\label{T7:eq2}
            \| \tilde \eta_{G}^{t} - \eta^*_{G} \|_0 = o\left(s\right).
        \end{equation}
        
        \item[(\romannumeral3)] The estimator $\tilde \beta^{t}$ achieves element-wise almost full recovery, that is,
        \begin{equation} \label{T7:eq3}
            \| \tilde \eta^{t} - \eta^* \|_0 = o\left(ss_0\right).
        \end{equation}
        
    \end{itemize}
\end{theorem}
Theorem \ref{T7} affirms that, for a sufficiently large number of iterations, $\tilde \beta^t$ achieves the oracle estimation rate. 
Crucially, \citet{bellec2018noise} demonstrated that convex estimators cannot achieve the oracle estimation rate even when the beta-min conditions are satisfied. This highlights the superiority of our DSIHT algorithm over sparse group Lasso.
Moreover, the beta-min conditions also ensure almost full recovery \citep{butucea2018variable} at both the element-wise and group-wise levels.
Specifically, we can control both type-I and type-II errors within the order of $o(ss_0)$ and $o(s)$ at the element-wise and group-wise levels, respectively. 

Similar to the approach in Table \ref{table:1}, when $\bar{s}_0 > s_0 $ or $\bar{s}_0 < s_0 $, we can utilize alternative parametric spaces, i.e., $\mathcal{S}^{m,d}(s,\bar{s}_0)$ or $ \mathcal{S}^{m,d}(ss_0/\bar{s}_0,\bar{s}_0)$, and obtain the corresponding oracle estimation rates. These outcomes are illustrated in Table \ref{table:2}.

\begin{table}[htb]   
	\begin{center}   
		\caption{Properties for $s_0$-mis-specified  models under the beta-min conditions.}  
		\label{table:2} 
		\renewcommand{\arraystretch}{2}
		\begin{tabular}{ m{1.4cm}<{\centering} m{2.5cm}<{\centering} c m{2.2cm}<{\centering} m{2.5cm}<{\centering} }
        \toprule
        \textbf{Value} & \textbf{Parameter space} &  \textbf{ Order of $\tilde \lambda_2$}  &  \textbf{Oracle Estimation Rate}& \textbf{Almost Full Recovery} \\[0.3cm] \midrule  
		 $\bar{s}_0 < s_0$ & $\mathcal{S}^{m,d}\left( \frac{ss_0}{\bar{s}_0},\bar{s}_0 \right)$ 
            & $ \sqrt{ \frac{\sigma^2}{n} \left( \log \frac{ed}{\bar{s}_0}+\frac{1}{\bar{s}_0}\log \frac{em\bar{s}_0}{ss_0} \right)  } $ & $\sqrt{\frac{\sigma^2}{n} ss_0}$ &  
            element-wise \\  [0.3cm]
   		    $\bar{s}_0 = s_0$ & $\mathcal{S}^{m,d}(s,s_0)$ & $  \sqrt{ \frac{\sigma^2}{n} \left(  \log \frac{ed}{ {s}_0} +\frac{1}{{s}_0}\log \frac{em}{s}\right)  } $ & $\sqrt{\frac{\sigma^2}{n} ss_0}$ &  
             element-wise and group-wise \\[0.3cm] 
		  $\bar{s}_0 > s_0$ & $\mathcal{S}^{m,d}(s,\bar{s}_0)$ &
   $ \sqrt{ \frac{\sigma^2}{n} \left(  \log \frac{ed}{\bar{s}_0} +\frac{1}{\bar{s}_0}\log \frac{em }{s }\right)  } $ &$\sqrt{\frac{\sigma^2}{n} s\bar{s}_0}$ & group-wise \\[0.3cm]      
			\bottomrule  
		\end{tabular}   
	\end{center}   
\end{table}
Table \ref{table:2} reveals that when $\bar s_0 \leq s_0$, the oracle estimation rate is $\sqrt{\frac{\sigma^2}{n} ss_0}$, showing insensitivity to the variations in $\bar s_0$. Conversely, for $\bar s_0 > s_0$, the oracle estimation rate increases to $\sqrt{\frac{\sigma^2}{n} s\bar s_0}$, further emphasizing the role of $s_0$ as a trade-off between IHT and group IHT as discussed in Section \ref{tradeoff}.

\begin{remark}

When $\bar s_0 = 1$, our results align with the assumptions and findings of element-wise IHT \citep{ndaoud2020scaled}. While both IHT and DSIHT attain the oracle estimation rate $O(\sqrt{\frac{\sigma^2}{n} ss_0})$ for $(s, s_0)$-sparse vectors with the beta-min conditions, Theorem \ref{T7} demonstrates that DSIHT not only achieves almost full recovery at the element level, as indicated by \eqref{T7:eq3}, but also at the group level, as indicated by \eqref{T7:eq2}. This underscores the superiority of DSIHT over IHT.
\end{remark}

\section{Numerical experiments}\label{numerical}
In this section, we present numerical experiments that shed light on the empirical performances of our proposals using both synthetic and real-world data sets.
Our algorithms are implemented in R package $\mathtt{ADSIHT}$. 
We compare against several state-of-the-art methods: sparse group Lasso (SGLasso, \citet{simon2013sparse}), which is fitted by R package $\mathtt{sparsegl}$ \citep{liang2022sparsegl}, group bridge (GBridge, \citet{huang2009group}), group exponential Lasso (GEL, \citet{breheny2015group}) and composite minimax concave penalty (CMCP, \citet{breheny2009penalized}), which are computed by R package $\mathtt{grpreg}$ \citep{breheny2015group}.
For SGLasso, we determine the tuning parameter by five-fold cross-validation. 
For the other comparison methods, we select the optimal solution using EBIC \citep{chen2008extended}.
For ADSIHT, we use our proposed DSIC with $K = 5$ to select the optimal model.
Moreover, we leave the remaining hyper-parameters to their default values in $\mathtt{sparsegl}$ and $\mathtt{grpreg}$.
All numerical experiments are conducted in R and executed on a personal laptop (AMD Ryzen 9 5900HX, 3.30 GHz, 16.00GB of RAM).

\subsection{Analysis on Synthetic Data}
Synthetic data sets are generated from the underlying model $y = X \beta^*+\xi$, where $\beta^* \in \mathbb{R}^p$ has $m$ groups with equal group size, namely, $p_1= \cdots =p_m=d$.
The design matrix $X$ is generated from a multivariate Gaussian distribution $\mathcal{MVN}(0, \Sigma)$.
The covariance matrix $\Sigma$ is considered as the auto-regressive structure, that is, $\Sigma_{ij} = 0.5^{|i-j|}$ for $1 \leq i,j\leq p$.
Next, the coefficients $\beta^*$ are generated under the following two scenarios:
\begin{itemize}
  \item Homogeneous signal: $\beta^*$ is randomly chosen from $\{1, -1\}$.
  \item Heterogeneous signal: $\beta^*$ is randomly chosen from $\mathcal{N}(0, 1)$.
\end{itemize}

Finally, the random error $\xi_i$ is generated independently from $N(0, \sigma^2)$, and $\sigma$ is chosen to achieve a desired signal-to-noise ratio (SNR).
All simulation results are based on 100 repetitions.
Given an output $(\hat S, \hat \beta)$,  
we use the following measures to assess the accuracy of variable selection and parameter estimation:
\begin{itemize}
  \item \textbf{Sparsity Error (SE)}: $|\hat S| - |S^*|$.
  \item \textbf{Group-wise Sparsity Error (GSE)}: $\|\hat\beta\|_{0,2} - \|\beta^*\|_{0,2}$.
  \item \textbf{Mathew's Correlation Coefficient (MCC)}:
  \begin{align*}
    \text{MCC} = \frac{\text{TP}\times \text{TN}-\text{FP}\times \text{FN}}{\sqrt{(\text{TP+FP)(TP+FN)(TN+FP)(TN+FN)}}},
  \end{align*}
  where TP= $\hat S \cap S^*$ and TN= $\hat S^c \cap (S^*)^c$ stand for true positives/negatives, respectively. FP= $\hat S \cap (S^*)^c$ and FN= $\hat S^c \cap S^*$ stand for false positives/negatives, respectively.
  \item \textbf{Estimation Error (EE)}: $\|\hat \beta - \beta^*\|_2$.
\end{itemize}
Here SE or GSE close to zero means better estimation results on the support set.
MCC ranges in $[-1, 1]$, and a larger MCC means a better variable selection performance.

\subsubsection{Statistical performance for varying SNR}
In this section, we study the effect of varying the SNR of model on the performance of ADSIHT and other state-of-the-art methods. 
We consider the generating model contains 50 nonzero coefficients, distributed evenly into 10 groups. We set sample size $n = 300$, group size $d = 10$, number of group $m=100$. The SNR increases from  to 20 with an increment equal to 2.
Figure \ref{plot1} shows the computational results of the homogeneous scenario and heterogeneous scenario in sub-figure {\bf{A}} and {\bf{B}}, respectively.

\begin{figure}[htbp]
  \centering
  \includegraphics[scale = 0.65]{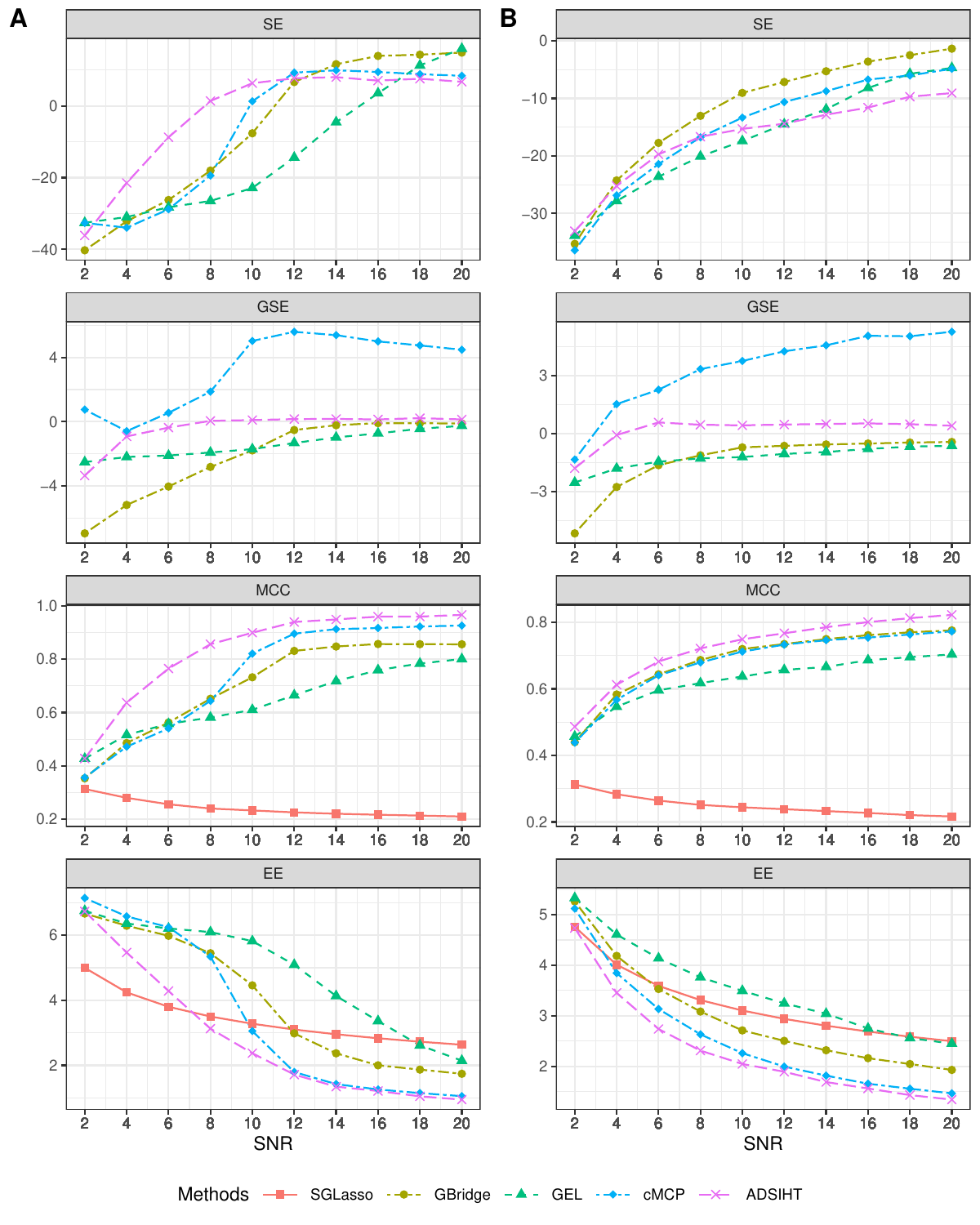}
  \caption{\label{plot1}Performance measures as the signal-to-noise ratio (SNR) increases from 1 to 10.
  (A) Computational results with homogeneous signal. 
  (B) Computational results with heterogeneous signal. 
}
\end{figure}

Figure \ref{plot1} shows that with the increase of SNR, all methods tend to perform better.
Our method exhibits excellent performances in terms of all measures across the whole SNR range.
For the homogeneous signal setup, our method is able to achieve full support recovery for high SNR.
On the other hand, although none of the considered methods can identify all the true variables accurately even for high SNR,
our method still shows its superiority in terms of variable selection and parameter estimation.

\subsubsection{Statistical performance for varying number of groups}
Here we study how the statistical metrics change with the number of groups.
We consider the generating model contains 50 nonzero coefficients, distributed evenly into 10 groups.
We set sample size $n=500$, group size $d = 10$ and SNR = 5.
The number of groups increases from 50 to 500 with an increment equal to 50.
We show the results in figure \ref{plot2}.

\begin{figure}[htbp]
  \centering
  \includegraphics[scale = 0.65]{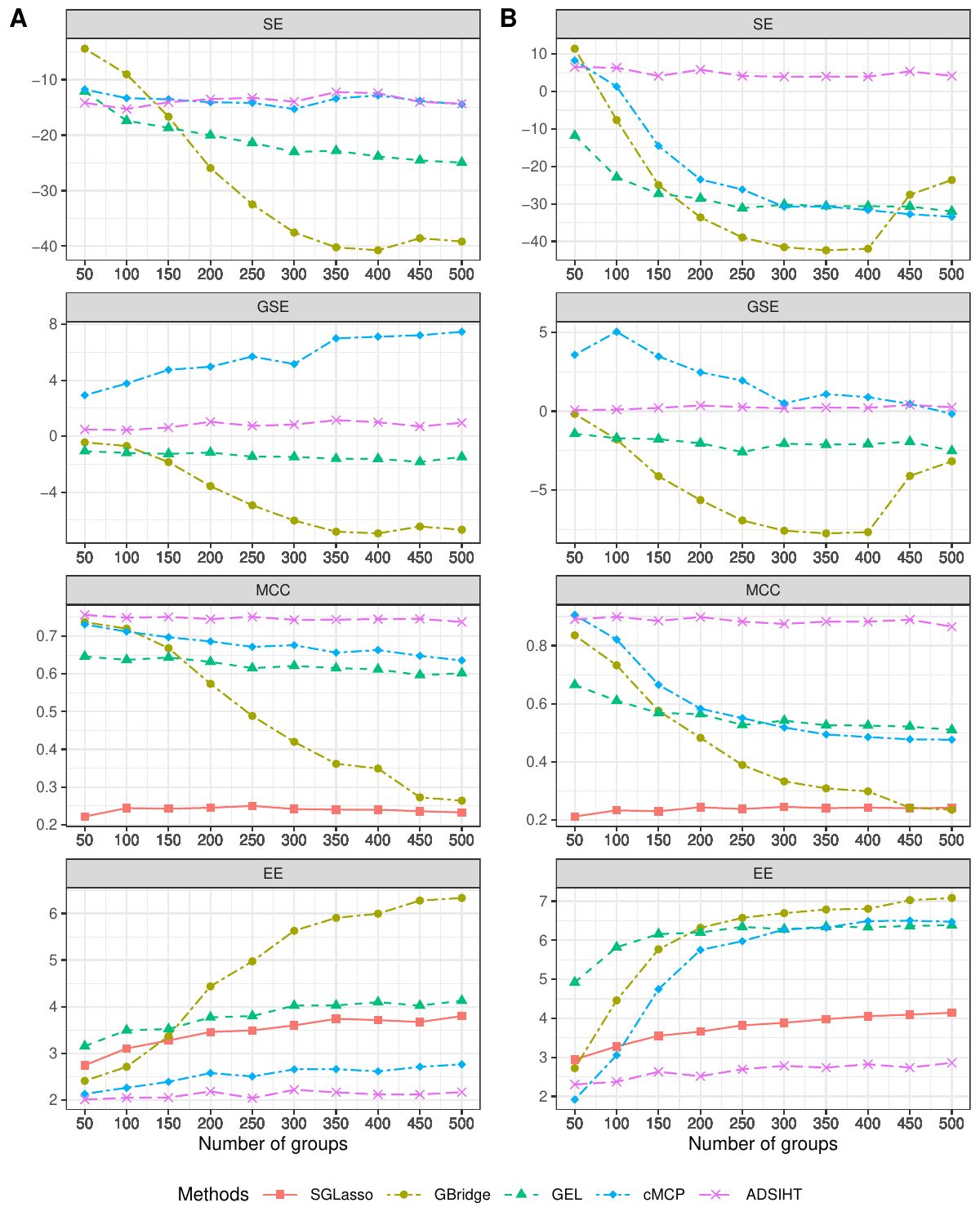}
  \caption{\label{plot2}Performance measures as the number of groups increases from 100 to 1200.
  (A) Computational results with homogeneous signal. 
  (B) Computational results with heterogeneous signal. 
}
\end{figure}

From Figure \ref{plot2}, we see that our method is more robust in the high-dimensional settings.
In terms of variable selection and parameter estimation, our method appears to outperform the other considered methods,
with the differences being most  pronounced in the high-dimensional settings.
As the number of groups increases, the performances of other methods, especially for GBridge, decrease significantly.

\subsubsection{Statistical performance for varying sample size}
Here we investigate the effect of varying the sample size on the performances while keeping the other parameters fixed. 
We consider the generating model contains 50 nonzero coefficients, distributed evenly into 5 groups.
We set group size $d = 20$, number of group $m=200$ and SNR$ = 5$.
The sample size increases from 300 to 1000 with an increment equal to 100.

\begin{figure}[htbp]
  \centering
  \includegraphics[scale = 0.65]{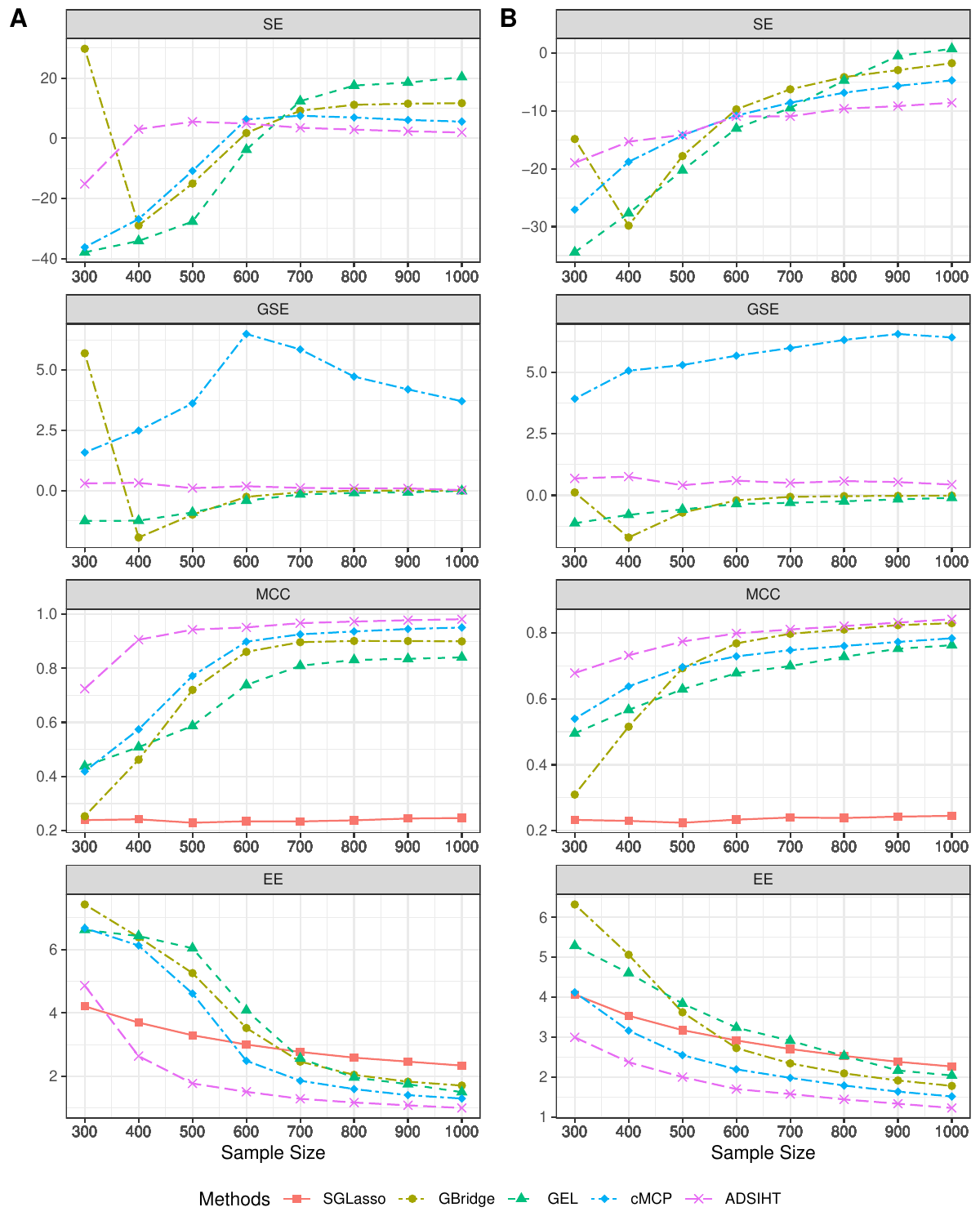}
  \caption{\label{plot3}Performance measures as the sample size increases from 300 to 1000.
  (A) Computational results with homogeneous signal. 
  (B) Computational results with heterogeneous signal. 
}
\end{figure}

As shown in Figure \ref{plot3}, the performances of all methods improve significantly as the sample size increases.
Our method notably outperforms the other methods across different statistical metrics.
For the homogeneous signal setup, our method perfectly recovers the support set when the sample size exceeds 800.
In comparison, other methods cannot achieve full support recovery even for a sufficiently large sample size.
In particular, for both setups of signals, our method can estimate the coefficients accurately, which aligns with the minimax optimality of our method in the sense of parameter estimation.

\subsection{Analysis on Real-world Data}
The TRIM32 dataset, which pertains to the Bardet-Biedl syndrome gene expression, was initially presented by \citet{s2006} and has been extensively studied in various statistical works \citep{huang2010, fan2011, zhang2022}.
In this study, 120 twelve-week-old male rats were gathered for tissue harvesting from the eyes and for micro-array analysis.
For this data set, TRIM32, a gene that has been associated with causing Bardet-Biedl syndrome \citep{chiang2006homozygosity}, serves as the response variable, while the remaining 18,975 gene probes that have the potential to impact TRIM32 expression are treated as covariates.

In this paper, we aim to identify the genes which are statistically significantly related to gene TRIM32 and build an accurate prediction model.
Of the 18,975 probes, the top 300 probes with the highest marginal ball correlation \citep{pan2019} are considered. Then, for each gene, we utilize a ten-term natural cubic spline basis expansion to form a group with 10 variables. This technique, which is commonly employed in scientific research \citep{huang2010, B2015, zhang2022}, allows us to analyze the data more effectively. After performing the aforementioned operations, this problem can be described as a high-dimensional variable selection problem with $n=120$, $m=300$, and $d = 10$.
In our analysis, the 120 rats are randomly split into a training set with 100 samples and
a test set with the remaining 20 samples. We repeat these random splitting procedures 200 times and
compute the average of the numbers of selected variables and groups and the prediction mean square error (PMSE) in the
test set. The computational results and the box plot of the PMSE are shown in Table \ref{tab2} and Figure \ref{fig5}, respectively.
\begin{table}[htbp]
\caption{\label{tab2}Computational results for TRIM32 dataset. The standard deviations are shown in parentheses.}
  \centering
  \setlength{\tabcolsep}{4mm}{
  \begin{tabular}{*{4}c}
    \toprule
    Method & Number of variables & Number of groups & 100$\times$PMSE   \\
    \midrule
    SGLasso & 139.26  (68.74) & 26.95  (12.45) &  1.71 (1.84) \\
    GBridge & 2.95 (0.81) & 1.04 (0.18) & 2.01 (2.00)\\
    GEL & 35.78  (28.03) & 7.07  (3.69) & 2.55 (2.70)\\
    cMCP & 21.60  (3.37) & 20.95 (3.08) & 1.92 (2.12) \\
    ADSIHT & 29.06  (11.93) & 9.20  (4.09) & \bf{1.70 (1.85)}\\
    \bottomrule
  \end{tabular}}
\end{table}

Table \ref{tab2} demonstrates that SGLasso identifies significantly more variables and groups than other methods. However, this does not lead to the best prediction performance on the test set. 
On the other hand, our proposed method delivers the highest statistical accuracy in predicting outcomes, despite using fewer variables and groups.
Furthermore, Figure \ref{fig5} illustrates that our approach is both accurate and robust in its predictive performance, demonstrating the superiority of our method over other methods.

\begin{figure}[H]
  \centering
  \includegraphics[scale = 0.5]{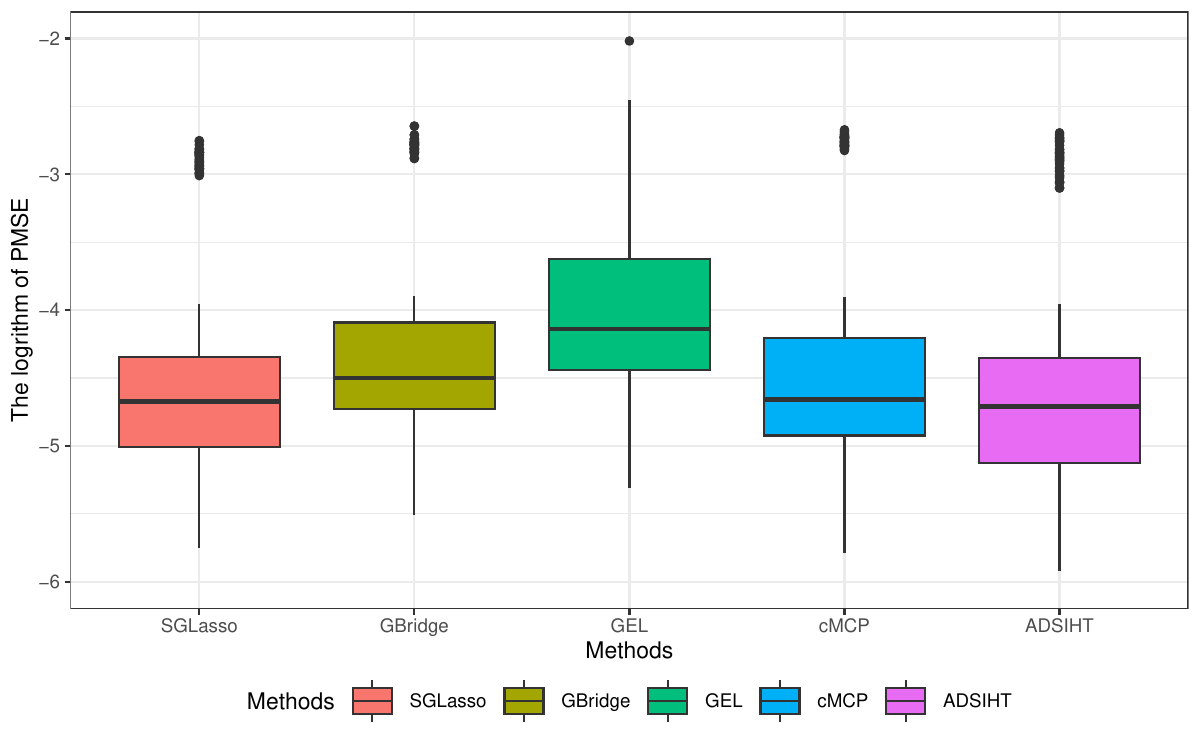}
  \caption{\label{fig5} Boxplot of the PMSE.}
\end{figure}

To perform further investigation, we consider the entire set of 120 samples to learn a double sparse linear model for TRIM32 expression. Figure \ref{fig:qq} displays QQ-plots of the residuals estimated from our proposed method and comparative methods. The sub-figures of cMCP and ADSIHT have points that roughly lie on the diagonal line, which indicates the satisfaction of the normality assumption.
In contrast, Figure \ref{fig:qq} reveals that the residual distributions of SGLasso, GBridge, and GEL have longer tails on the left side, which implies that analyzing this dataset using the fitted linear models may be unconvincing.
\begin{figure}[htbp]
  \centering
  \includegraphics[scale = 0.65]{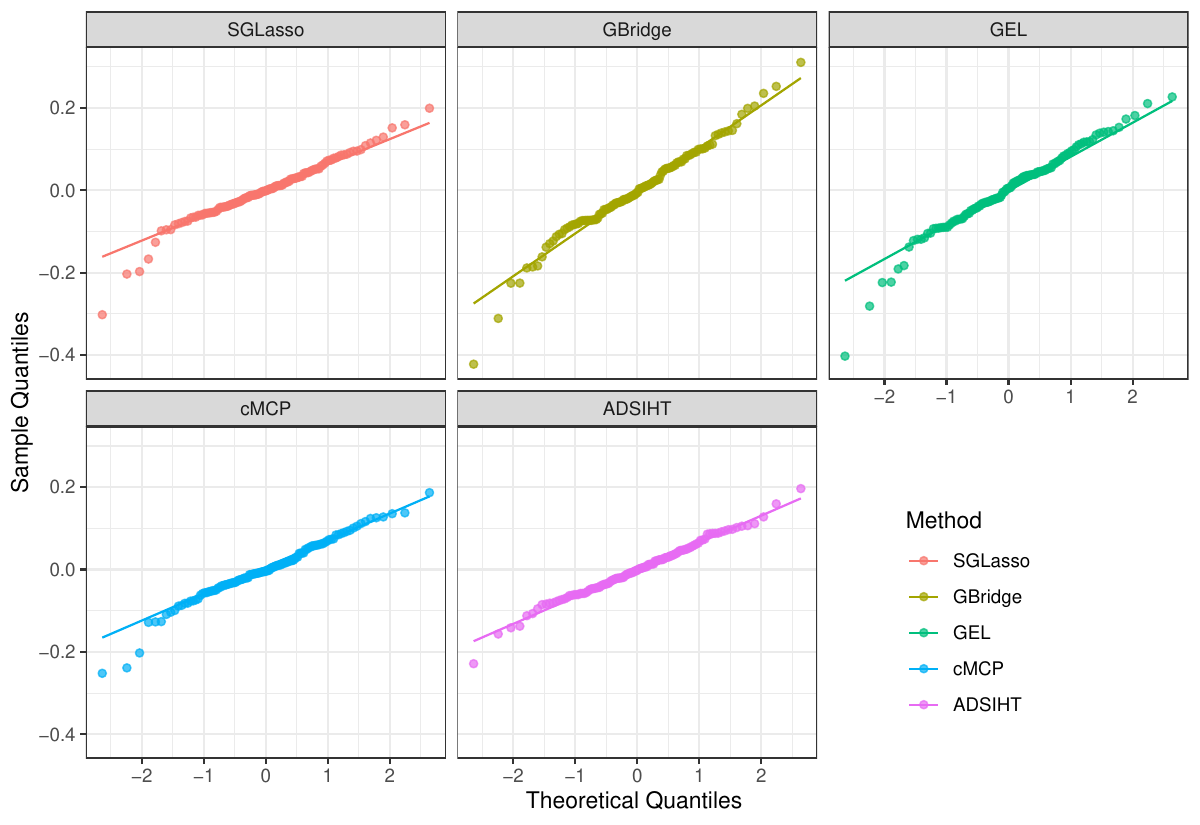}
  \caption{\label{fig:qq}QQ-plots of the residuals.}
\end{figure}
Moreover, we calculate the $\text{R}^2$ and adjusted $\text{R}^2$ for each method, as outlined in Table \ref{tab3}. The computational results in Table \ref{tab3} demonstrate the favorable fitting performance of ADSIHT. Specifically, ADSIHT effectively identifies 14 important groups and 31 significant variables within these groups, collectively explaining 79\% of the variance in TRIM32 expression. While SGLasso achieves the highest variance explanation in TRIM32 expression, there is a potential concern of overfitting, as it selects an excessively large model.
\begin{table}[htbp]
  \caption{\label{tab3}The $\text{R}^2$ and adjusted $\text{R}^2$ for each method. Adjusted $\text{R}^2$ is omitted for SGLasso due to the excessively large model size selected by SGLasso.}
  \centering
  \setlength{\tabcolsep}{4mm}{
  \begin{tabular}{*{6}c}
    \toprule
     & SGLasso &GBridge & GEL & cMCP & ADSIHT   \\
    \midrule
    $\text{R}^2$ & 88\% & 44\% & 54\% & 75\% & 79\%\\
      Adjusted  $\text{R}^2$ & $\times$ & 42\% & 52\% & 68\% & 71\%\\
    \bottomrule
  \end{tabular}}
\end{table}

\section{Conclusion}\label{conclusion}
In our work, we propose a minimax optimal IHT-style procedure for high-dimensional double sparse linear regression. In specific, we introduce a novel double sparse iterative hard thresholding (DSIHT) operator. To effectively control false discoveries, we iteratively decrease the threshold in the DSIHT operator until it reaches the optimal threshold. Under certain conditions, we prove that our DSIHT algorithm obtains a minimax optimal estimator.

Notably, for the $(s, s_0)$-sparse structure, we devise a fully adaptive optimal procedure that enables our algorithm to derive a minimax optimal estimator with unknown sparsity levels $s, s_0$, and variance $\sigma^2$. Initially, given $s_0$, we introduce an adaptive procedure that determines the optimal stopping time using a variant of the Birg$\mathrm{\acute{e}}$-Massart criterion, which is independent of $s$ and $\sigma$.
Importantly, we highlight the role of sparsity level $s_0$ as the trade-off between IHT and group IHT.
Building on this result, we propose a novel double sparse information criterion to select the optimal $s_0$, making our method a fully adaptive procedure. In theory, we demonstrate that our two-step adaptive procedure achieves optimal statistical accuracy with fast convergence. More importantly, to illustrate why our algorithm outperforms sparse group Lasso, we prove that under the beta-min conditions, our algorithm can attain the oracle estimation rate, which is unachievable for convex estimators, and achieve almost full recovery of the true support set. Finally, numerical experiments show that our methods exhibit more accurate and robust statistical performance than other state-of-the-art methods.

In this paper, we consider the double sparse structure in linear regression, and a similar approach can be explored in generalized linear models or single-index models. Moreover, our technical results can be applicable to various other problems with simultaneous sparsity structures, such as sparse additive models \citep{raskutti2012minimax,yuan2016minimax} and high-dimensional change point problems \citep{liu2021minimax}. We identify these avenues as potential future lines of research.




\newpage

\appendix
\section*{Appendix}The Appendix contains the technical proofs of all Theorems and Corollaries. The proofs of the main results are presented in Appendix A. Appendix B contains the proofs of the auxiliary lemmas. Appendix C provides an example of DSRIP condition under sub-Gaussian random design.
To simplify the notations of the appendix, we denote
\begin{align*}
\D \coloneqq \log \frac{ed}{s_0}+\frac{1}{s_0}\log \frac{em}{s}\quad \text{and}\quad \D' \coloneqq  \log \frac{ed}{s_0} + \frac{1}{s_0} \log em.
\end{align*}
Given a $p$-dimensional vector $\beta$ with $\|\beta\|_0 = \hat A$ and $\|\beta\|_{0,2}= \hat s$, denote
 $$\Omega^*(\beta):= (s+ \hat{s})\log \frac{em}{s+ \hat{s}}+(ss_0+ \hat{A})\log\frac{ed(s+ \hat{s})}{s s_0 +\hat{A}}.$$
\section*{Appendix A : Proofs of main results}

\subsection* {Proof of Lemma \ref{lemma:iht1}}

  From Theorem 2.1 of \citet{hsu2012tail}, $\forall S \in \mathcal{S}$, we have
  \begin{equation}\label{lem1eq1}
  P\left(\frac{\|X_S^{\top}\xi\|_2^2}{\sigma^2} \ge Tr(X_S^{\top} X_S)+ 2\|X_S^{\top} X_S\|_F\sqrt{t}+2\lambda_{\max}(X_S^{\top} X_S) t\right)\le e^{-t},
  \end{equation}
  where constant $t\geq0$.
  Since $X_S \in \mathbb{R}^{n \times ss_0}$ and $\|X_j\|_2 = \sqrt{n}, j\in [p]$, we have 
  \begin{equation}\label{lem1eq2}
	Tr(X_S^{\top} X_S) = \left(\sum_{j=1}^{ss_0}\sum_{i=1}^n X_{ij}^2\right) \le ss_0 n.
\end{equation}
On one hand, we have
\begin{equation}\label{lem1eq4}
\lambda_{\max}(X_S^{\top} X_S) \le n(1+\delta) .
\end{equation}
On the other hand, from \eqref{lem1eq2} and \eqref{lem1eq4}, we have
\begin{equation}\label{lem1eq3}
\|X_S^{\top} X_S\|_F = \sqrt{Tr(X_S^{\top} X_SX_S^{\top} X_S)} \le (1+\delta )\sqrt{ss_0} n,
\end{equation}
Substituting \eqref{lem1eq2} - \eqref{lem1eq3} into \eqref{lem1eq1}, we have
\begin{equation*}
P\left(\frac{1}{n\sigma^2}\|X_S^\top\xi\|_2^2 \ge 2(1+\delta)\left[\sqrt{t} +\frac{\sqrt{ss_0}}{2}\right]^2+ \frac{1-\delta}{2}ss_0\right)\le e^{-t}.
\end{equation*}
Note that $\delta < 1$ and $\D \gg 1$. For some positive constant $C$, let $t=(1+C) ss_0\D$, and we have
$2(1+\delta)\left[\sqrt{t} +\frac{\sqrt{ss_0}}{2}\right]^2+ \frac{1-\delta}{2}ss_0 < 4 ss_0\D.$
Consequently, we have 
\begin{equation}\label{lem1eq5}
  P\left(\frac{1}{n}\|X_S^\top\xi\|_2^2 \ge 4\sigma^2 ss_0\D\right)\le e^{-(1+C) ss_0\D}.
  \end{equation}
Note that 
\begin{equation}\label{lem1eq6}
  |\mathcal{S}^{m,d}(s,s_0)| \le {m \choose s}\times{sd \choose ss_0}\leq (\frac{em}{s})^s\times (\frac{ed}{s_0})^{ss_0}\leq e^{ss_0\Delta}.
\end{equation}
Therefore, combining \eqref{lem1eq5} and \eqref{lem1eq6}, we have
\begin{align*}
  &P\left(\forall {S \in \mathcal{S}^{m, d}(s,s_0)},\sum_{i \in S}\Xi_{i}^2 \le \frac{4\sigma^2 ss_0\D}{n}\right) \\
  =&1-P\left(\exists {S \in \mathcal{S}^{m, d}(s,s_0)},\sum_{i \in S}\Xi_{i}^2 > \frac{4\sigma^2 ss_0\D}{n}\right)\\
  \geq&1-|\mathcal{S}^{m,d}(s,s_0)|P\left(\sum_{i \in S}\Xi_{i}^2 > \frac{4\sigma^2 ss_0\D}{n}\right)\\
  \geq& 1-e^{-C ss_0\D},
\end{align*}
where the first inequality follows from the union bound. This completes the proof of Lemma \ref{lemma:iht1}.

\subsection* {Proof of Theorem \ref{th1}}

We proceed with the proof of Theorem \ref{th1} under the assumption that event $\mathcal{E}$ holds. Initially, it's straightforward to confirm that the results are trivial for $t = 0$. 
  Then, we assume that the results are true for step $t$, and prove them for step $t+1$.

 We first prove \eqref{eq:iht3} and \eqref{eq:iht5} by contradiction. 
  Assume that \eqref{eq:iht3} and \eqref{eq:iht5} are wrong for $t+1$, i.e., $S_{G^*}\cap S^{t+1} \cap (S^*)^c \notin \mathcal{S}^{m, d}(s, s_0)$ and $S_{G^*}^c\cap S^{t+1} \notin \mathcal{S}^{m, d}(s, s_0)$. 
  
\subsubsection*{Step 1}

For result \eqref{eq:iht3}, note that $S_{G^*}\cap(S^*)^c$ covers no more than $s$ groups. According to the {\bf{Case 1}} in Section \ref{analysis}, it holds that there exists a $(s, s_0)$-shaped subset $\tilde{S}_{1,t+1}$ of $S_{G^*}\cap(S^*)^c$
  with cardinality $ss_0$
  such that
  $$
  ss_0 \lambda_{t+1}^2 \le \sum_{i \in \tilde{S}_{1,t+1}}\left\{\mathcal{T}_{\lambda_{t+1,s_0}}(H^{t+1})\right\}_{i}^2.
  $$
  Note that $\beta^*_i = 0$ for $i \in \tilde{S}_{1,t+1} \subseteq(S^*)^c$. 
Then, using equation \eqref{eq:H} and the triangle inequality, we obtain
  $$
  \sqrt{ss_0}\lambda_{t+1} \le \sqrt{\sum_{i\in \tilde{S}_{1,t+1}}\langle \Phi_{i}^\top, \beta^{t}-\beta^*\rangle^2}+\sqrt{\sum_{i\in \tilde{S}_{1,t+1}}\Xi_{i}^2}.
  $$
  Recall that $\beta^*$ is $(s,s_0)$-sparse, and both \eqref{eq:iht3} and \eqref{eq:iht5} hold for $t$ by assumption. Then, we have $\beta^{t}-\beta^*$ is $(2s, \frac{3}{2}s_0)$-sparse. 
  Consequently, using the DSRIP condition and Lemma \ref{lemma:iht1}, we have
  \begin{align*}
    \sqrt{ss_0}\lambda_{t+1} 
\le&  \delta \|\beta^* - \beta^t\|_2 + 2\sigma\sqrt{\frac{ss_0\D}{n}}\\
\le&  \frac{3}{2}(1+\sqrt{2})\delta \sqrt{ss_0}\lambda_t + 2\sigma\sqrt{\frac{ss_0\D}{n}}\\
\le&  \frac{3}{2}(1+\sqrt{2})\delta^{\frac{9}{10}} \sqrt{ ss_0}\lambda_{t+1} + \frac{1}{2}\sqrt{ss_0}\lambda_{\infty}\\
\le& \left(\frac{1}{2}+ \frac{3}{2}(1+\sqrt{2})\delta^{\frac{9}{10}}\right)\sqrt{ss_0}\lambda_{t+1}\\\
<& \sqrt{ss_0}\lambda_{t+1},
  \end{align*}
  which leads to a contradiction. 
  Since we have assumed that \eqref{eq:iht4} holds for $t$, the second inequality holds based on it, and the last inequality follows from $\delta < 0.11 \wedge \kappa^{10}$.
  Therefore, we have $S_{G^*}\cap S^{t+1} \cap (S^*)^c \in \mathcal{S}^{m, d}(s, s_0)$, indicating that \eqref{eq:iht3} holds for $t+1$.
  
\subsubsection*{Step 2} 

For result \eqref{eq:iht5}, if $S_{G^*}^c\cap S^{t+1}$ covers no more than $s$ groups, 
  the analysis of result \eqref{eq:iht5} is the same as {\bf{Step 1}}. Otherwise, according to {\bf{Case 2}} in Section \ref{analysis}, there exists a $(s, s_0)$-shaped subset $\tilde{S}_{2,t+1}$ of $S_{G^*}^c$ such that
  $$
  ss_0 \lambda_{t+1}^2 \le \sum_{i \in \tilde{S}_{2,t+1}}\left\{\mathcal{T}_{\lambda_{t+1,s_0}}(H^{t+1})\right\}_{i}^2.
  $$
  The remaining proof of \eqref{eq:iht5} is similar to {\bf{Step 1}}.
  Therefore, \eqref{eq:iht5} holds for $t+1$.
 
\subsubsection*{Step 3} 

We now turn to the proof of \eqref{eq:iht4}. 
  Note that results \eqref{eq:iht3} and \eqref{eq:iht5} hold for $t+1$,
  which imply that $\beta^{t+1} - \beta^*$ is $(2s, \frac{3}{2}s_0)$-sparse. 
  Observe that for any $i \in [p]$, 
  \begin{equation}\label{eq:iht10}
  \beta^{t+1}_i - \beta^*_i = -H^{t+1}_i \mathrm{I}(i \notin S^{t+1}) + \langle \Phi_{i}^\top, \beta^*-\beta^{t}\rangle + \Xi_i.
\end{equation}
  On one hand, summing both sides of \eqref{eq:iht10} over set $S^{t+1}\cap (S^*)^c$, we have 
  \begin{align}\label{eq:iht8}
    \begin{split}
    \| \beta^{t+1}_{(S^*)^c}\|_2 \le& \sqrt{\sum_{i\in S^{t+1}\cap (S^*)^c}\langle \Phi_{i}^\top, \beta^*-\beta^{t}\rangle^2}+\sqrt{\sum_{i\in S^{t+1}\cap (S^*)^c}\Xi_{i}^2}\\
    \le& \delta \| \beta^*-\beta^{t} \|_2 + 2\sigma\sqrt{\frac{2ss_0\D}{n}},
  \end{split}
  \end{align}
  where the right-hand side of the second inequality comes from the accumulation of two parts of random errors corresponding to \eqref{eq:iht3} and \eqref{eq:iht5}.
  On the other hand, summing both sides of \eqref{eq:iht10} over support set $S^*$, we have 
  \begin{equation}\label{eq:iht9}
    \begin{aligned}
      \| (\beta^{t+1}-\beta^*)_{S^*} \|_2 \le& \sqrt{\sum_{i\in S^*}(H^{t+1}_i)^2\mathrm{I}(i \notin S^{t+1})}  +\sqrt{\sum_{i\in S^*}\langle \Phi_i^\top, \beta^*-\beta^{t}\rangle^2}+ \sqrt{\sum_{i\in S^*}\Xi_i^2}\\
      \le& \sqrt{2ss_0}\lambda_{t+1} + \delta\| \beta^*-\beta^{t}\|_2 +2\sigma\sqrt{\frac{ss_0\D}{n}}.
    \end{aligned}
  \end{equation}
  Since the procedure of operator $\mathcal{T}_{\lambda, s_0}(\cdot)$ has two steps, the term 
$\sum_{i\in S^*}(H^{t+1}_i)^2$ in \eqref{eq:iht9} is upper bounded by $2ss_0\lambda^2_{t+1}$.
Combining \eqref{eq:iht8} and \eqref{eq:iht9}, we conclude that
  $$
  \begin{aligned}
    \|\beta^{t+1}-\beta^* \|_2 
    &\le \| \beta^{t+1}_{(S^*)^c}\|_2+\|(\beta^{t+1}-\beta^*)_{S^*} \|_2\\
    &\le \sqrt{2ss_0}\lambda_{t+1} + 2\delta\|\beta^* - \beta^t\|_2 + 2(1+\sqrt{2})\sigma\sqrt{\frac{ss_0\D}{n}}\\
    &\le \left(\sqrt{2}+3(1+\sqrt{2}) \delta^{ \frac{9}{10}} +\frac{1+\sqrt{2}}{2}\right)\sqrt{ss_0}\lambda_{t+1}\\
    &\le \frac{3}{2}(1+\sqrt{2})\sqrt{ss_0}\lambda_{t+1},
  \end{aligned}
  $$
  where the third and the last inequalities follow from $\delta <0.11 \wedge \kappa^{10}$.
  We prove that $\eqref{eq:iht4}$ holds for $t+1$.
  
  Finally, we have proved that the results in Theorem \ref{th1} hold for $t+1$ under the induction hypothesis. This completes the proof of Theorem \ref{th1}.

\subsection* {Proof of Theorem \ref{thm:lower}}
  Consider the $\frac{ss_0}{4}$-packing set $\{\beta^1, \ldots, \beta^{M}\}$, where $M$ is the shorthand for the packing number $M(\frac{ss_0}{4};\widetilde\Theta^{m,d}(s, s_0), \|\cdot\|_H)$.
We set all the non-zero elements of $\beta \in \{\beta^1, \ldots, \beta^{M}\}$ equal to $\delta$, where $\delta$ is a parameter that need to be determined below.
  For any $\beta^{i} \neq \beta^{j}$, since each $\beta^i$ has at most $ss_0$ nonzero elements, we have
  \begin{equation}\label{ap16}
    \|\beta^{i}- \beta^{j}\|_2^2 \leq 2ss_0\delta^2,\ \forall i,j \in [M].
  \end{equation}
  On the other hand, since $\{\beta^1, \ldots, \beta^{M}\}$ is a $\frac{ss_0}{4}$-packing set of $\widetilde\Theta^{m,d}(s, s_0)$, we have 
  \begin{equation}\label{ap3}
    \|\beta^{i}- \beta^{j}\|_2^2 \geq \frac{1}{4}ss_0\delta^2,\ \forall i,j \in [M].
  \end{equation}
  Given design matrix $X$, denote $y^{i} = X\beta^{i} + \xi, \forall i \in [M]$. We consider the Kullback-Leibler divergence
between different distribution pairs as
$$
\begin{aligned}
KL\left(y^{i}|| y^{j}\right) &= \frac{1}{2\sigma^2}\|X(\beta^{i}-\beta^{j})\|_2^2\\
&\le \frac{n\vartheta_{\max}^2}{2\sigma^2}\|\beta^{i}-\beta^{j}\|_2^2,
\end{aligned}
$$
where the last inequality follows from the eigenvalue value condition of $X$ and $\beta^{i}- \beta^{j} \in \widetilde\Theta^{m,d}(2s, 2s_0)$.
Denote $B$ as the random vector uniformly distributed over the packing set. 
Observe that
  \begin{align}\label{ap6}
    \begin{split}
      I(y;B) \leq& \frac{1}{{M \choose 2}} \sum_{i \neq j} KL(y^{i} || y^{j})\\
      \leq & \frac{1}{{M \choose 2}} \sum_{i \neq j} \frac{n\vartheta_{\max}^2}{2\sigma^2}\|\beta^{i}-\beta^{j}\|_2^2\\
      \leq&\frac{n\vartheta_{\max}^2}{\sigma^2}ss_0\delta^2,
    \end{split}
  \end{align}
 where the last inequality uses \eqref{ap16}.
  Combining the generalized Fano's Lemma \citep{thomas2006elements} and \eqref{ap6}, we have
  \begin{align*}
    P(B \neq \widetilde\beta)
    &\geq 1 - \frac{I(y;B)+\log 2}{\log M}\\
    &\geq 1 - \frac{\frac{ n\vartheta_{\max}^2}{\sigma^2}ss_0\delta^2+\log 2}{\log M},
  \end{align*}
  where $\widetilde\beta$ takes value in the packing set.
  To guarantee $P(B \neq \widetilde \beta) \geq \frac{1}{2}$, it suffices to choose
  $\delta = \frac{1}{2}\sqrt{\frac{\sigma^2 \log M}{n\vartheta_{\max}^2ss_0}}$.
  Substituting it into equation \eqref{ap3} and from Lemma \ref{lem2}, we have
\begin{align*}
    &\inf_{\hat\beta}\sup_{\beta^* \in \Theta^{m,d}(s, s_0)} {\mathbf{E}}_{\hat \beta } \|\hat\beta-\beta^*\|_2^2 \\
\geq&  \frac1{16} ss_0 \delta^2 \cdot\inf_{B} P\left( B \ne \tilde \beta \right) \\
\geq& \frac{\sigma^2 \log M}{128n\vartheta_{\max}^2}\\
\ge& \frac{\sigma^2}{ 512 n\vartheta_{\max}^2} 
  \left( ss_0\log \frac{ed}{s_0}+s\log \frac{em}{s} \right),
  \end{align*}
  which completes the proof of Theorem \ref{thm:lower}.

\subsection* {Proof of Theorem \ref{lem:lam0}}

  Using Lemma \ref{lem:sigma}, with probability at least $1-\exp\left\{-Css_0\Delta\right\}$, we have
  \begin{equation}\label{lam0e1}
    \sigma_0 \ge \frac{19}{20}\sigma - \sqrt{1+\delta}\|\beta^*\|_2.
  \end{equation}
  With probability at least $1-\exp\left\{-Css_0\Delta\right\}$, we have
  \begin{align}\label{lam0e2}
    \begin{split}
    \|M_{S^*}\|_2 &= \|\left(\beta^* + \Phi \beta^* + \Xi\right)_{S^*}\|_2\\
    &\geq \|\beta^*\|_2 - \|\Phi \beta^*\|_2 - \|\Xi_{S^*}\|_2\\
    &\geq (1-\delta) \|\beta^*\|_2 - 2\sigma\sqrt{\frac{ss_0\D}{n}}. 
  \end{split}
  \end{align}
Note that $\sqrt{ss_0}\|M\|_{\infty} \geq \|M_{S^*}\|_2$. Combining \eqref{lam0e1} and \eqref{lam0e2}, with probability at least $1-\exp\left\{-Css_0\Delta\right\}$, we have
\begin{align*}
  \sqrt{ss_0} \lambda_0 
  &\geq \frac{100}{9}{\sigma}_0\sqrt{\frac{ss_0\D'}{n}} \vee \frac{19}{4}\sqrt{ss_0}\|M\|_{\infty}\\
  &\geq \frac{9}{19}\times \frac{100}{9}{\sigma}_0\sqrt{\frac{ss_0\D'}{n}}+\frac{10}{19}\times\frac{19}{4}\|M_{S^*}\|_2\\
  &\geq \frac{100}{19}\left(\frac{19}{20}\sigma - \sqrt{1+\delta}\|\beta^*\|_2 \right)\sqrt{\frac{ss_0\D'}{n}}+\frac{5}{2}\left((1-\delta)\|\beta^*\|_2 - 2\sigma \sqrt{\frac{ss_0\D}{n}}\right)\\
	&\geq \left(\frac{5}{2}(1-\delta) - \frac{100}{19}\sqrt{1+\delta}\sqrt{\frac{ss_0\D'}{n}}\right)\|\beta^*\|_2 \\
  &\geq \|\beta^*\|_2 
\end{align*}
where the fourth inequality uses the fact that $\D' \geq \D$, and the last inequality uses the fact that $n > 105^2ss_0\Delta'$ and $\delta < 0.11 $.
We complete the proof of Theorem \ref{lem:lam0}.

\subsection* {Proof of Theorem \ref{thm:stopping}}

  Note that $t_0 \leq t_{\infty}$ holds since $\D' \geq \D$.
  We first claim that $t_0 \geq \bar t$. 
  For any $t \leq t_0$, according to the definition of $t_0$, we have
  \begin{align}\label{stope1}
    \sigma\sqrt{\frac{\D'}{n}} \le \frac{1}{12}\lambda_t.
  \end{align}
  From Lemma \ref{lem:sigma} and Theorem \ref{th1}, with probability at least $1-\exp\left\{-Css_0\Delta\right\}$, we have
  \begin{align}\label{stope2}
    \begin{split}
		{\sigma}_t &\le \sqrt{1+\delta}\|\beta^* - \beta^t\|_2 + \frac{21}{20}\sigma\\
		&\le \sqrt{1+\delta}\frac{3}{2}(1+\sqrt{2})\sqrt{ss_0}\lambda_t+\frac{21}{20}\sigma.
    \end{split}
  \end{align}
  From \eqref{stope2}, it comes out that
  \begin{align}\label{stope3}
    \begin{split}
      8{\sigma}_t\sqrt{\frac{\D'}{n}} &\le 12(1+\sqrt{2})\sqrt{1+\delta}\lambda_t\sqrt{\frac{ss_0\D'}{n}}+ \frac{42}{5}\sigma\sqrt{\frac{\D'}{n}}\\
    &\le 12(1+\sqrt{2})\sqrt{1+\delta}\lambda_t\sqrt{\frac{ss_0\D'}{n}} + \frac{7}{10}\lambda_t\\
    &\le \lambda_t,
  \end{split}
  \end{align}
  where the first inequality uses \eqref{stope1}, and the second inequality follows from $\delta < 0.11 $ and $n > 105^2ss_0\Delta'$.
  \eqref{stope3} leads to the fact that $t \leq \bar t$, which deduces that $t_0 \leq \bar t$ holds with high probability.
  
  Next, we turn to the proof of $\bar t \leq t_{\infty}$. Since $t_0 \leq t_{\infty}$, Theorem \ref{th1} shows us that
  \begin{align}\label{stope4}
    \|\beta^{t_0} - \beta^*\|_2 \le 18(1+\sqrt{2})\sigma\sqrt{\frac{ss_0\D'}{n}}.
  \end{align}
  From Lemma $\ref{lem:sigma}$, for any $t_0 \leq t \leq t_{\infty}$, it holds with probability at least $1-\exp\left\{-Css_0\Delta\right\}$ that
  \begin{align*}
    |{\sigma}_t-\sigma| &\le \sqrt{1+\delta}\|\beta^* - \beta^t\|_2 + \frac{1}{20}\sigma\\
    &\le 18(1+\sqrt{2})\sqrt{1+\delta}\sigma\sqrt{\frac{ss_0\D'}{n}} + \frac{1}{20}\sigma\\
    &\le (\frac{9}{20}+\frac{1}{20})\sigma = \frac{1}{2}\sigma, 
  \end{align*}
  where the second inequality follows from \eqref{stope4}, and the last inequality follows from $\delta <0.11$ and  $n > 105^2ss_0\Delta'$. 
  Combining the above inequalities, we have 
  \begin{equation*}
    \frac{8{\sigma}_t}{\sqrt{n}}\sqrt{\D'} \ge \frac{4\sigma}{\sqrt{n}}\sqrt{\D'}\ge \frac{4\sigma}{\sqrt{n}}\sqrt{\D}.
  \end{equation*}
  This result implies that $\bar t \leq t_{\infty}$, which completes the proof of Theorem \ref{thm:stopping}.

\subsection* {Proof of Theorem \ref{thm:optimal}}

From Lemma \ref{lem:sigma} and \eqref{stope4}, with probability at least $1-\exp\{-Css_0\D\}$, we have
\begin{align}\label{optimale1}
  \begin{split}
  |\sigma_{\bar{t}} - \sigma| &\le \sqrt{1+\delta}\|\beta^{\bar{t}} - \beta^*\|_2 + \frac{1}{20}\sigma\\
  &\le \sigma \left(18(1+\sqrt{2})\sqrt{1+\delta}\sqrt{\frac{ss_0\D'}{n}} + \frac{1}{20}\right)\\
  &\le \frac{1}{10}\sigma,
  \end{split}
\end{align}
where  the last inequality uses $\delta < 0.11$ and  $n > 1000^2ss_0\Delta'$. 

First, we prove $\|\beta^{\tilde{t}}\|_G \leq 4s$ by contradiction. Let us assume that $\|\beta^{\tilde{t}}\|_G > 4s$.
According to the definition of $\tilde t$, we have
\begin{equation}\label{optimale5}
  \frac{1}{n}\left\|y - X{\beta}^{\tilde{t}}\right\|_2^2 + \frac{1000{\sigma}_{\bar{t}}^2\Omega({\beta}^{\tilde{t}})}{n} \leq \frac{1}{n}\left\|y - X{\beta}^{t_{\infty}}\right\|_2^2 + \frac{1000{\sigma}_{\bar{t}}^2\Omega({\beta}^{t_{\infty}})}{n}.
\end{equation}
On one hand, we have 
\begin{equation*}
	\begin{split}
		\left\|y - X{\beta}^{\tilde{t}}\right\|_2^2 &\ge \left\|\xi\right\|_2^2 + \left\|X({\beta}^{\tilde{t}}-\beta^*)\right\|_2^2-2\left|\left\langle\xi,X({\beta}^{\tilde{t}}-\beta^*)\right\rangle\right|\\
    &\ge \left\|\xi\right\|_2^2 + \left\|X({\beta}^{\tilde{t}}-\beta^*)\right\|_2^2-2\sigma\sqrt{3\Omega^*({\beta}^{\tilde{t}})}\left\|X({\beta}^{\tilde{t}} - \beta^*)\right\|_2\\
		&\ge \left\|\xi\right\|_2^2 + \left\|X({\beta}^{\tilde{t}}-\beta^*)\right\|_2^2-2\sigma\sqrt{3\times\frac{5}{4}\Omega({\beta}^{\tilde{t}})}\left\|X({\beta}^{\tilde{t}} - \beta^*)\right\|_2\\
		&\ge \|\xi\|^2 - \frac{15}{4}\sigma^2\Omega({\beta}^{\tilde{t}}),
	\end{split}
\end{equation*}
where the second inequality follows from Lemma \ref{lem:inner}, and the third inequality uses the fact that $\frac{5}{4}\Omega({\beta}^{\tilde{t}}) \geq \Omega^*({\beta}^{\tilde{t}})$ when $\|\beta^{\tilde{t}}\|_G > 4s$. 
The definition of $\Omega^*(\beta)$ is given at the beginning of the Appendix.
By some simple algebras, it comes out that 
\begin{align}\label{optimale2}
  \begin{split}
  \frac{1}{n}\left\|y - X{\beta}^{\tilde{t}}\right\|_2^2 + \frac{1000{\sigma}_{\bar{t}}^2\Omega({\beta}^{\tilde{t}})}{n} \geq & \frac{\|\xi\|_2^2}{n}-\frac{15\sigma^2\Omega({\beta}^{\tilde{t}})}{4n} + \frac{1000{\sigma}_{\bar{t}}^2\Omega({\beta}^{\tilde{t}})}{n}\\
  \geq& \frac{\|\xi\|_2^2}{n}+ \frac{950{\sigma}_{\bar{t}}^2\Omega({\beta}^{\tilde{t}})}{n}.
  \end{split}
\end{align}
On the other hand, we have
\begin{equation*}
	\begin{split}
		\left\|y - X{\beta}^{t_{\infty}}\right\|_2^2 &\le \left\|\xi\right\|_2^2 + \left\|X({\beta}^{t_{\infty}}-\beta^*)\right\|_2^2+2\left|\left\langle\xi,X({\beta}^{t_{\infty}}-\beta^*)\right\rangle\right|\\
		&\le \left\|\xi\right\|_2^2 + \left\|X({\beta}^{t_{\infty}}-\beta^*)\right\|_2^2+6\sigma\sqrt{\Omega(\beta^*)}\left\|X({\beta}^{t_{\infty}} - \beta^*)\right\|_2\\
		&\le \|\xi\|_2^2 + 2\left\|X({\beta}^{t_{\infty}}-\beta^*)\right\|_2^2 +  9\sigma^2\Omega(\beta^*),
	\end{split}
\end{equation*}
where the second inequality follows from \eqref{orthogonal} in Lemma \ref{lem:inner} since $\beta^{t_{\infty}}-\beta^*$ is $(2s, \frac{3}{2}s_0)$-sparse. 
By some simple algebras, it comes out that 
\begin{align}\label{optimale3}
  \begin{split}
  \frac{1}{n}\left\|y - X{\beta}^{t_{\infty}}\right\|_2^2 + \frac{1000{\sigma}_{\bar{t}}^2\Omega({\beta}^{t_{\infty}})}{n} \leq& \frac{\|\xi\|_2^2}{n} + \frac{2}{n}\left\|X({\beta}^{t_{\infty}}-\beta^*)\right\|_2^2 +  \frac{9\sigma^2\Omega(\beta^*)}{n}\\
  ~~~+& \frac{1000{\sigma}_{\bar{t}}^2\Omega({\beta}^{t_{\infty}})}{n}\\
  \leq& \frac{\|\xi\|_2^2}{n} + 2(1+\delta)\left(\frac{3}{2}(1+\sqrt{2})\right)^2\times \frac{16\sigma^2ss_0\D}{n} \\
  ~~~+&  \frac{9\sigma^2ss_0\D}{n} + \frac{1000{\sigma}_{\bar{t}}^2}{n}\left(2s\log \frac{em}{2s}+3ss_0\log \frac{ed}{s_0}\right)\\
  \leq & \frac{\|\xi\|_2^2}{n} + \frac{480\sigma^2ss_0\D}{n} + \frac{3000{\sigma}_{\bar{t}}^2ss_0\D}{n}\\
  \leq & \frac{\|\xi\|_2^2}{n} + \frac{3600{\sigma}_{\bar{t}}^2ss_0\D}{n},
\end{split}
\end{align}
where the third inequality holds for $\delta < \frac{1}{10}$ and the last inequality follows from \eqref{optimale1}.
Combining \eqref{optimale5}-\eqref{optimale3}, from the triangle relationship, we have
\begin{equation*}
  950 \Omega({\beta}^{\tilde{t}}) \le 3600ss_0\D.
\end{equation*}
Recall $\Omega({\beta}^{\tilde{t}}) > 4ss_0\D$ under the assumption $\|\beta^{\tilde{t}}\|_G > 4s$.
Then, we have
 \begin{equation*}
  3800ss_0\D< 950 \Omega({\beta}^{\tilde{t}}) \le 3600ss_0\D.
\end{equation*}
which contradicts the assumption of $\|\beta^{\tilde{t}}\|_G > 4s$. Therefore, we must have $\|\beta^{\tilde{t}}\|_G \leq 4s$.

Next, we show the upper bounds for the estimation error $\|{\beta}^{\tilde{t}}-\beta^*\|$.
From \eqref{optimale5}, we have
\begin{align*}
  \left\|y - X{\beta}^{t_{\infty}}\right\|_2^2 + 1000{\sigma}_{\bar{t}}^2\Omega({\beta}^{t_{\infty}}) \geq& \left\|y - X{\beta}^{\tilde{t}}\right\|_2^2 + 1000{\sigma}_{\bar{t}}^2\Omega({\beta}^{\tilde{t}})\\
  \geq & \left\|y - X{\beta}^{\tilde{t}}\right\|_2^2\\
  \geq & \left\|\xi\right\|_2^2 + \left\|X({\beta}^{\tilde{t}}-\beta^*)\right\|_2^2-2\sigma\sqrt{3\Omega^*({\beta}^{\tilde{t}})}\left\|X({\beta}^{\tilde{t}} - \beta^*)\right\|_2.
\end{align*}
Combining the above inequalities and \eqref{optimale3}, we have 
\begin{equation}\label{optimale6}
  \left\|X({\beta}^{\tilde{t}}-\beta^*)\right\|_2^2-2\sigma\sqrt{3\Omega^*({\beta}^{\tilde{t}})}\left\|X({\beta}^{\tilde{t}} - \beta^*)\right\|_2 \leq 3600\sigma_{\bar t}^2 ss_0\D.
\end{equation}
By solving the quadratic inequalities \eqref{optimale6}, we have
$$
\left\|X({\beta}^{\tilde{t}}-\beta^*)\right\|_2 \le 140\sigma\sqrt{ss_0\D}.
$$
Note that $\beta^{\tilde{t}}$ is $(4s, s_0)$-sparse. Then, we conclude that by DSRIP condition, we have 
\begin{equation*}
  \left\|{\beta}^{\tilde{t}}-\beta^*\right\|_2\leq \frac{\left\|X({\beta}^{\tilde{t}}-\beta^*)\right\|_2}{\sqrt{n(1-\delta)}} \leq 150\sigma \sqrt{\frac{ss_0\D}{n}},
\end{equation*}
which completes the proof of Theorem \ref{thm:optimal}.

\subsection* {Proof of Corollary \ref{cor:iter_num}}

Note that with probability at least $1-p^{-C}$, 
\begin{align}\label{iter_num:e1}
  \begin{split}
  \|M\|_{\infty} \leq& \|\beta^* + \Phi \beta^*\|_{\infty} + \|\Xi\|_{\infty}\\
  \leq& \|\beta^*+\Phi \beta^*\|_2+\|\Xi\|_{\infty}\\
  \leq&(1+\delta)\|\beta^*\|_2+2\sigma\sqrt{\frac{\log ep}{n}}\\
  \leq&4\left(\|\beta^*\|_2 \vee \sigma \sqrt{\frac{\log ep}{n}}\right),
  \end{split}
\end{align}
where the last inequality uses $\delta \leq 1$.
Substituting \eqref{iter_num:e1} into the definition of $\lambda_0$, we have
\begin{align}\label{iter_num:eq2}
  \begin{split}
    \lambda_0 =& \frac{100}{9}{\sigma}_0\sqrt{\frac{\D'}{n}} \vee \frac{19}{4}\|M\|_{\infty}\\
    \leq& 19 \left(\|\beta^*\|_2 \vee \sigma\sqrt{\frac{\log ep}{n}}\right),
  \end{split}
\end{align}
where the last inequality uses $\D' \leq \log (ep)$.
Observe that $\kappa^T \lambda_0 \leq 4\frac{\sigma_{\bar t}}{\sqrt{n}}$. 
By some simple algebras, with probability at least $1-\exp\{-Css_0\D\}$, we have
\begin{align*}
  \begin{split}
    T \leq& \log\left(\frac{\lambda_0 \sqrt{n}}{4\sigma_{\bar t}}\right)/\log(\frac{1}{\kappa})\\
    \leq&\log\left(\frac{5\lambda_0 \sqrt{n}}{18\sigma}\right)/\log(\frac{1}{\kappa})\\
  \leq& \log\left(6(\frac{\sqrt{n}\|\beta^*\|_2}{\sigma} \vee \sqrt{\log ep})\right)/\log(\frac{1}{\kappa}),
  \end{split}
\end{align*}
where the second inequality follows from \ref{optimale1} and the last inequality uses \eqref{iter_num:eq2}.

Therefore,
$$
\sup\limits_{S^* \in \mathcal{S}^{m, d}(s, s_0)} P\left(T \geq \log\left(6(\frac{\sqrt{n}\|\beta^*\|_2}{\sigma} \vee \sqrt{\log ep})\right)/\log(\frac{1}{\kappa})+1 \right)\leq e^{-Css_0 \Delta}.
$$

\subsection*{Proof of Theorem \ref{adaptive2}}
Our technique for tuning $s_0$ is notably distinct and more complex than that of \citet{verzelen2012minimax}. As discussed in Section \ref{tradeoff}, the theoretical properties differ significantly between the cases $\bar s_0 > s_0$ and $\bar s_0 < s_0$. We can control the sparsity at either the element-wise or group-wise level, but not both simultaneously. Additionally, as illustrated in Figure \ref{fig:rate}, the minimax rates for different values of $\bar s_0$ exhibit a``U-shaped" curve, rather than the monotonically increasing trend observed under element-wise sparsity \citep{raskutti2011minimax, verzelen2012minimax}. Therefore, we must separately analyze the cases  $\bar s_0 > s_0$ and $\bar s_0 < s_0$.

\subsubsection*{\bf{Basic inequality of Verzelen's procedure} }

In this part, we give the basic comparable inequality used in the proof. This part is similar to Theorem 5.2 in \citet{verzelen2012minimax}. Denote
	\begin{equation*}
		\operatorname{pen}\left(\bar{s}_0\right)=\frac{K}{n}\left(\A \log ed +\s \log \frac{e m}{\s}\right),\quad \operatorname{pen}{ }^{\prime}\left(\bar{s}_0\right)=-1+\exp \left(\operatorname{pen}\left(\bar{s}_0\right)\right).
	\end{equation*}
By the definition of $\hat{s}_0$, we have

\begin{equation}\label{compara0}
		\frac{1}{n}\left\|y-X \hat{\beta}\left(\hat{s}_0\right)\right\|_2^2  \cdot\left(1+\operatorname{pen}^{\prime}(\hat{s}_0)\right) \leq  \frac{1}{n}\left\|y-X \hat{\beta}\left(s_0\right)\right\|_2^2\cdot \left(1+\operatorname{pen}{ }^{\prime}\left(s_0\right)\right) .
\end{equation}
For the right-hand side of  \eqref{compara0}, a strategy similar to \eqref{optimale3} leads that
\begin{equation}\label{compara1}
		\begin{aligned}
			\frac{1}{n}	 \left\|y-X \hat{\beta}\left(s_0\right)\right\|_2^2  \leq  	\frac{1}{n}\|\xi\|_2^2 +C_1  \frac{\sigma^2 ss_0 \Delta}{n}.
		\end{aligned}
	\end{equation}
Recall that we assume $n$ is large enough so that
$$
\operatorname{pen} (s_0) 
\le \frac{4K  }{n} \left( s s_0 \log ed + s \log(em/s) \right)
< 0.1,
$$
which implies that $ 1+\operatorname{pen}^{\prime}\left(s_0\right)=\exp \left\{\operatorname{pen}\left(s_0\right)\right\} \leqslant e
$. 
Combining \eqref{compara0} and \eqref{compara1}, we have
\begin{equation}
	\frac{1}{n}\left\|y-X \hat{\beta}\left(\hat s_0\right)\right\|_2^2\cdot \left(1+\operatorname{pen}{ }^{\prime}\left( \hat s_0\right)\right) 
 \leq \frac{1}{n}\|\xi\|_2^2 \left(1+ \operatorname{pen}^\prime \left(s_0\right)\right)+C_2 \frac{\sigma^2 ss_0 \Delta}{n}.
\end{equation}

For the left-hand side of \eqref{compara0}, a strategy similar to \eqref{optimale2} leads that
\begin{equation}\label{compara4}
\begin{aligned}
	\frac{1}{n}\left\|y-X \hat{\beta}\left(\hat{s}_0\right)\right\|_2^2  &\geq \frac{1}{n}\|\xi\|_2^2 +\frac{1}{n}\left\|X\left(\beta^*-\hat{\beta}\left(\hat{s}_0\right)\right)\right\|_2^2  \\
	&-\frac2n \left\|X\left(\beta^*-\hat{\beta}\left(\hat{s}_0\right)\right)\right\|_2 
    \cdot  \left| \left\langle \xi, \frac{X\left(\beta^*-\hat{\beta}\left(\hat{s}_0\right)\right)}{\left\|X\left(\beta^*-\hat{\beta}\left(\hat{s}_0\right)\right)\right\|_2} \right\rangle \right|.
\end{aligned}
\end{equation}
Then, we can also upper bound the inner product by Lemma \ref{lem:inner} as
\begin{equation}\label{inner1}
\begin{aligned}
	 \left|\left\langle \frac{\xi}{\sigma}, \frac{X\left(\beta^*-\hat{\beta}\left(\hat{s}_0\right)\right)}{\left\| X\left(\beta^*-\hat{\beta}\left(\hat s_0\right) \right)\right\|_2}\right\rangle\right|^2
    &\precsim \left(ss_0+\hat A(\hat s_0) \right) \log \frac{e d  (s+\hat s(\hat s_0)  ) }{ss_0+\hat A(\hat s_0) } + \left(s+ \hat s(\hat s_0) \right) \log \frac{e m}{s+ \hat s(\hat s_0) } \\
    &\leq \left(ss_0+\hat A(\hat s_0) \right) \log e d + \left(s+\hat s(\hat s_0)  \right) \log \frac{e m}{s+\hat s(\hat s_0)} .\\
\end{aligned}
\end{equation}
Denote
$\mathcal{L}:=\left\|X\left(\beta^*-\hat{\beta}\left(\hat s_0 \right)\right)\right\|_2$ and $\hat{\Gamma} :=\left(ss_0+ \hat A (\hat s_0 )\right) \log {e d} +\left(s+\hat s(\hat s_0)\right) \log \frac{e m}{s+\hat s(\hat s_0) }$. Therefore, by \eqref{compara0}-\eqref{inner1}, we obtain
\begin{equation}\label{compara3}
	 \mathcal{L}^2-C_2' \sigma\sqrt{\hat{\Gamma} } \mathcal{L} \le \operatorname{pen}'(s_0)\| \xi \|_2^2 + C_2 \sigma^2 ss_0 \Delta.
\end{equation}
By Lemma 1 of \citet{laurent2000}, we conclude that $0.9\sigma^2 \le \frac1n \| \xi\|_2^2 \le 1.1 \sigma^2$ holds with probability at least $1-\exp(-C_4n)$. 
Besides, by $\operatorname{pen}(s_0) < 0.1$, we derive that $\operatorname{pen}'(s_0) = \exp(\operatorname{pen}(s_0))-1 \le 2 \operatorname{pen}(s_0)$, therefore
$$
\operatorname{pen}'(s_0)\| \xi \|_2^2 + C_2 \sigma^2 ss_0 \Delta
\le C_3  \sigma^2 \big(ss_0 \log d + s \log(em/s) \big)
\le C_3 \sigma^2 \hat \Gamma,
$$
which leads to
\begin{equation}\label{compara3.5}
	 \mathcal{L}^2-C_2' \sigma\sqrt{\hat{\Gamma} } \mathcal{L} \le  C_3 \sigma^2 \hat \Gamma.
\end{equation}
By solving inequality \eqref{compara3.5}, we obtain the upper bound $
\mathcal{L}^2 \le C_4 \sigma^2  \hat{\Gamma}$. Therefore, to get the optimal upper bound for estimation error, we just need to prove that
\begin{equation}\label{eq:pp}
    \hat{\Gamma} \precsim  ss_0 \log ed + s \log(em/s) . 
\end{equation}
By far, based on table \ref{table:1} we know that $\hat A (\hat s_0) \le 4 ss_0$ for $\hat{s}_0 \le s_0$, and $\hat s (\hat s_0) \le 4 s$ for  $\hat{s}_0 \ge s_0$. 
Therefore, with high probability, we conclude that 
\begin{equation}\label{th16:key}
\hat{\Gamma} \le 
\begin{cases}
5 ss_0\log e d+\left(s+ \hat s (\hat s_0) \right) \log \frac{e m}{s+\hat s(\hat s_0) }, & \hat{s}_0 \le s_0.\\
\left(ss_0+ \hat A (\hat s_0)\right) \log e d+ 5s \log \frac{e m}{s},  & \hat{s}_0 > s_0.
\end{cases}
\end{equation}
For convenience, we divide the next proof into two cases: $ss_0\log  ed \le s\log\frac{em}{s}$ or $ss_0\log ed \ge s\log\frac{em}{s}$.

\subsubsection*{\bf{Assumption A: $ss_0\log ed \ge s\log\frac{em}{s}$.}}

{\bf CASE 1: $\hat{s}_0 \ge s_0$.} 

By \eqref{th16:key}, we need to prove $ \hat A (\hat s_0) \leq 9 ss_0$. By using contradiction, we assume $ \hat A (\hat s_0) > 9s s_0$ holds at first and obtain:
\begin{align}
\begin{split}\label{eq:c1a}
\frac{1+\operatorname{pen}^{\prime}\left(\hat{s}_0\right)}{1+\operatorname{pen}^{\prime}\left( s_0\right)} &\geq \exp \left\{\frac{K}{n}\left( \hat A (\hat s_0) \log ed +  \hat s (\hat s_0) \log \frac{e m}{ \hat s (\hat s_0)}\right)- 4\frac{K}{n}\left(ss_0 \log ed +s \log \frac{e m}{s}\right)\right\} \\
	& \geq\exp \left\{\frac{K}{n}  \hat A (\hat s_0)\log ed -4 \frac{K}{n}\left( ss_0 \log ed +s\log \frac{e m}{s}\right)\right\} \\
	& \geq \exp \left\{\frac{K}{n}  \hat A (\hat s_0) \log e d-8 \frac{K}{n} ss_0 \log ed \right\} \\
	& \geq \exp \left\{\frac{K}{9n}  \hat A (\hat s_0) \log {ed} \right\} .
 \end{split}
	\end{align}
Combining with \eqref{compara1} we have
\begin{equation}\label{eq:c1b}
\frac{1}{n}\left\|y-X \hat{\beta}\left(s_0\right)\right\|_2^2  
\leq  \frac{1}{n}\|\xi\|_2^2 +C_1 \frac{\sigma^2 s s_0 \Delta}{n} 
\leq \frac{1}{n}\|\xi\|_2^2 +C_6 \sigma^2 \frac{ \hat A (\hat s_0) \log {ed} }{n}.
\end{equation}
Besides, by \eqref{compara4}, we also have
\begin{equation}\label{eq:c1c}
 \frac{1}{n} \left\|y-X \hat{\beta}\left(\hat{s}_0\right)\right\|_2^2  \geq  \frac{1}{n}\|\xi\|_2^2 - \frac{1}{n}C_7\sigma^2 \left( \hat A (\hat s_0) \log ed \right) ,
\end{equation}
	where we use $a^2-2ab \ge - b^2$, and the inner product term of \eqref{compara4} is upper bounded  by \eqref{inner1}.
Therefore, combining \eqref{compara0} and \eqref{eq:c1a}-\eqref{eq:c1c}, we have 
\begin{equation}\label{compara5}
	\left( \frac{1}{n}\|\xi\|_2^2 -\frac{C_7\sigma^2}{n} \hat A (\hat s_0) \log  ed \right) \exp \left\{\frac{K}{9 n}\hat A (\hat s_0)\log {ed} \right\} 
 \leq \frac{1}{n}\|\xi\|_2^2 + \frac{C_6\sigma^2}{n}\hat A (\hat s_0)\log ed .
\end{equation}
Let $t = \frac{1}{9 n} \hat A (\hat s_0)\log ed $. Note that $n$ is large enough such that $t \in (0,\frac{1}{K})$ by Assumption \ref{samplesize}. To establish a contradiction, we need to verify that for $t \in (0,\frac{1}{K})$, 
\begin{equation}\label{ft}
F(t) = \left(1-\frac{9C_7 t}{\|\xi\|_n^2/\sigma^2} \right)\exp(Kt) - \left(1+ \frac{9 C_6 t}{\|\xi\|_n^2/\sigma^2} \right) > 0
\end{equation}
always holds. Note that $F(0) = 0$ and $F'(t) = \exp(Kt)\left\{K-\frac{9C_7(Kt+1) }{\|\xi\|_n^2/\sigma^2}\right\} -\frac{9C_6}{\|\xi\|_n^2/\sigma^2}$. 
By $ 0.9\sigma^2 \le \frac{1}{n}\|\xi\|_2^2  \le 1.1 \sigma^2$ and $Kt \in (0,1)$, we could select a sufficiently large $K \ge 10 C_6 + 20 C_7$. Hence, we verify that $F'(t) > 0$ for $\forall t \in (0, \frac{1}{K})$, which leads to an absurd to \eqref{compara5} with high probability.

Therefore, we prove $\hat A (\hat s_0) \le 9ss_0$ holds with high probability. Then, based on \eqref{th16:key} we derive that
\begin{equation}
	\hat \Gamma \le 10  \left( ss_0 \log ed + s \log(em/s)\right), ~ \forall \hat{s}_0 \ge s_0,
\end{equation}
which proves \eqref{eq:pp} holds with high probability.

{\bf CASE 2: $\hat {s}_0 <s_0$.} 

We divide this case into two subcases and analyse them respectively.
\begin{enumerate}[(a)]
    \item If  $9ss_0\log ed  >\hat s (\hat s_0) \log \frac{e m}{\hat s (\hat s_0)}$, we just bound the inner product in \eqref{inner1} by:
 $$
 \begin{aligned}
	 \left|\left\langle\xi, \frac{X\left(\beta^*-\hat{\beta}\left(\hat{s}_0\right)\right)}{\left\| X\left(\beta^*-\hat{\beta}\left(\hat s_0\right) \right)\right\|_2}\right\rangle\right|^2 &\leq 5s s_0 \log ed +(s+\hat s (\hat s_0)) \log \frac{e m}{(s+\hat s (\hat s_0))}  \\
 & \leq  5 ss_0 \log ed + s \log(em/s) + 9 ss_0 \log ed \\
 &\le 14 \left( ss_0 \log {e d} + s \log ({e m}/{s } ) \right),
 \end{aligned}
 $$
 where the first inequality uses $\hat A (\hat s_0) < 4 ss_0$ for $\hat s_0 < s_0$, and the second inequality uses $\hat s (\hat s_0) \log \frac{e m}{\hat s (\hat s_0)} < 9ss_0\log ed $.
By solving \eqref{compara3}, we derive an upper bound for $\mathcal{L}$ as
\begin{equation}\label{c2win}
	\mathcal{L}^2 \precsim \frac{\sigma^2}{n}\left(ss_0\log ed +s\log\frac{em}{s}\right),
\end{equation}
therefore by DSRIP condition we prove \eqref{eq:so}.

\item If $9 s s_0 \log ed  \leq \hat s (\hat s_0) \log \frac{e m}{\hat s (\hat s_0) }$.
Then, similar to case 1, we will find an absurd with high probability. First, we obtain
$$
\begin{aligned}
	\frac{1+\operatorname{pen}^{\prime}\left(\hat{s}_0\right)}{1+\operatorname{pen}^{\prime}\left(s_0\right)} &\geq \exp \left\{\frac{K}{n} \hat s (\hat s_0) \log \frac{e m}{\hat s (\hat s_0)}-8 \frac{K}{n} s s_0 \log ed\right\}\\
	& \ge \exp \left\{\frac{K}{9 n} \hat s (\hat s_0) \log \frac{e m}{\hat s (\hat s_0)}\right\}.
\end{aligned}
$$
Then, use similar techniques in \eqref{eq:c1b} and \eqref{eq:c1c},we obtain the following inequalities:
\begin{equation}\label{Acase2b1}
\begin{aligned}
\frac{1}{n}\left\|y-X \hat{\beta}\left(s_0\right)\right\|_2^2  
&\leq  \frac{1}{n}\|\xi\|_2^2 +C_1 \frac{\sigma^2 s s_0 \Delta}{n} \\
&\leq \frac{1}{n}\|\xi\|_2^2 + \frac{2C_1}{9n}\sigma^2 \hat s (\hat s_0) \log \frac{e m}{\hat s (\hat s_0) },\\
\frac{1}{n} \left\|y-X \hat{\beta}\left(\hat{s}_0\right)\right\|_2^2  
&\geq  \frac{1}{n}\|\xi\|_2^2 - \frac{3\sigma^2}{n} \left( \big(s+ \hat s(\hat s_0)\big) \log \frac{e m}{s+\hat s(\hat s_0) }+ 4ss_0\log ed \right)\\
&\geq  \frac{1}{n}\|\xi\|_2^2 - \frac{14\sigma^2}{3n}\hat s(\hat s_0)\log \frac{e m}{\hat s(\hat s_0)},
\end{aligned}
\end{equation}
and thus
\begin{align}
\begin{split}\label{compara6}
	&\left( \frac{1}{n}\|\xi\|_2^2 -\frac{14\sigma^2}{3n}\hat s(\hat s_0)\log \frac{e m}{\hat s(\hat s_0)}\right) \exp \left\{\frac{K}{9 n}  \hat s (\hat s_0) \log \frac{e m}{\hat s (\hat s_0)} \right\} \\
 \leq &\frac{1}{n}\|\xi\|_2^2 +\frac{2C_1}{9n}\sigma^2 \hat s (\hat s_0) \log \frac{e m}{\hat s (\hat s_0) }.
 \end{split}
\end{align}

Now let $t = \frac{1}{9n}\hat s (\hat s_0) \log \frac{e m}{\hat s (\hat s_0) }$, and using the same technique corresponding to \eqref{ft}, with a sufficiently large $K \ge \frac{20C_1 + 840}{9}$ and we get an absurd.
Therefore, we prove that with high probability, $9 s s_0 \log ed > \hat s (\hat s_0) \log \frac{e m}{\hat s (\hat s_0) }$ holds, which leads to \eqref{c2win} and completes the proof in case 2.

By far, we have finished the proof in \textbf{Assumption A: $ss_0\log ed \ge s\log\frac{em}{s}$}. When $ss_0\log ed < s\log\frac{em}{s}$, the proof strategy is similar and we just give a proof sketch below.
\end{enumerate}

\subsubsection*{\bf{Assumption B: $ss_0\log ed < s\log\frac{em}{s}$.}}

{\bf CASE 3: $\hat {s}_0 \ge s_0$.} 

Similar to case 2, we continue to divide this case into two subcases:
\begin{enumerate}[(a)]
    \item If $\hat A (\hat s_0) \log ed \le 9s\log(em/ s)$, we obtain
    \begin{equation}\label{Bcase3eq1}
    \begin{aligned}
    \hat{\Gamma} &=\left(ss_0+ \hat A (\hat s_0 )\right) \log {e d} +\left(s+\hat s(\hat s_0)\right) \log \frac{e m}{s+\hat s(\hat s_0) }\\
    &\le ss_0 \log ed + 9 s\log (em/ s) + 5s \log (em/ s)\\ 
    &\le 14\left( ss_0 \log {e d} + s \log \frac{e m}{s } \right).
    \end{aligned}
    \end{equation}
    Therefore, we prove that \eqref{eq:pp} holds.

    \item If $\hat A (\hat s_0) \log ed > 9s\log(em/ s)$, we show that
$$
\frac{1+\operatorname{pen}^\prime\left(\hat{s}_0\right)}{1+\operatorname{pen}^\prime\left( {s}_0\right)} 
\ge \exp \left\{\frac{K}{9 n} \hat A(\hat s_0) \log ed \right\}.
$$
    Hence by using a strategy similar to \eqref{Acase2b1} and \eqref{compara6}, we show that $\hat A (\hat s_0) \log ed > 9s\log(em/ s)$ can not hold with high probability. 
    Therefore, in case 3, \eqref{Bcase3eq1} holds with high probability, which prove that \eqref{eq:pp} holds.
\end{enumerate}

{\bf CASE 4: $\hat {s}_0 <s_0$.} 

In this case, we just need to control $\hat s(\hat s_0)$. 
By using a contradiction similar to case 1,  at first, we assume $\hat s(\hat s_0) \ge 9s$ holds, which leads that
\begin{equation} 
	\begin{aligned}
 \frac{1+\operatorname{pen}^{\prime}\left(\hat{s}_0\right)}{1+\operatorname{pen}^{\prime}\left( s_0\right)} &\geq \exp \left\{\frac{K}{n}\left( \hat A (\hat s_0) \log ed +  \hat s (\hat s_0) \log \frac{e m}{ \hat s (\hat s_0)}\right)- 4\frac{K}{n}\left(ss_0 \log ed +s \log \frac{e m}{s}\right)\right\} \\
	& \geq \exp \left\{\frac{K}{9n}  \hat s (\hat s_0) \log \frac{em}{\hat s (\hat s_0)} \right\},
	\end{aligned}
\end{equation}
and by using a strategy similar to \eqref{eq:c1b}-\eqref{ft} we prove that $\hat s(\hat s_0) \ge 9s$ can not hold with high probability. Therefore we obtain
\begin{equation}
\begin{aligned}
\hat{\Gamma} &=\left(ss_0+ \hat A (\hat s_0 )\right) \log {e d} +\left(s+\hat s(\hat s_0)\right) \log \frac{e m}{s+\hat s(\hat s_0) }\\
&\le 5ss_0\log ed  + 10s \log (em/ s)\\ 
&\le 10\left( ss_0 \log {e d} + s \log \frac{e m}{s } \right),~ \forall \hat s_0 < s_0,
\end{aligned}
\end{equation}
which leads to \eqref{eq:pp}.

Overall, combining these 4 cases, we derive the upper bound for $\mathcal{L}$ as
 $$
 	\mathcal{L}^2 \precsim  \frac{\sigma^2\left(ss_0\log ed +s\log\frac{em}{s}\right)}{n},
 $$
 and by DSRIP condition, we finally complete the proof of Theorem \ref{adaptive2}.

\subsection*{Proof of Theorem \ref{T6}}

We use a strategy similar to Theorem \ref{th1} to prove these results. 
For ease to display, we define $\Upsilon(A, \tilde \beta^t): =  \sum_{(i,j) \in A} \langle \Phi_{ij}^\top, \beta^*-\tilde \beta^{t}\rangle^2 $. In specific, if $A \cup \text{supp}(\beta^*- \tilde \beta^t) \in \mathcal S^{m,d}(3s,\frac{5s_0}{3})$, by DSRIP$(3s, \frac53s_0, \delta)$ condition, we have $\Upsilon(A, \tilde \beta^t) \le \delta^2 \| \beta^*-\tilde \beta^{t} \|_2^2$.
In the proof of Theorem \ref{T6} and \ref{T7} (and also in Lemma \ref{OGChi2}-\ref{supporterror}), we use double index $(i,j)$ to denote the $i$-th entry (variable) of the $j$-th group $G_j$. 
Firstly, we provide the probability inequalities used frequently in this proof:
    \begin{align}
    \begin{split}\label{probT6}
    P\left\{ \forall S \in \mathcal S(s',s_0),~\|\Xi_{S } \|_2^2 > \frac{6\sigma^2 s's_0}{n} \Delta(s',s_0) \right\}
    & = o(1),~\text{ where } s' = \frac{s}{8\Delta^2};\\
    P\left\{ \sum_{(i,j) \in S_{G^*} } \Xi_{ij}^2 \mathrm I\Big\{ |\Xi_{ij}|\ge \tilde \tilde \lambda_1 \Big\} \ge \frac{ \sigma^2 ss_0}{n\Delta}  \right\}
    & = o(1); \\
    P \left( \sum_{(i,j) \in S^* } \tilde \lambda_2^2 \cdot \mathrm I\Big( |\Xi_{ij}|> \frac{\epsilon}2 \tilde \lambda_2 \Big)  
                     \ge \frac{\sigma^2ss_0}{n \Delta} \right)
    & = o(1); \\
    P \left\{ \sum_{j \in G^*} s_0 \tilde \lambda_2^2 \cdot 
                \mathrm I\left( \sum_{k \in S^* \cap S_{G_j}}\Xi_{kj}^2 > \frac{\epsilon^2}4 (s_j \vee s_0) \tilde \lambda_2^2 \right)
                \ge \frac{\sigma^2 ss_0}{n\Delta} \right\}
    & = o(1); \\ 
    P \left( \|\Xi_{S^*} \|_2^2 \ge \frac{2\sigma^2 ss_0}{n} \right) 
    &= o(1),
    \end{split}
    \end{align}
 as $\min\{ \Delta,~ {ss_0}/{\Delta}  \} \to \infty$. Define 
 $
  \Delta(s',s_0):= \frac{1}{s_0}\log\frac{em}{s'} + \log\frac{ed}{s_0}.
 $
 We provide the proof of the above inequalities in Appendix B.

Here we prove Theorem \ref{T6} by mathematical induction. From the assumption, the initial estimator $\tilde \beta^0 = \hat \beta$ is $(2s,\frac32 s_0)$-sparse and minimax optimal. It is easy to check that the three results in Theorem \ref{T6} hold for $t=0$. Now for $\forall t \ge 0$, assume the conclusions in Theorem \ref{T6} hold for the $t$-th iteration, we will prove these hold for the $(t+1)$-th iteration by contradiction and induction.

\subsubsection*{Step 1 (Control falsely discovered groups).} 

Assume that more than $s$ groups are falsely discovered in the $(t+1)$-th iteration. Then, we can always choose arbitrary $s$ falsely discovered groups and construct a $(s ,s_0)$-sparse set $S_{OG}' \in \tilde S^{t+1} \cap S_{G^*}^c$. The details of the selection process can be described as follows:

For any selected group $j \notin G^*$, if it has more than $s_0$ falsely discovered entries, then choose arbitrarily $s_0$ non-zero entries of these falsely discovered entries into $S_{OG}'$; if it has less than $s_0$, then we choose all these falsely discovered entries into $S_{OG}'$. We repeat this operation $s$ times for any $s$ falsely discovered groups, and we obtain a $(s, s_0)$-sparse set $S_{OG}'$.

Then, based on the definition of DSIHT operator $\mathcal{T}_{s_0, \tilde{\lambda}_2}(\cdot)$, for any falsely discovered group $j$ selected into set $S_{OG}'$, we have $\| \tilde \beta_{G_j \cap S_{OG}'}^{t+1} \|_2^2 \ge s_0 \tilde \lambda_{2}^2 $, which yields that
\begin{equation}\label{scaledabsurd}
    \begin{aligned}
    \sqrt{ ss_0 }\tilde \lambda_2 
    \le& \sqrt{ \Upsilon \left( S_{OG}', \tilde\beta^t \right) }
        + \sqrt{\sum_{(i,j) \in S_{OG}'} \Xi_{ij}^2 \mathrm I\Big\{ \mathcal T_{\lambda_{2},s_0} 
    \big(\tilde H^{t+1}\big)_{ij}\ne 0\Big\} }\\
    \overset{(i)}{\le} & \frac72 \delta \|\tilde \beta^t - \beta^* \|_2 + \sqrt{\frac{ \sigma^2 ss_0}{n\Delta} } ,
    \end{aligned}
\end{equation}
where inequality (i) follows Lemma \ref{OGChi2}. From the assumption of mathematical induction, since \eqref{scaledoracle} holds for $t$-iteration, we have $ \left\|\tilde \beta^{t} - \beta^* \right\|_2
            <~  16\sigma \sqrt{ \frac{ ss_0 \Delta}{n} } +   
                16 \sqrt{ \frac{\sigma^2 ss_0}{n} }$. Combining with \eqref{scaledabsurd}, we obtain
\begin{equation}\label{1stabsurd}
    \sqrt{ ss_0 }\tilde \lambda_2 = \sqrt{\frac{ 32\sigma^2 ss_0\Delta}{n} }
< 2.8 \sqrt{\frac{ \sigma^2 ss_0\Delta}{n} } + 3.8 \sqrt{ \frac{ \sigma^2 ss_0}{n} },
\end{equation}
which can not hold when $\Delta > 2.5$. Thus we find the absurd.

We have proved that no more than $s$ groups are falsely discovered in the $(t+1)$-th iteration. Next, we will prove that no more than $ss_0$ entries will be falsely discovered outside true groups $G^*$. If not so, we can construct an $(s,s_0)$-sparse set $S_{OG}'' \in \tilde S^t \cap S_{G^*}^c$. Then we obtain
\begin{equation}\label{step1.5}
    \begin{aligned}
    \sqrt{ ss_0 }\tilde \lambda_2 
    \le& \sqrt{ \Upsilon \left( S_{OG}'', \tilde\beta^t \right)}
        + \sqrt{\sum_{(i,j) \in S_{OG}''} \Xi_{ij}^2 \mathrm I\Big\{ \mathcal T_{\lambda_{2},s_0} 
    \big(\tilde H^{t+1}\big)_{ij}\ne 0\Big\} }\\
    \overset{(i)}{\le} & \frac72 \delta \|\tilde \beta^t - \beta^* \|_2 + \sqrt{\frac{ \sigma^2 ss_0}{n\Delta} } ,
    \end{aligned}
\end{equation}
where inequality (i) follows Lemma \ref{OGChi2}. This leads to a contradiction as \eqref{1stabsurd}.

\subsubsection*{Step 2 (Control falsely discovered entries in $S_{G^*}$).} 

Assume that there are more than $ss_0$ falsely discovered entries within the true groups $G^*$. Then, we can construct a $(s,s_0)$-sparse set $S_{IG} \in S_{G^*} \cap \tilde S^t \cap (S^*)^c$ such that for each entry in $S_{IG}$, $|\tilde \beta_{ij}^{t+1}| = |\Xi_{ij} +\langle\Phi_{ij}^\top, \beta^* - \tilde\beta^{t} \rangle|\ge \tilde \lambda_2$ always holds, which yields that
\begin{equation}\label{step2}
   \begin{aligned}
    \sqrt{ss_0}\tilde \lambda_2
    \le & \sqrt{ \Upsilon \left( S_{IG}, \tilde\beta^t \right) }
        +\sqrt{\sum_{(i,j)\in S_{IG}}\Xi_{ij}^2 \mathrm I\Big\{ |\Xi_{ij}+ \langle\Phi_{ij}^\top, \beta^* - \tilde\beta^{t} \rangle |\ge \tilde \lambda_2 \Big\} }\\
    \le & \delta \|\tilde \beta^t - \beta^* \|_2
        +\sqrt{ \sum_{(i,j) \in S_{G^*}}  \Xi_{ij}^2 \mathrm I\Big\{ |\Xi_{ij}|\ge \tilde \lambda_1 \Big\} }
        \\
        &+\sqrt{\sum_{(i,j)\in S_{IG}}\Xi_{ij}^2 \mathrm I\Big\{ |\Xi_{ij}|<\tilde \lambda_1 < |\langle\Phi_{ij}^\top, \beta^* - \tilde\beta^{t} \rangle| \Big\} }\\
    \overset{(i)}{<}& 2\delta \|\tilde \beta^t - \beta^* \|_2 + \sqrt{\frac{ \sigma^2 ss_0}{n\Delta} },
    \end{aligned}
\end{equation}
where inequality (i) follows Lemma \ref{IGChi2}. Since $\eqref{scaledoracle}$ holds for $t$-th iteration, it leads to a contradiction as \eqref{1stabsurd} again.

\subsubsection*{Step 3 ($\ell_2$ estimation error of $\tilde \beta^{t+1}$).} 

Now we have already proved the first two conclusions in Theorem \ref{T6} still hold in the $(t+1)$-th iteration, and then we will prove the third one also holds for $(t+1)$-th iteration. Note that
\begin{equation}
    \tilde \beta_{ij}^{t+1} - \beta_{ij}^* = - \tilde H_i^{t+1}  \cdot \mathrm I \left((i,j) \notin \tilde S^{t+1} \right) + \langle\Phi_{kj}^\top, \beta^* - \tilde\beta^{t} \rangle  + \Xi_{ij}
\end{equation}

We now focus on the estimation error on $S^*$ and $\tilde S^{t+1} \cap (S^*)^c$ respectively. On $S^*$, we have
\begin{equation}\label{S*}
    \begin{aligned}
    \|\tilde \beta_{S^*}^{t+1} - \beta_{S^*}^* \|_2 
    \le& \sqrt{ \sum_{(i,j) \in S^*} \left( \tilde H_{ij}^{t+1} \right)^2 \mathrm I \left( (i,j) \notin \tilde S^{t+1} \right)}
        + \sqrt{ \Upsilon \left( S^*, \tilde\beta^t \right) }
        +  \sqrt{ \sum_{(i,j) \in S^*} \Xi_{ij}^2 }\\
    \overset{(i)}{\le} & \frac4\epsilon \delta \left\| \tilde\beta^t - \beta^* \right\|_2 + 2\sqrt{\frac{\sigma^2ss_0}{n \Delta}} 
        + \sqrt{ \Upsilon \left( S^*, \tilde\beta^t \right) } 
        + \sqrt{ \sum_{(i,j) \in S^*} \Xi_{ij}^2 }\\
    \overset{(ii)}{\le} & \left( \frac4\epsilon+1 \right) \delta \left\| \tilde\beta^t -\beta^* \right\|_2
        + 2\sqrt{\frac{\sigma^2ss_0}{n \Delta}} + \sqrt{ \frac{2\sigma^2ss_0}{n } },
    \end{aligned}
\end{equation}
where inequality (i) uses the result of Lemma \ref{supporterror}.
Inequality (ii) uses Lemma \ref{lemma:iht1} and Theorem 2.1 in \citet{hsu2012tail}, that is, for a fixed set $S^* \in \mathcal S^{m,d}(s,s_0)$ and every $t>0$, we have
\begin{equation}
    P \left( \frac n{\sigma^2} \|\Xi_{S^*} \|_2^2 \ge ss_0 + 2(1+\delta)\sqrt{ss_0 t} + 2(1+ \delta)t \right) \le e^{-t}.
\end{equation}
Let $t = \frac{ss_0}{10}$. Based on $\delta < \frac15$, we obtain that $P \left( \|\Xi_{S^*} \|_2^2 \ge \frac{2\sigma^2 ss_0}n \right) \le \exp\left(-\frac{ss_0}{10}\right) = o(1)$ as $ss_0 \to \infty$. 

On $\tilde S^{t+1} \cap (S^*)^c$, when $X$ satisfies DSRIP$(3s,\frac53 s_0, \delta )$, we have 
\begin{equation}\label{exceptS*}
    \begin{aligned}
    &\left\|\tilde \beta_{\tilde S^{t+1} \cap (S^*)^c}^{t+1} - \beta_{\tilde S^{t+1} \cap (S^*)^c}^* \right\|_2 = \left\|\tilde \beta_{\tilde S^{t+1} \cap (S^*)^c}^{t+1} \right\|_2\\
    \le & \sqrt{ \Upsilon \left( \tilde S^{t+1} \cap (S^*)^c, \tilde \beta^t \right) }
         + \sqrt{\sum_{(i,j) \in \tilde S^{t+1} \cap (S^*)^c} \Xi_{ij}^2 \mathrm I \left( T_{\lambda_{2},s_0} \big(\tilde H^{t+1}\big)_{ij}\ne 0\right) }\\
    \le & \delta \left\| \tilde\beta^t -\beta^* \right\|_2 
         + \sqrt{\sum_{(i,j) \in S_{OG}} \Xi_{ij}^2 \mathrm I \left( T_{\lambda_{2},s_0} \big(\tilde H^{t+1}\big)_{ij}\ne 0\right) }
         + \sqrt{\sum_{(i,j) \in S_{IG}} \Xi_{ij}^2 \mathrm I \left( T_{\lambda_{2},s_0} \big(\tilde H^{t+1}\big)_{ij}\ne 0\right) }\\
    \overset{(i)}{<} & \frac92 \delta \left\| \tilde\beta^t -\beta^* \right\|_2 
         + 2\sqrt{\frac{ \sigma^2 ss_0}{n\Delta} } ,
    \end{aligned}
\end{equation}
where inequality (i) follows from Lemma \ref{OGChi2}, \eqref{step2} and Lemma \ref{IGChi2}.

Finally, based on \eqref{S*} and \eqref{exceptS*}, we have 
\begin{equation}
    \begin{aligned}
    \left\|\tilde \beta^{t+1} - \beta^* \right\|_2 
    \le &~ \left\|\tilde \beta_{S^*}^{t+1} - \beta_{S^*}^* \right\|_2 
         + \left\|\tilde\beta_{\tilde S^{t+1}\cap(S^*)^c}^{t+1} - \beta_{\tilde S^{t+1}\cap(S^*)^c}^* \right\|_2\\
    \le & \left( \frac4\epsilon+ \frac{11}{2} \right) \delta \left\| \tilde\beta^t -\beta^* \right\|_2  + 4\sqrt{\frac{\sigma^2ss_0}{n \Delta}} + \sqrt{ \frac{2\sigma^2ss_0}{n } }\\
    \overset{(i)}{<} & \left( \frac4\epsilon+ \frac{11}{2} \right) \delta \left\| \tilde\beta^t -\beta^* \right\|_2  + 4\sqrt{ \frac{\sigma^2ss_0}{n } },
    \end{aligned}
\end{equation}
where in inequality (i) we assume $\Delta>2.5$, which leads $\frac{4}{\sqrt{\Delta}} + \sqrt{2} < 4$.

Then, based on the initialized inequality $\left\| \tilde \beta^0 - \beta^* \right\|_2 \le 16 \sigma \sqrt{ \frac{ ss_0 \Delta}{n} }$ and $ \delta \le \epsilon^4 \wedge 0.05$, we have $\delta \left(\frac4\epsilon + \frac{11}{2} \right) \le \frac34$, which leads 
\begin{equation}
     \left\|\tilde \beta^{t+1} - \beta^* \right\|_2
<~  16\left(\frac34 \right)^{t+1} 
    \sigma \sqrt{ \frac{ ss_0 \Delta}{n} } +   
    16 \sqrt{ \frac{\sigma^2 ss_0}{n} }.
\end{equation}
Consequently, we prove that the conclusions in Theorem \ref{T6} hold for the $(t+1)$-th iteration, which completes the proof.

\subsection*{ proof of Theorem \ref{T7} }

Under the conditions of Theorem \ref{T6}, note that the probability inequalities in \eqref{probT6} still hold.

\subsubsection*{Step 1 (Sharp upper bound).} 

Let $t > 2 \log\left( 256\Delta \right) $ and we have
\begin{equation}
    16 \left(\frac34 \right)^{t} \sigma \sqrt{ \frac{ ss_0 \Delta}{n} }
    <  \sqrt{ \frac{\sigma^2 ss_0 }{n} }.
\end{equation}
From \eqref{scaledoracle}, we have
\begin{equation}
    \left\|\tilde \beta^{t} - \beta^* \right\|_2
    <~\sqrt{ \frac{\sigma^2 ss_0}{n} } + 16\sqrt{ \frac{\sigma^2 ss_0}{n} }
    = 17\sqrt{ \frac{\sigma^2 ss_0}{n} }.
\end{equation}

\subsubsection*{Step 2 (Group-wise almost full recovery).}

Note that based on the first conclusion of Theorem \ref{T6}, no more than $s$ groups are falsely discovered in the $(t+1)$-th iteration. Denote $G_{FD}^{t+1}$ as the falsely discovered group index set in the $(t+1)$-th iteration, which satisfies $ \left | G_{FD}^{t+1}  \right|< s$. Then, we have
\begin{equation}\label{AFRgroup}
    \begin{aligned}
     \| \tilde \eta_{G}^{t+1} - \eta^*_{G} \|_0
     =& \sum_{j=1}^m | (\tilde \eta^{t+1}_G )_j - (\eta^*_G)_j  | \\
       =& \sum_{j\in G^*}  | (\tilde \eta^{t+1}_G )_j - 1 | 
            +  \sum_{j \in G_{FD}^{t+1} } | (\tilde \eta^{t+1}_G )_j -0 |\\ 
       =& \sum_{j\in G^*} \mathrm I \left( \mathcal T_{\tilde\lambda_{2},s_0} \big(\tilde H^{t+1}_{G_j} \big) = \mathbf 0  \right)
           + \sum_{j \in  G_{FD}^{t+1} } \mathrm I \left( \mathcal T_{\tilde \lambda_2,s_0} 
            \big(\tilde H^{t+1}_{G_j} \big) \ne \mathbf 0  \right).
    \end{aligned}
\end{equation}
For the first term in \eqref{AFRgroup}, based on Lemma \ref{supporterror}, we have
\begin{equation}\label{AFRgroup1}
    \begin{aligned}
    &\sum_{j\in G^*} \mathrm I \left( \mathcal T_{\tilde \lambda_2,s_0} \big(\tilde H^{t+1}_{G_j} \big) = \mathbf 0  \right)\\
    \le & \sum_{j\in G^*} \mathrm I\left( \sum_{k \in S_{G_j} \cap S^* }\left( \tilde H_{kj}^{t+1}\right)^2 
         \mathrm I \left(|\tilde H_{kj}^{t+1}| \ge \tilde \lambda_2 \right)< s_0 \tilde \lambda_2^2 \right)\\
    \le & \sum_{j\in G^*} \mathrm I\left( \sum_{k \in S_{G_j} \cap S^* }\left( \tilde H_{kj}^{t+1}\right)^2 < (s_0+s_j) \tilde \lambda_2^2 \right)\\
    \overset{(i)}{\le} & \sum_{j\in G^*} \mathrm I \left( 
    \Upsilon \left( S_{G_j} \cap  S^*, \tilde\beta^t \right)    
    > \frac{\epsilon^2}4 (s_j \vee s_0) \tilde \lambda_2^2 \right)
         +\sum_{j\in G^*} \mathrm I \left( \sum_{k \in S_{G_j} \cap S^*} \Xi_{kj}^2 > \frac{\epsilon^2}4 (s_j \vee s_0) \tilde \lambda_2^2 \right)\\
    \overset{(ii)}{\le}  & \frac{4\delta^2 \left\|\tilde \beta^{t} - \beta^* \right\|_2^2}{\epsilon^2 s_0 \tilde \lambda_2^2}
         + \frac{\sigma^2 s}{ n\tilde \lambda_2^2 \Delta } 
    \\
    \precsim& \frac{s}{\Delta}+ \frac{s}{\Delta^2 } \\
    = & O\left(\frac{s}{\Delta}\right),~\text{ as } \Delta \to \infty,
    \end{aligned}
\end{equation}
where inequality (i) follows from \eqref{supportgrouptype2} and inequality (ii) follows from \eqref{xivee}.

For the second term in \eqref{AFRgroup}, based on Lemma \ref{OGChi2}, we have
\begin{equation}\label{AFRgroup2}
    \begin{aligned}
    & \sum_{j \in  G_{FD}^{t+1} } \mathrm I \left( \mathcal T_{\tilde \lambda_2,s_0} \big(\tilde H^{t+1}_{G_j} \big) \ne \mathbf 0 \right)\\
    \le & \sum_{j \in  G_{FD}^{t+1} } \mathrm I \left( \mathcal T_{\tilde \lambda_1,s_0} \big( \Xi_{\tilde S^{t+1} \cap S_{G_j}}\big) \ne \mathbf 0 \right)
         + \sum_{j \in  G_{FD}^{t+1} } \mathrm I \left( \mathcal T_{\tilde\lambda_{1},s_0} \big( \Xi_{\tilde S^{t+1} \cap  S_{G_j} }\big)=\mathbf 0,
         ~ \mathcal T_{\tilde \lambda_2,s_0} \big(\tilde H^{t+1}_{G_j} \big) \ne \mathbf 0 \right)\\
    \overset{(i)}{\le} & \frac{s}{8\Delta^2 } + 
         \sum_{j \in  G_{FD}^{t+1} } \mathrm I\left( \sum_{ k \in \tilde S^{t+1} \cap S_{G_j} } \Xi_{kj}^2 \mathrm I (|\Xi_{kj}|\ge \tilde \lambda_1)< s_0 \tilde \lambda_1^2,~
                 \mathcal T_{\tilde \lambda_2,s_0} \big(\tilde H^{t+1}_{G_j}\big)\ne \mathbf 0\right)\\
    \overset{(ii)}{\le} & \frac{s}{8\Delta^2 } + \sum_{j \in  G_{FD}^{t+1} } \mathrm I\left( s_0 \tilde \lambda_1^2 \le  2
    \Upsilon \left( \tilde S^{t+1} \cap S_{G_j}, \tilde \beta^t \right)
    \right)\\
    \precsim & \frac{s}{ \Delta^2} + \frac{s}{\Delta}  
    = O\left(\frac{s}{\Delta}\right),~\text{ as } \Delta \to \infty,
    \end{aligned}
\end{equation}
where inequality (i) follows from a similar contradiction in the proof of the first term in Lemma \ref{OGChi2}, and inequality (ii) follows from the result of \eqref{OGgroupscaled}.

Combining \eqref{AFRgroup}, \eqref{AFRgroup1} and \eqref{AFRgroup2} together, we prove that $\| \tilde \eta_{G}^{t+1} - \eta^*_{G} \|_0 = O\left(\frac{s}{\Delta}\right)$.

\subsubsection*{Step 3 (Element-wise almost full recovery).}

Based on the first two conclusions of Theorem \ref{T6}, we have
\begin{equation}\label{AFRentry}
    \begin{aligned}
     \| \tilde \eta^{t+1} - \eta^* \|_0
    =& \sum_{j=1}^m \sum_{i=1}^d | \tilde \eta_{ij}^{t+1} - \eta_{ij}^* | \\
    = & \sum_{(i,j)\in S^*} | \tilde \eta_{ij}^{t+1} -1| 
         + \sum_{(i,j)\in S_{G^*} \cap (S^*)^c \cap \tilde S^{t+1}} | \tilde \eta_{ij}^{t+1} -0| 
         + \sum_{(i,j)\in S_{G^*}^c \cap \tilde S^{t+1}} | \tilde \eta_{ij}^{t+1} -0|. 
    \end{aligned}
\end{equation}
We can just analyze these three terms respectively. For the first one, note that
\begin{equation}\label{AFRentry1}
    \begin{aligned}
    \sum_{(i,j)\in S^*} | \tilde \eta_{ij}^{t+1} -1|
    =& \sum_{(i,j)\in S^*} \mathrm I\left( (i,j) \notin \tilde S^{t+1}\right) \\
    \overset{(i)}{\le}&  \sum_{(i,j)\in S^*} \mathrm I\left( |\tilde H_{ij}^{t+1}| < \tilde \lambda_2 \right)\\
        & + \frac1{\tilde \lambda_2^2} \sum_{(i,j)\in S^*} \left(\tilde H_{ij}^{t+1} \right)^2 \mathrm I\left( |\tilde H_{ij}^{t+1}| \ge \tilde \lambda_2, ~
           \sum_{k \in S_{G_j} \cap S^*} \left( \tilde H_{kj}^{t+1}\right)^2  < (s_0+ s_j) \tilde \lambda_2^2 \right)\\
    \overset{(ii)}{\le} & \frac{16\delta^2 \left\|\tilde \beta^{t} - \beta^* \right\|_2^2}{\epsilon^2 \tilde \lambda_2^2}
         + \frac{4\sigma^2 ss_0}{ n\tilde \lambda_2^2 \Delta } 
    \precsim ~ \frac{ss_0}{\Delta} + \frac{ss_0}{\Delta^2 }\\
    =& O\left(\frac{ss_0}{\Delta}\right),~\text{ as } \Delta \to \infty,
    \end{aligned}
\end{equation}
where inequality (i) follows from the first inequality of \eqref{supportH} in Lemma \ref{supporterror}, and inequality (ii) follows from the last inequality of \eqref{supportH}, \eqref{supportH1} and \eqref{supportH2}.

For the second term, we obtain
\begin{equation}\label{AFRentry2}
    \begin{aligned}
    \sum_{(i,j)\in S_{G^*} \cap (S^*)^c \cap \tilde S^{t+1} } | \tilde \eta_{ij}^{t+1} -0|
         \le& \sum_{(i,j)\in S_{G^*} \cap (S^*)^c \cap \tilde S^{t+1}} \mathrm I \left( \left|\tilde H_{ij}^{t+1}\right| \ge \tilde \lambda_2 \right)\\
    \overset{(i)}{\le} & \sum_{(i,j)\in S_{G^*} } \mathrm I \left( \left|\Xi_{ij} \right| \ge \tilde\lambda_1 \right) 
         \\
         &+\sum_{(i,j)\in S_{G^*} \cap (S^*)^c \cap \tilde S^{t+1} } \mathrm I \big( |\Xi_{ij}| < \tilde\lambda_1 < |\langle\Phi_{ij}^\top, \beta^* - \tilde\beta^{t} \rangle| \big) \\
    \overset{(ii)}{\le} &~  \frac{\sigma^2 ss_0}{n \tilde \lambda_1^2 \Delta}  + \frac{\delta^2 \left\|\tilde \beta^{t} - \beta^* \right\|_2^2}{\tilde\lambda_1^2}
    \precsim ~ \frac{ss_0}{\Delta^2(s,s_0)} + \frac{ss_0}{\Delta}\\
    =& O\left(\frac{ss_0}{\Delta}\right),~\text{ as } \Delta \to \infty,
    \end{aligned}
\end{equation}
where inequality (i) follows from \eqref{step2} in Theorem \ref{T6}, inequality (ii) follows from the probability inequality \eqref{IGFDentry} in Lemma \ref{IGChi2} and $\sum_{(i,j) \in S_{G^*} } \tilde \lambda_1^2 \mathrm I\Big\{ |\Xi_{ij}|\ge \tilde \lambda_1 \Big\} \le \sum_{(i,j) \in S_{G^*} } \Xi_{ij}^2 \mathrm I\Big\{ |\Xi_{ij}|\ge \tilde \lambda_1 \Big\}$.

For the third term, we obtain
\begin{equation}\label{AFRentry3}
    \begin{aligned}
    & \sum_{(i,j)\in S_{G^*}^c \cap \tilde S^{t+1}} | \tilde \eta_{ij}^{t+1} -0|
     = \sum_{(i,j)\in S_{G^*}^c \cap \tilde S^{t+1} } 
         \mathrm I \left( \mathcal T_{\tilde \lambda_2,s_0} \big(\tilde H^{t+1}\big)_{ij}\ne 0 \right)\\
    \overset{(i)}{\le} & \sum_{(i,j)\in S_{G^*}^c \cap \tilde S^{t+1} } \mathrm I\Big\{\mathcal T_{\tilde \lambda_1,s_0} 
                \big(\Xi_{S_{OG}}\big)_{ij}\ne 0 \Big\} 
         +\sum_{(i,j)\in S_{G^*}^c \cap \tilde S^{t+1} } \mathrm I\Big\{|\Xi_{ij}|< \tilde \lambda_1,
                 ~ |\Xi_{ij}+ \langle\Phi_{ij}^\top, \beta^* - \tilde\beta^{t} \rangle | \ge \tilde \lambda_2 \Big\}  \\ 
         & ~~+\frac1{\tilde \lambda_1^2} \sum_{(i,j)\in S_{G^*}^c \cap \tilde S^{t+1} }\Xi_{ij}^2\mathrm I\left\{|\Xi_{ij}|\ge \tilde \lambda_1,
                ~ \sum_{k \in S_{G_j} \cap \tilde S^{t+1} } \Xi_{kj}^2 \mathrm I (|\Xi_{kj}|\ge \tilde \lambda_1)< s_0\tilde \lambda_1^2,~
                  \mathcal T_{\tilde \lambda_2,s_0} \big(\tilde H^{t+1}\big)_{ij}\ne 0\right\}  \\
    \overset{(ii)}{<} & s's_0 + \frac{\delta^2 \left\|\tilde \beta^{t} - \beta^* \right\|_2^2}{\tilde \lambda_1^2}
         + \frac{2\delta^2 \left\|\tilde \beta^{t} - \beta^* \right\|_2^2}{\tilde \lambda_1^2}\\
    \precsim &  \frac{ss_0}{\Delta^2 } + \frac{ss_0}{\Delta} 
    = O\left(\frac{ss_0}{\Delta}\right),~\text{ as } \Delta \to \infty,
    \end{aligned}
\end{equation}
where inequality (i) follows from \eqref{ogxi} in Lemma \ref{OGChi2}, and inequality (ii) follows from the framework of the first term in Lemma \ref{OGChi2}, \eqref{OGgroupscaled} and \eqref{GroupXifailHsuccess}, and recall $s' = \frac{ s}{8\Delta^2}$.

Combining \eqref{AFRentry}, \eqref{AFRentry1}, \eqref{AFRentry2} and \eqref{AFRentry3} together, we prove that $\| \tilde \eta^{t+1} - \eta^* \|_0 = O\left(\frac{ss_0}{\Delta}\right)$.

\section*{Appendix B : Auxiliary lemmas}

\begin{lemma}\label{lem:sigma}
Assume that $X$ satisfies $\mbox{DSRIP}(2s, \frac{3}{2}s_0,\frac{\delta}{2})$. 
Then, with probability at least  $1-\exp\left\{-Css_0\Delta\right\}$, we have
\begin{equation*}
  | \sigma_t - \sigma| \leq \sqrt{1+\delta}\|\beta^* - \beta^t\|_2+\frac{1}{20}\sigma.
\end{equation*}
\end{lemma}

\begin{proof}

  Denote event  $\mathcal{A} = \{|\frac{\|\xi\|_2}{\sigma}-\sqrt{n}| \leq \frac{1}{20}\sqrt{n}\}$.
  From Hanson-Wright inequality \citep{10.1214/ECP.v18-2865}, it holds that $P(\mathcal{A}) \geq 1-e^{-Cn}\geq 1 - e^{-C ss_0\D}$.
  Therefore, 
  \begin{align*}
    | \sigma_t - \sigma|&\leq| \sigma_t - \frac{\|\xi\|_2}{\sqrt{n}}|+|\frac{\|\xi\|_2}{\sqrt{n}}-\sigma|\\
    &\le \frac{1}{\sqrt{n}}\left|\|X(\beta^* - \beta^t)+ \xi\|_2-\|\xi\|_2\right| + \frac{1}{20}\sigma\\
    &\le \sqrt{1+\delta}\|\beta^* - \beta^t\|_2 + \frac{1}{20}\sigma,
  \end{align*}
  where the second inequality follows from event $\mathcal{A}$, and the last inequality follows from DSRIP condition.
    
\end{proof}

To control the inner product between $\xi$ and $X\left(\hat \beta - \beta^* \right)$, we provide a useful lemma.

\begin{lemma}\label{lem:inner}
Given integers $v_1, v_2 > 0$,    and assume that $\beta$ is a $(v_1,v_2/v_1)$-sparse vector. we have 
  \begin{equation}\label{orthogonal}
    P\left(\sup\limits_{\beta \in \Theta^{m,d}(v_1, \frac{v_2}{v_1})}
    \left|	\left\langle\xi,\frac{X\beta}{\|X\beta\|_2}\right\rangle\right|^2
    \ge 3 \sigma^2(v_1 \log \frac{em}{v_1}+v_2 \log\frac{ed v_1}{v_2}) \right) 
    \le  e^{-C\left(v_1 \log \frac{em}{v_1}+v_2 \log\frac{ed v_1}{v_2}\right)}.
\end{equation}
In specific, if $\beta^*$ is a $(s, s_0)$-sparse vector and $\hat{\beta} \in \Theta^{m,d}(\hat{s}, \frac{\hat{A}}{\hat{s}})$, i.e., $\|\hat{\beta}\|_{0} \le \hat{A}$ and $\|\hat{\beta}\|_{0,2} \le \hat{s}$, 
we have

\begin{equation}\label{orthogonal2}
        P\left(\sup\limits_{\hat{\beta} \in \Theta^{m,d}(\hat{s}, \frac{\hat{A}}{\hat{s}})}
    \left|	\left\langle\xi,\frac{X(\hat{\beta}-\beta^*)}{\|X(\hat{\beta}-\beta^*)\|_2}\right\rangle\right|^2
    \ge 3 \sigma^2\Omega^*(\hat \beta) \right) 
    \le  e^{-C\Omega^*(\hat\beta)},
\end{equation}
where $\Omega^*(\hat\beta)$ is defined at the beginning of the Appendix.
\end{lemma}

\begin{proof}
For a fixed set $S$ satisfies $\text{supp}(\beta) \subseteq S$, denote $X_S$ as the span space of columns of $X$ indexed by $S$, thus we have $\langle \xi, X\beta \rangle = \langle \xi, X_S\beta_S \rangle$. Denote $\pi_S = X_S(X_S^{\top}X_S)^{-1}X_S^{\top} \in \mathbb R^{n \times n}$, which is an orthogonal matrix of rank no more than $|S|$.  Therefore, for $\forall \beta \in  \Theta^{m,d}(v_1, \frac{v_2}{v_1})$, we obtain the following by Cauchy-Schwartz inequality:
  
  \begin{align}\label{innereq1}
    \begin{split}
    \left|	\left\langle\xi,\frac{X\beta}{\|X\beta\|_2}\right\rangle\right| 
 =      \left|	\left\langle \pi_S\xi,\frac{X_S\beta_S}{\|X_S\beta_S\|_2}\right\rangle\right| \le \left\| \pi_S \xi \right\|_2
  \le \sup\limits_{S \in \mathcal S^{m,d}\left(v_1, \frac{v_2}{v_1}\right)} \left\| \pi_S \xi \right\|_2,
  \end{split}
  \end{align}
Note that for $\forall S \in \mathcal S^{m,d}\left(v_1, \frac{v_2}{v_1}\right)$, we have $rank\big(\pi_S \big) \le  v_2$, so that $Tr(\pi_S) \le rank\big(\pi_S \big) \cdot \|\pi_S\|_2 \le v_2$. Thus by Theorem 2.1 of \citet{hsu2012tail}, for $\forall t >0$, we have
  \begin{equation}\label{innereq2}
    P\left(\frac{\|\pi_S \xi\|_2^2}{\sigma^2} \geq \frac{5}{2}t\right)
    \leq P\left(\frac{\|\pi_S \xi\|_2^2}{\sigma^2} \geq v_2 + 2\sqrt{ v_2 t} + 2t\right)
    \leq  e^{-t},
  \end{equation}
  where the first inequality holds when $t \gg  v_2$.

  Similarly to \eqref{lem1eq6}, we have $\left| S^{m,d}\left(v_1, \frac{v_2}{v_1}\right) \right| \le \left( \frac{em}{ v_1} \right)^{v_1} \times \left(\frac{edv_1}{v_2} \right)^{v_2}$, thus by \eqref{innereq2} we get a union bound as:
$$
 P\left( \sup\limits_{S \in \mathcal S^{m,d}\left(v_1, \frac{v_2}{v_1}\right)} \left\| \pi_S \xi \right\|_2^2 \ge 3 \sigma^2 \left(v_1 \log \frac{em}{v_1}+v_2 \log\frac{ed v_1}{v_2}\right) \right)
 \leq  e^{-C\left(v_1 \log \frac{em}{v_1}+v_2 \log\frac{ed v_1}{v_2}\right)}.
$$
Let $t = (1+C)(v_1\log \frac{em}{v_2}+v_1+\log \frac{e d v_1}{v_2})$ for some constant $0<C<\frac{1}{5}$, which satisfies $t \gg  v_2$. We complete \eqref{orthogonal}. 

For \eqref{orthogonal2}, for any $\hat{\beta} \in \Theta^{m,d}(\hat{s}, \frac{\hat{A}}{\hat{s}})$, combined with $\beta^* \in \Theta^{m,d}(s, s_0)$, so we have:
$$
\hat{\beta} -\beta^* \in \Theta^{m,d}\left(\hat{s}+s, \frac{ss_0+\hat{A}}{\hat{s}+s}\right).
$$
We let $v_1 = \hat{s}+s$ and $v_2 = ss_0+\hat{A}$, we obtain the \eqref{orthogonal2} directly by \eqref{orthogonal}.
\end{proof}

For ease of display, in the next three lemmas, we use double index $(i,j)$ to denote the $i$-th entry (variable) in the $j$-th group $G_j$.
Besides, we recall the abbreviation $\Upsilon(A, \tilde \beta^t) =  \sum_{(i,j) \in A} \langle \Phi_{ij}^\top, \beta^*-\tilde \beta^{t}\rangle^2 $, $\Delta = \frac{1}{s_0}\log(em/s)+ \log(ed/s_0)$ and $\tilde\lambda_a = a  \sqrt{ \frac{8\sigma^2}{n} \left(\log \frac{ed}{s_0}+\frac{1}{s_0}\log \frac{em}{s} \right)  }$. Denote $S_{OG}:= \tilde S^{t+1} \cap S_{G^*}^c$

Firstly, to bound the $\ell_2$ norm of the selected entries of $\Xi$ in $S_{OG}$, we give the following lemma.
\begin{lemma}\label{OGChi2}
Assume all the conditions in Theorem \ref{T6} hold. For $\forall t \ge 0$, as $\Delta, \frac{ss_0}{\Delta} \to \infty$, we have 
\begin{equation} 
    P\left( \sqrt{ \sum_{(i,j) \in S_{OG}} \Xi_i^2 \mathrm I\Big\{ \mathcal T_{\tilde \lambda_2,s_0} \big(\tilde H^{t+1}\big)_{ij}\ne 0\Big\} }
    <  \sqrt{\frac{ \sigma^2 ss_0}{n\Delta} } + \frac52 \delta \|\tilde \beta^t - \beta^* \|_2 \right) \to 1.
\end{equation}
\end{lemma}
\begin{proof}
    Note that 
    \begin{equation}\label{ogxi}
        \begin{aligned}
         & \sqrt{ \sum_{(i,j) \in S_{OG}} \Xi_{ij}^2 \mathrm I\Big\{ \mathcal T_{\tilde \lambda_2,s_0} 
        \big(\tilde H^{t+1}\big)_{ij}\ne 0\Big\} } \\
        \le & \sqrt{ \sum_{(i,j) \in S_{OG}} \Xi_{ij}^2 \mathrm I\Big\{\mathcal T_{\tilde \lambda_{1},s_0} 
                    \big(\Xi_{S_{OG}}\big)_{ij}\ne 0,~ \mathcal T_{\tilde \lambda_2,s_0} \big(\tilde H^{t+1}\big)_{ij}\ne 0\Big\} }\\
        &~~+ \sqrt{ \sum_{(i,j) \in S_{OG}}  \Xi_{ij}^2 \mathrm I\Big\{\mathcal T_{\tilde \lambda_1,s_0} 
            \big(\Xi_{S_{OG}}\big)_{ij}= 0,~ \mathcal T_{\tilde \lambda_2,s_0} \big(\tilde H^{t+1}\big)_{ij}\ne 0\Big\} }\\
        \le & \sqrt{ \sum_{(i,j) \in S_{OG}} \Xi_{ij}^2 \mathrm I\Big\{\mathcal T_{\tilde \lambda_1,s_0}  \big(\Xi_{S_{OG}}\big)_{ij}\ne 0,~ \mathcal T_{\tilde \lambda_2,s_0} \big(\tilde H^{t+1}\big)_{ij}\ne 0\Big\} }\\
         & ~~  + \sqrt{ \sum_{(i,j) \in S_{OG}} \Xi_{ij}^2\mathrm I\Big\{|\Xi_{ij}|< \tilde \lambda_1,
             ~ \mathcal T_{\tilde \lambda_2,s_0} \big(\tilde H^{t+1}\big)_{ij}\ne 0\Big\} }\\ 
         & ~~  + \sqrt{ \sum_{(i,j) \in S_{OG}} \Xi_{ij}^2\mathrm I\left\{|\Xi_{ij}|\ge \tilde \lambda_1,
            ~ \sum_{k \in S_{G_j} \cap S_{OG}} \Xi_{kj}^2 \mathrm I (|\Xi_{kj}|\ge \tilde \lambda_1)< s_0\tilde \lambda_1^2,~
              \mathcal T_{\tilde \lambda_2,s_0} \big(\tilde H^{t+1}\big)_{ij}\ne 0\right\} }\\ 
        \le & \sqrt{ \sum_{(i,j) \in S_{OG}}  \Xi_{ij}^2 \mathrm I\Big\{\mathcal T_{\tilde \lambda_1,s_0} 
            \big(\Xi_{S_{OG}}\big)_{ij}\ne 0 \Big\} } 
        + \sqrt{ \sum_{(i,j) \in S_{OG}}  \Xi_{ij}^2\mathrm I\Big\{|\Xi_{ij}|< \tilde \lambda_1,
             ~ |\Xi_{ij}+ \langle\Phi_{ij}^\top, \beta^* - \tilde\beta^{t} \rangle| \ge \tilde \lambda_2 \Big\} }\\ 
         & ~~  + \sqrt{\sum_{j \in \tilde G^{t+1} \cap (G^*)^c} s_0 \tilde \lambda_1^2 \cdot \mathrm I\left\{ 
             \sum_{k \in S_{G_j} \cap S_{OG}} \Xi_{kj}^2 \mathrm I (|\Xi_{kj}|\ge \tilde \lambda_1)< s_0\tilde \lambda_1^2,~
             \mathcal T_{\tilde \lambda_2,s_0} \big(\tilde H^{t+1}\big)_{G_j}\ne 0 \right\} }.
        \end{aligned}
    \end{equation}

     Next, we bound the three terms in the last inequality respectively. 

    \textbf{First term.}
    Let $s' = \frac{ s}{8\Delta^2}$, $s_0'=s_0$. Then, we show that under the event $\mathcal E(s',s_0)$ in Lemma \ref{lemma:iht1}, only less than $s'$ groups in $\Xi_{S_{OG}}$ could be discovered by $ \mathcal T_{\tilde \lambda_1,s_0}$. If not so, choose any $s'$ discovered groups and construct an $S' \subset S_{OG}$ and $S' \in \mathcal S^{m,d}(s',s_0)$, which satisfies
    $$
        \sum_{(i,j) \in S'} \Xi_{ij}^2  
        \ge \sum_{(i,j) \in S'} \Xi_{ij}^2 \mathrm I\Big\{\mathcal T_{\tilde \lambda_1,s_0} \big(\Xi_{S_{OG}}\big)_{ij}\ne 0 \Big\} 
        \ge~  s's_0\tilde \lambda_1^2\\
        \ge~ \frac{ s}{8\Delta^2}\cdot s_0 \cdot  \frac{8\sigma^2}{n}\Delta 
        = \frac{ ss_0\sigma^2}{n\Delta}.
    $$
When $\Delta$ is sufficiently large, we can show that $ \log\left( 8 \Delta^2\right) < \frac{s_0}3 \Delta$, which leads that 
    $$
    \Delta< \Delta(s',s_0):= \frac{1}{s_0}\log\frac{em}{s'} + \log\frac{ed}{s_0} < \frac43 \Delta.
    $$
    Thus we have 
    $$
        \sum_{(i,j) \in S'} \Xi_{ij}^2 
        \ge \frac{ ss_0\sigma^2}{n\Delta}
        =\frac{ 8 s's_0\sigma^2 \Delta}{n} 
        > \frac{8 s's_0\sigma^2 }{n}\cdot \frac34 \Delta(s',s_0)
        =\frac{6s's_0\sigma^2\Delta(s',s_0)}{n},
    $$
    which contradicts the event $\mathcal E(s',s_0)$ in Lemma \ref{lemma:iht1} with high probability. Thus we show only less than $s'$ groups in $\Xi_{S_{OG}}$ are discovered. Similarly, we can show only less than $s's_0$ entries are discovered in $\Xi_{S_{OG}}$. If not so, take $S_2 \in \mathcal S^{m,d}(s',s_0)$ and $S_2 \subset S_{OG}$, whose entries are all falsely discovered in $S_{OG}$, which leads
    \begin{equation}\label{entryOGabsurd}
        \sum_{(i,j) \in S_2} \Xi_{ij}^2 
        \ge s's_0 \tilde \lambda_1^2 
        \ge \frac{ ss_0\sigma^2}{n\Delta}
        \ge \frac{6s's_0\sigma^2\Delta(s',s_0)}{n}.
    \end{equation}
    Under the event $\mathcal E(s',s_0)$ in Lemma \ref{lemma:iht1}, \eqref{entryOGabsurd} leads to an absurd again. Thus we can bound the first term in \eqref{ogxi} by
    \begin{equation}\label{Xisucess}
        \sum_{(i,j) \in S_{OG}} \Xi_{ij}^2 \mathrm I\Big\{\mathcal T_{\tilde \lambda_1,s_0} 
                    \big(\Xi_{S_{OG}}\big)_{ij}\ne 0 \Big\} 
        \le  \sup_{S_2 \in \mathcal S(s',s_0)} \sum_{(i,j)\in S_2} \Xi_{ij}^2 
        \le   \frac{6\sigma^2 s' s_0 \Delta(s',s_0)}{n} 
        \le \frac{ \sigma^2 s s_0 }{n\Delta}.
    \end{equation}

    \textbf{Second term.} 
    Note that 
    \begin{equation}
        \begin{aligned}
        & \sum_{(i,j) \in S_{OG}} \Xi_{ij}^2\mathrm I\Big\{|\Xi_{ij}|< \tilde \lambda_1,
                     ~ |\Xi_{ij}+ \langle\Phi_{ij}^\top, \beta^* - \tilde\beta^{t} \rangle | \ge \tilde \lambda_2 \Big\} \\
        \le & \sum_{(i,j) \in S_{OG}} \Xi_{ij}^2\mathrm I\Big\{|\Xi_{ij}|< \tilde \lambda_1,
                     ~ |\Xi_{ij}| + |\langle\Phi_{ij}^\top, \beta^* - \tilde\beta^{t} \rangle| \ge 2\tilde \lambda_1 \Big\}\\
        \le & \sum_{(i,j) \in S_{OG}} \Xi_{ij}^2\mathrm I\Big\{|\Xi_{ij}|< \tilde \lambda_1 \le|\langle\Phi_{ij}^\top, \beta^* - \tilde\beta^{t} \rangle|  \Big\}\\
        \le & \sum_{(i,j) \in S_{OG} } \langle \Phi_{(i,j)}^\top, \beta^*-\tilde \beta^t \rangle^2
        = \left\| \Phi_{S_{OG}}  \left( \beta^*-\tilde \beta^t \right) \right\|_2^2.
        \end{aligned}
    \end{equation}
    Thus we can bound the second term in \eqref{ogxi} by
    \begin{equation}\label{ElementXifailHsuccess}
        \sqrt{\sum_{(i,j) \in S_{OG}}\Xi_{ij}^2\mathrm I\Big\{|\Xi_{ij}|< \tilde \lambda_1,
             ~ |\Xi_{ij}+ \langle\Phi_{ij}^\top, \beta^* - \tilde\beta^{t} \rangle| \ge \tilde \lambda_2 \Big\} }
        \le \left\| \Phi_{S_{OG}}  \left( \beta^*-\tilde \beta^t \right) \right\|_2
        \le \delta \left\| \tilde \beta^t - \beta^* \right\|_2.
    \end{equation}

    \textbf{Third term.} 
    For any group $j \notin G^*$ such that $\sum_{k \in S_{G_j}\cap S_{OG}} \Xi_{kj}^2 \mathrm I (|\Xi_{kj}|\ge \tilde \lambda_1)< s_0\tilde \lambda_1^2$ and $\mathcal T_{\tilde \lambda_2,s_0} \big(\tilde H^{t+1}\big)_{G_j} \ne \mathbf 0$ (where the index ranges over $S_{OG}$), we have 
    \begin{equation}\label{OGgroupscaled}
        \begin{aligned}
        s_0\tilde \lambda_2^2  
        \le& \sum_{k \in S_{G_j} \cap S_{OG}} \Big( \underbrace{\Xi_{kj} + \langle\Phi_{kj}^\top, \beta^* - \tilde\beta^{t} \rangle}_{\tilde H_{kj}^{t+1}} \Big)^2 
            \mathrm I\Big( |\Xi_{kj} + \langle\Phi_{kj}^\top, \beta^* - \tilde\beta^{t} \rangle | \ge \tilde \lambda_2 \Big)\\
        \le& \sum_{k \in S_{G_j} \cap S_{OG}} 2 \Xi_{kj}^2 \mathrm I\Big( |\tilde H_{kj}^{t+1}|\ge\tilde \lambda_2\Big)
            +\sum_{k \in S_{G_j} \cap S_{OG}} 2 \langle\Phi_{kj}^\top, \beta^* - \tilde\beta^{t} \rangle^2 \mathrm I\Big( |\tilde H_{kj}^{t+1}|\ge\tilde \lambda_2\Big)\\
        \le& \sum_{k \in S_{G_j} \cap S_{OG}} 2 \Xi_{kj}^2 \mathrm I\Big( |\Xi_{kj}|\ge\tilde \lambda_1\Big)
            + \sum_{k \in S_{G_j} \cap S_{OG}} 2 \Xi_{kj}^2 \mathrm I\Big( |\Xi_{kj}|<\tilde \lambda_1 \le |\langle\Phi_{kj}^\top, \beta^* - \tilde\beta^{t} \rangle|\Big)\\
            &+  2 \sum_{k \in S_{G_j} \cap S_{OG}} \langle\Phi_{kj}^\top, \beta^* - \tilde\beta^{t} \rangle^2\\
        \le & 2s_0 \tilde \lambda_1^2 + 4\Upsilon\left(S_{G_j} \cap S_{OG}, \tilde \beta^t \right),
        \end{aligned}
    \end{equation}
    which leads to $s_0 \tilde \lambda_1^2 \le  2\Upsilon\left(S_{G_j} \cap S_{OG}, \tilde \beta^t \right)$. Thus we can bound the third term in \eqref{ogxi} as
    \begin{equation}\label{GroupXifailHsuccess}
        \begin{aligned}
        & \sqrt{\sum_{(i,j) \in S_{OG}} \Xi_{ij}^2\mathrm I\Big\{|\Xi_{ij}|\ge \tilde \lambda_1,
                    ~ \sum_{k \in S_{G_j} \cap S_{OG}} \Xi_{kj}^2 \mathrm I (|\Xi_{kj}|\ge \tilde \lambda_1)< s_0\tilde \lambda_1^2,~
                     \mathcal T_{\tilde \lambda_2,s_0} \big(\tilde H^{t+1}\big)_{ij}\ne 0 \Big\} }\\
        \le & \sqrt{\sum_{j \in \tilde G^{t+1} \cap (G^*)^c} s_0\tilde \lambda_1^2 \mathrm I \Big\{
                     s_0 \tilde \lambda_1^2 \le  2\Upsilon\left(S_{G_j} \cap S_{OG}, \tilde \beta^t \right) \Big\} }\\
        \le & \sqrt{2 \Upsilon\left(S_{OG}, \tilde \beta^t \right) }\\
        < & \frac32\delta \left\| \tilde \beta^t - \beta^* \right\|_2.
        \end{aligned}
    \end{equation}

    Combining these three terms \eqref{Xisucess}, \eqref{ElementXifailHsuccess} and \eqref{GroupXifailHsuccess} together, we finally get that 
    \begin{equation}
        P\left( \sqrt{\sum_{(i,j) \in S_{OG}} \Xi_i^2 \mathrm I\Big\{ \mathcal T_{\tilde \lambda_2,s_0} 
        \big(\tilde H^{t+1}\big)_{ij}\ne 0\Big\} }
        <  \sqrt{\frac{ \sigma^2 ss_0}{n\Delta} } + \frac52 \delta \|\tilde \beta^t - \beta^* \|_2 \right) \to 1,
    \end{equation}
    as $\Delta, \frac{ss_0}{\Delta} \to \infty$.

\end{proof}

Similarly, we can bound the $\ell_2$-norm of the selected entries of $\Xi$ within the true groups $G^*$, which can be expressed in the following lemma.
\begin{lemma}\label{IGChi2}
Assume all the conditions in Theorem \ref{T6} hold. As ${\Delta} \to \infty$, we have 
\begin{equation} 
    P\left( \sqrt{\sum_{(i,j) \in S_{G^*} } \Xi_{ij}^2 \mathrm I\Big\{ |\Xi_{ij}|\ge \tilde \lambda_1 \Big\} }
    <  \sqrt{\frac{ \sigma^2 ss_0}{n\Delta} } \right) \to 1.
\end{equation}
\end{lemma}

 \begin{proof}
   Since for any $ (i,j) \in S_{G^*} $, $\Xi_{ij}$ is sub-Gaussian with parameter $\frac{\sigma^2}{n}$, we conclude that 
    \begin{equation}
        \begin{aligned}
        \mathbf E \Big( \Xi_{ij}^2 \mathrm I \left\{ |\Xi_{ij}|\ge \tilde \lambda_1 \right\} \Big)
        =& \int_0^\infty P\Big( \Xi_{ij}^2 \mathrm I \left\{ |\Xi_{ij}|\ge \tilde \lambda_1 \right\} >u \Big) \mathrm d u\\
        =& \int_0^{\tilde \lambda_1^2} P\Big( |\Xi_{ij}|\ge \tilde \lambda_1 \Big) \mathrm d u
             + \int_{\tilde \lambda_1^2}^\infty P\Big( |\Xi_{ij}|\ge \sqrt u \Big) \mathrm d u\\
        \le& 2\tilde \lambda_1^2 \exp\left( -\frac{n\tilde \lambda_1^2}{2\sigma^2} \right) +
             \int_{\tilde \lambda_1^2}^\infty 2 \exp\left( -\frac{nu}{2\sigma^2} \right)\mathrm d u\\
        =& \left( 2\tilde \lambda_1^2 + \frac{4\sigma^2}{n} \right) \exp\left( -\frac{n\tilde \lambda_1^2}{2\sigma^2} \right)\\
        \le& 3\tilde \lambda_1^2 \exp\left( -\frac{n\tilde \lambda_1^2}{2\sigma^2} \right),
        \end{aligned}
    \end{equation}
    where the last inequality follows from $\tilde\lambda_1^2 =   \frac{8\sigma^2}{n} \Delta \ge \frac{4\sigma^2}{n}$.
    Thus, based on Markov inequality we have 
    \begin{equation}\label{IGFDentry}
        \begin{aligned}
        P\left(\sum_{(i,j) \in S_{G^*} } \Xi_{ij}^2 \mathrm I\Big\{ |\Xi_{ij}|\ge \tilde \lambda_1 \Big\} \ge \frac{ \sigma^2 ss_0}{n\Delta} \right)
        \le & \frac{n\Delta}{ \sigma^2 ss_0}\cdot \mathbf E \left( \sum_{(i,j) \in S_{G^*} } \Xi_{ij}^2 \mathrm I\Big\{ |\Xi_{ij}|\ge \tilde \lambda_1 \Big\} \right)\\
        \le &3 \Delta \cdot \frac{\tilde \lambda_1^2 n }{ \sigma^2 } \cdot \frac{d}{s_0} \cdot \exp\left( -\frac{n\tilde \lambda_1^2}{2\sigma^2} \right)\\
        \le & \frac34 \left( \frac{n\tilde \lambda_1^2}{\sigma^2} \right)^2 \exp\left( -\frac{n\tilde \lambda_1^2}{4\sigma^2} \right)\\
        = &  o(1),~\text{ as } \Delta \to \infty,
        \end{aligned}
    \end{equation}
    where the last inequality uses $\tilde \lambda_1^2 \ge \frac{4\sigma^2\Delta}{n}$ and $\frac d{s_0} < \exp(\Delta) \le \exp\left( \frac{n\tilde \lambda_1^2}{4\sigma^2} \right)$.
 \end{proof}

Now we turn to analyze the term of the estimation error on $S^*$. Under proper beta-min conditions, we can bound $\sum_{(i,j) \in S^*} \left( \tilde H_{ij}^{t+1}\right)^2 \mathrm I\left( (i,j) \notin \tilde S^{t+1}\right)$ by the following lemma.
\begin{lemma}\label{supporterror}
    Assume all the conditions in Theorem \ref{T6} hold. Then, for any $\epsilon >0$, we have 
\begin{equation} 
    P\left( \sqrt{\sum_{(i,j) \in S^*} \left( \tilde H_{ij}^{t+1}\right)^2 \mathrm I\left( (i,j) \notin \tilde S^{t+1}\right)} 
    < \frac4\epsilon \delta \left\| \tilde\beta^t - \beta^* \right\|_2 + 2\sqrt{\frac{\sigma^2ss_0}{n \Delta}} \right) \to 1,
\end{equation}
as $\Delta \to \infty$.
\end{lemma}

\begin{proof}
Note that 
\begin{equation}\label{supportH}
    \begin{aligned}
    & \sqrt{\sum_{(i,j) \in S^*}\left( \tilde H_{ij}^{t+1}\right)^2 
         \mathrm I\left( (i,j) \notin \tilde S^{t+1}\right)}\\
    \le & \sqrt{\sum_{(i,j) \in S^*} \left( \tilde H_{ij}^{t+1}\right)^2 
         \mathrm I\left( |\tilde H_{ij}^{t+1}| < \tilde \lambda_2 \right)} \\
    &~+ \sqrt{\sum_{(i,j) \in S^*} \left( \tilde H_{ij}^{t+1}\right)^2 
         \mathrm I\left( |\tilde H_{ij}^{t+1}| \ge \tilde \lambda_2, ~
         \sum_{k \in S^* \cap S_{G_j}}\left( \tilde H_{kj}^{t+1}\right)^2  \mathrm I ( |\tilde H_{kj}^{t+1}| \ge \tilde \lambda_2) < s_0 \tilde \lambda_2^2 \right)}\\
    < & \sqrt{\sum_{(i,j) \in S^*}\tilde \lambda_2^2 \cdot
         \mathrm I\left( |\tilde H_{ij}^{t+1}| < \tilde \lambda_2 \right)}
        \\&+ \sqrt{\sum_{j \in G^*} s_0 \tilde \lambda_2^2 \cdot 
         \mathrm I\left( \sum_{k \in S^* \cap S_{G_j}} \left( \tilde H_{kj}^{t+1}\right)^2 \mathrm I (|\tilde H_{kj}^{t+1}| \ge \tilde \lambda_2)< s_0 \tilde \lambda_2^2 \right)}\\
    \le & \sqrt{\sum_{(i,j) \in S^*} \tilde \lambda_2^2 \cdot
         \mathrm I\left( |\tilde H_{ij}^{t+1}| < \tilde \lambda_2 \right)} 
         + \sqrt{\sum_{j \in G^*} s_0 \tilde \lambda_2^2 \cdot 
         \mathrm I\left( \sum_{k \in S^* \cap S_{G_j}} \left( \tilde H_{kj}^{t+1}\right)^2 < (s_j+s_0) \tilde \lambda_2^2 \right)}.\\
    \end{aligned}
\end{equation}
Next, we analyze these two terms respectively.

\textbf{First term.} 
    Recall that $\tilde H_{ij}^{t+1} = \beta_{ij}^* + \langle \Phi_{ij}^\top, \beta^* - \tilde \beta^{t} \rangle + \Xi_{ij}$ and $|\beta_{ij}^*| \ge (1+ \epsilon)\tilde \lambda_2$ holds for every support entry. Therefore, we have 
    \begin{equation}
        \begin{aligned}
             \sqrt{\sum_{(i,j) \in S^*} \tilde \lambda_2^2 \cdot
                     \mathrm I\left( |\tilde H_{ij}^{t+1}| < \tilde \lambda_2 \right)}
            \le & \sqrt{\sum_{(i,j) \in S^*} \tilde \lambda_2^2 \cdot
                     \mathrm I\Big( |\beta_{ij}| - | \langle \Phi_{ij}^\top, \beta^* - \tilde \beta^{t} \rangle + \Xi_{ij}| < \tilde \lambda_2 \Big)}\\ 
            \le & \sqrt{\sum_{(i,j) \in S^*} \tilde \lambda_2^2 \cdot
                     \mathrm I\Big( \epsilon \tilde \lambda_2 < |\langle \Phi_{ij}^\top, \beta^* - \tilde \beta^{t} \rangle| + |\Xi_{ij}| \Big)}\\ 
            \le & \sqrt{\sum_{(i,j) \in S^*} \tilde \lambda_2^2 \cdot
                     \mathrm I\Big( |\langle \Phi_{ij}^\top, \beta^* - \tilde \beta^{t} \rangle|> \frac{\epsilon}2 \tilde \lambda_2 \Big)}\\
                 &+ \sqrt{\sum_{(i,j) \in S^*} \tilde \lambda_2^2 \cdot
                     \mathrm I\Big( |\Xi_{ij}|> \frac{\epsilon}2 \tilde \lambda_2 \Big)}\\
            < & \frac2\epsilon \sqrt{\Upsilon\left(S^*, \tilde \beta^t \right) }
                 + \sqrt{\sum_{(i,j) \in S^*} \tilde \lambda_2^2 \cdot
                     \mathrm I\Big( |\Xi_{ij}|> \frac{\epsilon}2 \tilde \lambda_2 \Big)}. \\ 
        \end{aligned}
    \end{equation}
Under the fixed $S^*$ and based on Markov inequality, we have
    \begin{equation}
        \begin{aligned}
            P \left( \sum_{(i,j) \in S^*} \tilde \lambda_2^2 \cdot
                     \mathrm I\Big( |\Xi_{ij}|> \frac{\epsilon}2 \tilde \lambda_2 \Big)  
                 \ge \frac{\sigma^2ss_0}{n \Delta} \right)
            \le & \frac{n \Delta}{\sigma^2ss_0} \sum_{(i,j) \in S^*} \tilde \lambda_2^2 \cdot  P \Big( |\Xi_{ij}|> \frac{\epsilon}2 \tilde \lambda_2 \Big)\\
            \le & \frac1{16} (\frac{n\tilde \lambda_2^2}{\sigma^2} )^2 \exp ( - \frac{\epsilon^2}{8} \cdot \frac{n\tilde \lambda_2^2}{\sigma^2} )\\
            =&o(1) ,~ \text{ as }\Delta \to \infty,
        \end{aligned}
    \end{equation}
     where recall that $\tilde\lambda_2 =2\sqrt{ \frac{8\sigma^2}{n} \Delta  }$.
     Thus the first term in \eqref{supportH} is bounded by 
    \begin{equation}\label{supportH1}
        \sqrt{ \sum_{(i,j) \in S^*} \tilde \lambda_2^2 \cdot
                     \mathrm I\left( |\tilde H_{ij}^{t+1}| < \tilde \lambda_2 \right)}
        < \frac2\epsilon \delta \left\| \tilde \beta^{t} - \beta^* \right\|_2 
           + \sqrt{\frac{\sigma^2ss_0}{n \Delta}}.
    \end{equation}
    
    \textbf{Second term.}  
    Let $s_j = \| \beta^*_{{G_j}}\|_0$ for $j \in G^*$. For $\forall j \in G^*$, by element-wise beta-min condition $\min\limits_{(i,j) \in S^*}|\beta^*_{ij}| \ge (\sqrt{2} + \epsilon)\tilde \lambda_2$ and group-wise beta-min condition $\min\limits_{j \in G^*}\|\beta^*_{{G_j}}\|_2 \ge (\sqrt{2} + \epsilon) \sqrt{s_0} \tilde \lambda_2$, we conclude that
    $$
    \|\beta^*_{S_{G_j}}\|_2
    \ge (\sqrt{2} + \epsilon) \sqrt{s_0 \vee s_j} \tilde \lambda_2
    \ge \sqrt{s_j + s_0} \tilde \lambda_2 + \epsilon \sqrt{ s_j \vee s_0 } \tilde \lambda_2.
    $$
    Therefore, we have 
    \begin{equation}\label{supportgrouptype2}
       \begin{aligned}
        &\mathrm I\left(  \sum_{k \in S^* \cap S_{G_j}} \left( \tilde H_{kj}^{t+1}\right)^2 < (s_j+s_0) \tilde \lambda_2^2 \right) \\
        \le & \mathrm I\left(  \sqrt{ \sum_{k \in S^* \cap S_{G_j}}  \left(\beta_{kj}^* \right)^2 } - \sqrt{ \sum_{k \in S^* \cap S_{G_j}} \big( \langle \Phi_{ij}^\top, \beta^* - \tilde \beta^{t} \rangle + \Xi_{kj} \big)^2} 
            < \sqrt{s_j+s_0} \tilde \lambda_2 \right) \\
        \le & \mathrm I\left(  \sum_{k \in S^* \cap S_{G_j}} \big( \langle \Phi_{ij}^\top, \beta^* - \tilde \beta^{t} \rangle + \Xi_{kj} \big)^2 
            > \epsilon^2 (s_j \vee s_0) \tilde \lambda_2^2 \right) \\
        \le & \mathrm I\left( \Upsilon\left( S^* \cap S_{G_j}, \tilde \beta^t \right)  + \sum_{k \in S^* \cap S_{G_j}}  \Xi_{kj}^2 
            > \frac{\epsilon^2}2 (s_j \vee s_0) \tilde \lambda_2^2 \right)\\
        \le & \mathrm I\left( \Upsilon\left( S^* \cap S_{G_j}, \tilde \beta^t \right)  > \frac{\epsilon^2}4 (s_j \vee s_0) \tilde \lambda_2^2 \right)  + \mathrm I\left( \sum_{k \in S^* \cap S_{G_j}}  \Xi_{kj}^2 
            > \frac{\epsilon^2}4 (s_j \vee s_0) \tilde \lambda_2^2 \right),
        \end{aligned}
    \end{equation}
    which yields that 
    \begin{equation}\label{supportsecond}
        \begin{aligned}
        & \sqrt{\sum_{j \in G^*} s_0 \tilde \lambda_2^2 \cdot 
             \mathrm I\left( \sum_{k \in S^* \cap S_{G_j}} \left( \tilde H_{kj}^{t+1}\right)^2 < (s_j+s_0) \tilde \lambda_2^2 \right)}\\
        \le& \sqrt{\sum_{j \in G^*} s_0 \tilde \lambda_2^2 \cdot 
             \mathrm I\left( \Upsilon\left( S^* \cap S_{G_j}, \tilde \beta^t \right) > \frac{\epsilon^2}4 (s_j \vee s_0) \tilde \lambda_2^2 \right)}
             \\
             &+\sqrt{\sum_{j \in G^*} s_0 \tilde \lambda_2^2 \cdot 
             \mathrm I\left( \sum_{k \in S^* \cap S_{G_j}} \Xi_{kj}^2 > \frac{\epsilon^2}4 (s_j \vee s_0) \tilde \lambda_2^2 \right)}\\
        \le& \frac2{\epsilon} \sqrt{ \Upsilon\left( S^*, \tilde \beta^t \right) } + \sqrt{\sum_{j \in G^*} s_0 \tilde \lambda_2^2 \cdot 
             \mathrm I\left( \sum_{k \in S^* \cap S_{G_j}} \Xi_{kj}^2 > \frac{\epsilon^2}4 (s_j \vee s_0) \tilde \lambda_2^2 \right)}\\
        \le&  \frac2{\epsilon} \delta \left\| \tilde \beta^{t} - \beta^* \right\|_2 
             + \sqrt{\sum_{j \in G^*} s_0 \tilde \lambda_2^2 \cdot 
             \mathrm I\left( \sum_{k \in S^* \cap S_{G_j}} \Xi_{kj}^2 > \frac{\epsilon^2}4 (s_j \vee s_0) \tilde \lambda_2^2 \right)}.
        \end{aligned} 
    \end{equation}
    Now, based on Lemma \ref{lemma:iht1} and Theorem 2.1 in \citet{hsu2012tail}, for every $t>0$ and every support group $G_j$, we obtain that
    \begin{equation}
        P \left( \frac n{\sigma^2} \|\Xi_{S^* \cap S_{G_j}} \|_2^2 \ge s_j + 2(1+\delta)\sqrt{s_j t} + 2(1+ \delta)t \right) \le e^{-t}.
    \end{equation}
    Let $t = \frac{n\epsilon^2 (s_j \vee s_0)}{24 \sigma^2} \tilde \lambda_2^2$. We can show $t > s_j$, as $\Delta\to \infty$. From $\delta \le \frac14$ we obtain 
    \begin{equation} 
         s_j + 2(1+\delta)\sqrt{s_j t} + 2(1+ \delta)t 
         ~\le~ 6t 
         = \frac{n\epsilon^2 (s_j \vee s_0)}{4\sigma^2} \tilde \lambda_2^2,
    \end{equation}
    which implies that $P\left( \sum_{k \in S^* \cap S_{G_j}} \Xi_{kj}^2 > \frac{\epsilon^2}4 (s_j \vee s_0) \tilde \lambda_2^2\right) 
    \le \exp\left(-\frac{n\epsilon^2 (s_j \vee s_0)}{24 \sigma^2} \tilde \lambda_2^2 \right)
    $. Therefore, by Markov inequality, we have
    \begin{equation}\label{xivee}
        \begin{aligned}
        & P \left\{ \sum_{j \in G^*}  s_0 \tilde \lambda_2^2 \cdot 
            \mathrm I\left( \sum_{k \in S^* \cap S_{G_j}} \Xi_{kj}^2 > \frac{\epsilon^2}4 (s_j \vee s_0) \tilde \lambda_2^2 \right)
            \ge \frac{\sigma^2 ss_0}{n\Delta} \right\} \\
        \le & \frac{n\Delta}{\sigma^2 ss_0} \sum_{j \in G^*} s_0 \tilde \lambda_2^2 \cdot 
            P\left( \sum_{k \in S^* \cap S_{G_j}} \Xi_{kj}^2 > \frac{\epsilon^2}4 (s_j \vee s_0) \tilde \lambda_2^2 \right) \\
        \le & \frac{n\Delta}{\sigma^2 ss_0} \sum_{j \in G^*} s_0 \tilde \lambda_2^2 \cdot  \exp\left(-\frac{n\epsilon^2 (s_j \vee s_0)}{24 \sigma^2} \tilde \lambda_2^2 \right)\\
        \le & \frac{n\Delta}{\sigma^2 } \tilde \lambda_2^2 \cdot 
            \exp(-\frac{n\epsilon^2 s_0}{24 \sigma^2} \tilde \lambda_2^2 )\\
        = & \frac{1}{32} (\frac{n\tilde \lambda_2^2}{\sigma^2} )^2 
            \exp ( -\frac{\epsilon^2 s_0}{24} \cdot \frac{n\tilde \lambda_2^2}{\sigma^2} )\\
        = & o(1), ~\text{ as }\Delta \to \infty.
        \end{aligned}
    \end{equation}
    Combining \eqref{supportsecond} and \eqref{xivee}, we bound the second term by 
    \begin{equation}\label{supportH2}
        \sqrt{\sum_{ j \in G^*} s_0 \tilde \lambda_2^2 \cdot 
             \mathrm I\left( \sum_{k \in S^* \cap S_{G_j}} \left( \tilde H_{kj}^{t+1}\right)^2 < (s_j+s_0) \tilde \lambda_2^2 \right)}\\
        \le \frac2{\epsilon} \delta \left\| \tilde \beta^{t} - \beta^* \right\|_2 
             + \sqrt{\frac{\sigma^2 ss_0}{n\Delta} }.
    \end{equation}
Finally, based on \eqref{supportH1} and \eqref{supportH2}, we have 
\begin{equation}
    P\left( \sqrt{\sum_{(i,j) \in S^*} \left( \tilde H_{ij}^{t+1}\right)^2 \mathrm I\left( (i,j) \notin \tilde S^{t+1}\right)} 
    < \frac4\epsilon \delta \left\| \tilde\beta^t - \beta^* \right\|_2 + 2\sqrt{\frac{\sigma^2ss_0}{n \Delta}} \right) \to 1,
\end{equation}
as $\Delta \to \infty$.
    
\end{proof}

\section*{Appendix C: Example of sub-Gaussian random design}\label{radnomdesign}
Assume $\zeta_1,\zeta_2,\cdots,\zeta_n$ are independent and identically distributed $p$-dimensional isotropic, sub-Gaussian random vectors, forming a random matrix $Z\in  \mathbb{R}^{n\times p}$ , whose $i$-th row $Z_i$ is denoted by $\zeta_i$. In this paper, we consider a random design matrix $X$, which is generated as follows:
\begin{equation}\label{randomdesign}
X = Z\Sigma^{\frac{1}{2}},
\end{equation}
where $\Sigma$ is the covariance matrix.

According to the theoretical framework of \citet{zhou2009restricted} and \citet{mendelson2008uniform}, given the vector space $\mathcal{V}\in  \mathbb{R}^p$, the key point is to construct the restricted isometric properties between $Xv$ and $\Sigma^{\frac{1}{2}}v$ for $v \in \mathcal{V}$. The empirical process technique plays an important role, and we define Gaussian complexity first:
\begin{definition}[Gaussian complexity]
Given a subset $\mathcal{V}\subseteq  \mathbb{R}^p$, we define the Gaussian complexity of $\mathcal{V}$ as follows:
$$
\ell^*(\mathcal{V}) \coloneqq \mathbf{E}_g\sup_{\theta\in \mathcal{V}}\left|\sum_{i=1}^pg_i\theta_i\right|,
$$
where $\theta_i$ is each component of vector $\theta$, and $g_1, g_2, \cdots, g_p$ are independently drawn from $\mathcal{N}(0,1)$ distributions. In particular, given a non-negative definite matrix $\Sigma$, we define
$$
\tilde{\ell}^*(\mathcal{V}) \coloneqq \ell^*(\Sigma^{\frac{1}{2}}\mathcal{V}) = \mathbf{E}_g\sup_{v\in \mathcal{V}}\left|\langle \Sigma^{\frac{1}{2}}v, g\rangle\right| = \mathbf{E}_g\sup_{v\in \mathcal{V}}\left|\langle v, \Sigma^{\frac{ 1}{2}}g\rangle\right|.
$$
\end{definition}
According to the homogeneity of the norm, we only need to consider the subset of the unit ball sphere $S^{p-1}$, which is defined as:
$$
S^{p-1} \coloneqq \left\{v\in  \mathbb{R}^p: \|v\|_2 = 1\right\}.
$$
The main technique we use is the following empirical process result:
\begin{lemma}[Theorem 2.1 in \citet{mendelson2008uniform}]\label{mendelson1}
Let $1 \leq n \leq p$ and $0 < \delta < 1$. Let $\zeta \in \mathbb{R}^p$ be an isotropic sub-Gaussian random vector with parameter $\alpha$. Let $\zeta_1,\dots,\zeta_n$ be the independent copies of $\zeta$. Define $X$ as the random matrix in \eqref{randomdesign}, and let $\mathcal{V}$ satisfy $\Sigma^{\frac{1}{2}} v \in S^{p-1}$ for all $v \in \mathcal{V}$. If sample size $n$ satisfies
$
  n > c'\alpha^4 \delta^2 \tilde{\ell}^*(\mathcal{V})^2,
  $
  then with probability of at least $1 - \exp(-\bar{c}\delta^2n/\alpha^4)$, for all $v \in \mathcal{V}$, we have
$$
1 - \delta \leq \| X v \|_2 / \sqrt{n} \leq 1 + \delta,
$$
where $c',\bar{c} >0$ are some absolute constants.
\end{lemma}
Denote parameter space 
$
\mathcal{V} := \Theta^{m,d}(s, s_0)\cap \{v:\Sigma^{\frac{1}{2}}v \in S^{p-1}\}.
$
Then, given any $ v \in \mathcal{V}$, we assume that
$$
\rho_{\min} \leq \frac{\left\|\Sigma^{1/2} v\right\|_{2 }}{\|v\|_2} \leq \rho_{\max}.
$$
Next, we derive the Gaussian complexity $\tilde{\ell}^*(\mathcal{V})$ for the double sparse structure. We denote
$$
U := \Theta^{m,d}(s, s_0)\cap \{v:\|\Sigma^{\frac{1}{2}}v\|_2 \le 1\}.
$$
Recall that $m\times d = p$. 
Then, we have
\begin{equation*}
\begin{aligned}
\tilde{\ell}^*\left(\mathcal{V}\right) \le \widetilde{\ell}^*\left(U\right) & = \mathbf{E}_g \sup _{t \in U}\left|\left\langle t, \Sigma^{1 / 2} g\right\rangle\right| \\
& \leq 3 \sqrt{\log \left|\Theta^{m,d}(s, s_0)\right|} \sup _{t \in U} \sqrt{ \mathbf{E}_g \left|\left\langle t, \Sigma^{1 / 2} g\right\rangle\right|^2} \\
& \leq C \sqrt{ss_0 \log (e d / s_0)+s\log (em/s)} \sup _{t \in U}\left\|\Sigma^{1 / 2} t\right\|_2 \\
& \leq C \sqrt{ss_0 \log (e d / s_0)+s\log (em/s)},
\end{aligned}
\end{equation*}
where the first inequality follows from  Chapter 3 in \citet{ledoux1991probability}.
Note that $$\frac{\|Xv\|_2}{\sqrt{n}\|v\|_2}=\frac{\|Xv\|_2}{\sqrt{n}\|\Sigma^{\frac{1 }{2}}_Sv\|_2}\cdot\frac{\|\Sigma^{\frac{1}{2}}_Sv\|_2}{\|v\|_2}.$$
Therefore, by Lemma \ref{mendelson1},  for $
  n > C' \alpha^4 \delta^2 \cdot\left(ss_0 \log (e d / s_0)+s\log (em/s)\right)
  $, we have
$$
(1 - \delta)\rho_{\min} \leq \frac{\| X v \|_2}{\sqrt{n}\|v\|_2} \leq (1 + \delta)\rho_{\max}.
$$
This proves the satisfaction of the DSRIP condition under the sub-Gaussian random design.
\vskip 0.2in

\bibliographystyle{unsrtnat}
\bibliography{ref}

\end{document}